\documentclass[11pt,makeidx]{amsart}
\usepackage{amssymb,enumerate,mathrsfs, amsmath, amsthm}
\usepackage{url,titletoc,xspace}

\makeindex

\usepackage[usenames,dvipsnames]{xcolor}
\definecolor{darkblue}{rgb}{0.0, 0.0, 0.55}
\usepackage[pagebackref,colorlinks,linkcolor=BrickRed,citecolor=OliveGreen,urlcolor=darkblue,hypertexnames=true]{hyperref}

\usepackage{upgreek}
\usepackage{cmap}

\usepackage{mdwlist}

     \DeclareRobustCommand{\SkipTocEntry}[5]{}

\renewcommand{\qedsymbol}{\rule[.12ex]{1.2ex}{1.2ex}}
\renewcommand{\subset}{\subseteq}
\renewcommand{\emptyset}{\varnothing}

\newcommand{\ax}{\langle x\rangle}
\newcommand{\axs}{\langle x,x^* \rangle}

\DeclareMathOperator{\Span}{span}

\makeatletter
\newcommand{\mycontentsbox}{%
{
\addtolength{\parskip}{-2.0pt}\footnotesize
\tableofcontents}}
\def\enddoc@text{\ifx\@empty\@translators \else\@settranslators\fi
\ifx\@empty\addresses \else\@setaddresses\fi
\newpage\mycontentsbox
}
\makeatother

\usepackage[margin=2.5cm]{geometry}

\newtheorem{theorem}{Theorem}[section]
\newtheorem{cor}[theorem]{Corollary}
\newtheorem{lemma}[theorem]{Lemma}

\newtheorem{prop}[theorem]{Proposition}
\newtheorem{definition}[theorem]{Definition}

\newtheorem{XxmpX}[theorem]{Example} 
\newenvironment{example}    
  {\renewcommand{\qedsymbol}{$\lozenge$}%
   \pushQED{\qed}\begin{XxmpX}}
  {\popQED\end{XxmpX}}

\newtheorem{XxmpY}[theorem]{Remark} 
\newenvironment{remark}    
  {\renewcommand{\qedsymbol}{$\lozenge$}%
   \pushQED{\qed}\begin{XxmpY}}
  {\popQED\end{XxmpY}}

\newtheorem{conj}[theorem]{Conjecture}

\newtheorem{thm}[theorem]{Theorem}
\newtheorem{lem}[theorem]{Lemma}

\def\tg{\tilde{g}}

\def\tA{{\tilde{A}}}

\def\cD{\mathcal D}

\def\cJ{\mathcal J}

\def\ptA{\mathbf{A}}
\def\ptAp{\mathbf{A}^\prime}
\def\tptA{\tilde{\ptA}}
\def\tcE{\tilde{\mathcal{E}}}

\def\ben{\begin{enumerate}}
\def\een{\end{enumerate}}

\def\bs{\bigskip}

\def\bem{\begin{pmatrix}}
\def\eem{\end{pmatrix}}

\def\beq{\begin{equation}}
\def\eeq{\end{equation}}

\renewcommand{\qedsymbol}{\rule[.12ex]{1.2ex}{1.2ex}}

\def\sssec{\subsubsection}

\def\tg{{\tilde{g}}}
\def\tq{\tilde{q}}

\def\tB{{\tilde{B}}}
\def\tp{\tilde{p}}

\def\cN{\mathcal N}
\def\cR{\mathcal R}
\def\cZ{\mathcal Z}

\def\Wem{{ W_\emptyset }}
\def\Wema{W^{\ast}_\emptyset}

\def\fU{{\sW}}

\def\matn{\mathbb \rC^{n\times n}}
\def\matng{(\mathbb \rC^{n\times n})^g}
\def\matnjg{(\mathbb \rC^{n_j\times n_j})^g}
\def\matdg{M_d(\C)^g}
\def\matdtg{M_d(\C)^\tg}

\def\matetg{M_e(\C)^\tg}
\def\mateg{M_e(\C)^g}

\def\cH{\mathscr H}
\def\calH{\mathcal H}
\def\cE{\mathcal E}
\def\cN{\mathcal N}

\def\R{ {\mathbb{R}} }
\def\C{ {\mathbb{C}} }
\def\rC{{\mathbb C}}

\def\N{ {\mathbb{N}} }

\def\tW{\tilde{W}}

\def\cG{\mathcal G}

\DeclareMathOperator{\rg}{rg}

\def\IHA{I_H\otimes A}

\def\xij{{\Xi_j}}

\def\hatg{{\hat{g}}}

\def\cH{\mathcal{H}}
\def\sH{\mathscr{H}}

\def\sW{\mathscr W}
\def\tsW{\tilde{\mathscr{W}}}
\def\cEp{\mathcal{E}^\prime}
\def\SR{\Delta}
\def\cV{\mathcal V}

\newcommand{\df}[1]{{\bf{#1}}{\index{#1}}}

\newcommand{\La}{\Lambda}

 \def\bCxshort{\mathbb C\langle x\rangle}
  \def\bCx{\mathbb C\langle x_1,\dots,x_g\rangle}

\def\Mnns{M_n(\C)}

\newcommand{\ct}{convexotonic\xspace}
\newcommand{\Ct}{Convexotonic\xspace}

\title{Bianalytic Maps Between Free Spectrahedra}

\author[M.L. Augat]{Meric Augat}
\address{Meric Augat, Department of Mathematics\\
  University of Florida\\ Gainesville 
   }
   \email{mlaugat@math.ufl.edu}

\author[J.W. Helton]{J. William Helton${}^1$}
\address{J. William Helton, Department of Mathematics\\
  University of California \\
  San Diego}
\email{helton@math.ucsd.edu}
\thanks{${}^1$Research supported by the NSF grant
DMS 1201498, and the Ford Motor Co.}

\author[I. Klep]{Igor Klep${}^{2}$}
\address{Igor Klep, Department of Mathematics, 
The University of Auckland, New Zealand}
\email{igor.klep@auckland.ac.nz}
\thanks{${}^2$Supported by the Marsden Fund Council of the Royal Society of New Zealand. Partially supported by the Slovenian Research Agency grants P1-0222, L1-6722, N1-0057 and J1-8132.}

\author[S. McCullough]{Scott McCullough${}^3$}
\address{Scott McCullough, Department of Mathematics\\
  University of Florida\\ Gainesville 
   }
   \email{sam@math.ufl.edu}
\thanks{${}^3$Research supported by the NSF grant DMS-1361501} 

\date{\today}



\newcounter{counterPQ}

\begin{document}

\numberwithin{equation}{section}

\begin{abstract} 
Linear matrix inequalities (LMIs)
$I_d + \sum_{j=1}^g A_jx_j + \sum_{j=1}^g 
A_j^*x_j^*
\succeq0$
play a role in many areas of applications.
The set of solutions of an LMI is a spectrahedron.
LMIs in (dimension--free)
matrix variables model most problems in linear systems engineering, and their solution sets are called free spectrahedra.
Free spectrahedra are exactly the free semialgebraic convex sets.

  This paper studies free  analytic maps between free spectrahedra and, under certain (generically valid) irreducibility assumptions,  classifies all those that are bianalytic.
  The foundation of such maps turns out to be a very small  class of birational maps we call \ct.  The \ct maps in $g$ variables sit in   
   correspondence
   with $g$-dimensional algebras.
   If two bounded free spectrahedra $\cD_A$ and $\cD_B$ meeting our irreducibility assumptions
    are free bianalytic with map denoted $p$,
  then $p$ must (after possibly an affine linear transform) extend to a  \ct map 
corresponding to a $g$-dimensional algebra spanned by  $(U-I)A_1,\ldots,(U-I)A_g$ for some unitary $U$.
Furthermore, 
 $B$ and $UA$ are unitarily equivalent.\looseness=-1

The article also establishes  a Positivstellensatz for free analytic functions
whose real part is positive semidefinite on a free spectrahedron
and proves a representation for a free analytic map 
from $\cD_A$ to $\cD_B$ (not necessarily bianalytic).
Another result shows that a  function  analytic on any radial expansion
of a free spectrahedron  is approximable by 
polynomials uniformly  on the spectrahedron.
These theorems are needed for classifying free  bianalytic maps. 
\end{abstract}

\maketitle

\setcounter{tocdepth}{3}
\contentsmargin{2.55em} 
\dottedcontents{section}[3.8em]{}{2.3em}{.4pc} 
\dottedcontents{subsection}[6.1em]{}{3.2em}{.4pc}
\dottedcontents{subsubsection}[8.4em]{}{4.1em}{.4pc}

\section{Introduction}
Given a tuple $A=(A_1,\dots,A_g)$  of complex $d\times d$  matrices and indeterminates $x=(x_1,\dots,x_g)$, the expression
\[
L_A(x) = I_d+ \sum_{j=1}^g A_j x_j + \sum_{j=1}^g A_j^* x_j^*
\]
is a monic linear pencil.  The set 
\[
 \cD_A(1) = \{z\in\mathbb C^g: L_A(z) \mbox{ is positive semidefinite}\}
\]
is known as a \df{spectrahedron} (synonymously \df{LMI domain}). Spectrahedra play a central role in semidefinite programming, convex optimization and in real algebraic geometry \cite{BPR13}. They also figure prominently in the study of determinantal representations \cite{Bra11,GK-VVW,NT12,Vin93}, the solution of the Lax conjecture \cite{HV07},
in the solution of the Kadison-Singer paving conjecture \cite{MSS15},
 and in systems engineering \cite{BGFB94, Skelton}.  The monic linear pencil $L_A$ is naturally evaluated at a tuple $X=(X_1,\dots,X_g)$ of $n\times n$ matrices
 using the Kronecker product as
\[
 L_A(X)= I_d\otimes I_n + \sum_{j=1}^g A_j \otimes X_j + \sum_{j=1}^g A_j^* \otimes X_j^*
\]
 with output a $dn\times dn$ self-adjoint matrix. Let 
 $M_n(\C)$ denote the $n\times n$ matrices with entries from $\C$ and \index{$M_n(\C)$} \index{$M_n(\C)^g$} \index{$(\cD_A(n))_n$}
$M_n(\mathbb C)^g$ denote the set of $g$-tuples of $n\times n$ matrices. We call the sequence $(\cD_A(n))_n$, where
\[
 \cD_A(n)=\{X\in M_n(\mathbb C)^g: L_A(X) \mbox{ is positive semidefinite}\}
\]
 a \df{free spectrahedron} (or a \df{free LMI domain}). Free spectrahedra arise naturally in many systems engineering problems described by a signal flow diagram \cite{dOHMP09}.  They are also canonical examples of {matrix convex sets} \cite{EW,HKMjems} and thus are intimately connected to the theory of completely positive maps and operator systems and spaces \cite{Pau}.

In this article we study bianalytic maps $p$ between free spectrahedra. Our belief, supported by the results in this paper and our experience with free spectrahedra (see for instance \cite{HM12}, \cite{HKMS} and \cite{HKM12b}), is that the existence of bianalytic maps imposes rigid, but elegant, structure on both the free spectrahedra as well as the map $p$. Motivation for this study comes from several sources. Free analysis, including free analytic functions, is a recent development \cite{KVV14,Tay72,Pas,AM15,BGM,Po1,Po2,KS,HKM12b,BKP16} with 
close
ties to free probability \cite{Voi04, Voi10} and quantum information theory \cite{NC,HKMjems}. 
In engineering systems theory  certain  model problems can be described by a system of {matrix inequalities}.  For optimization and design purposes, it is hoped that these inequalities have a convex solution set. In this case, under a boundedness hypothesis, the solution set is a free spectrahedron \cite{HM12}.  If the domain is not convex one might replace it by its matrix convex hull \cite{HKM16} or map it bianalytically to a free spectrahedron.  Two such maps then lead to a 
bianalytic map
between free spectrahedra. 

Studying bianalytic maps between free spectrahedra is a free analog of rigidity problems in  several complex variables \cite{DAn,For89,For,HJ01,HJY14,Krantz}. Indeed, there is a large literature  on bianalytic maps on convex sets. 
For instance, Faran \cite{Far86}
 showed that any proper analytic map from
the unit ball  in $\C^n$ to the unit ball in $\C^N$
 with $N\leq2n-2$
 that is real analytic up to the boundary, is 
(up to automorphisms of the domain and codomain) the standard 
linear embedding $z\mapsto (z,0)$.
When $N=2n-1$, Huang and Ji \cite{HJ01} proved
this map and the Whitney map $z=(z',z_n)\mapsto (z',z_n z)$
are the only such maps. 
Forstneri\v c
\cite{For}
showed that any proper analytic map between balls 
with sufficient regularity at the boundary
must be rational.
We refer to \cite{HJY14} for further recent developments.

The remainder of this introduction is organized as follows. Basic terminology and background appear  in Subsection \ref{sec:basic}. A novel family of maps we call \ct maps and that we believe comprise, up to affine linear equivalence, exactly the bianalytic maps between free spectrahedra, is described in Subsection \ref{sec:ct}.  Subsection \ref{sec:ct} also contains the main result of the article, Theorem \ref{thm:main} on bianalytic mappings between free spectrahedra. Subsections  \ref{sec:approximate} and \ref{sec:introsatz}  describe Positivstellens\"atze and results related to recent free Oka-Weil theorems \cite{AM14,BMV} on (uniform) polynomial approximation of free spectrahedra and  functions analytic in a suitable neighborhood of a spectrahedron. Both are ingredients in the proof of Theorem \ref{thm:main}.

\subsection{Basic definitions}
\label{sec:basic}
 Notations, definitions and background needed, but not already introduced, to describe the results in this paper are collected in this section.  

\subsubsection{Free polynomials}
Let $x=(x_1,\ldots,x_g)$ denote $g$ freely noncommuting letters and 
$\langle x\rangle$ the set of \df{words} in $x$,  including 
  the empty word denoted by either $1$ or $\emptyset$.
  The \df{length of a word} $w\in\langle x\rangle$ is  denoted by $|w|$. 
  Let $\bCxshort=\bCx$ denote the $\C$-algebra freely generated by $x$. \index{$\bCxshort$} %
Its elements are
linear combinations of words in $x$ and are called
{\bf analytic free  polynomials}. \index{analytic free polynomials}
We shall also consider the \df{free polynomials} $\C\axs$  in both the variables $x=(x_1,\dots,x_g)$
  and their formal adjoints, $x^*=(x_1^*,\dots,x_g^*)$. 
  For instance, $x_1x_2 + x_2 x_1 + 5x_1^3$ is analytic,
  but $x_1^*x_2 + 3x_2 x_1^5 $ is not. 
A polynomial is \df{hereditary} 
provided all the $x^*$ variables, if any, always appear on the left of all $x$ variables. Thus an hereditary polynomial is a finite linear combination of terms $v^* w$ where $v$ and $w$ are words in $x$. A special case are polynomials of the form analytic plus anti-analytic; that is $f+g^*$ for  some
$f,g\in\C\ax$. These definitions naturally extend to matrices over polynomials.

Given a word  $w=x_{i_1} x_{i_2}\cdots x_{i_m}$ and a tuple $X=(X_1,\dots,X_g)\in \Mnns^g$, let $w(X)=X^w = X_{i_1} X_{i_2}\cdots X_{i_m}$. 
A matrix-valued free polynomial $p = \sum p_w w$ is evaluated at $X$ using the Kronecker product as
\[
 p(X) = \sum p_w\otimes w(X).
\]

\subsubsection{Free domains, matrix convex sets and spectrahedra}
\label{sec:free domains}
Let $M(\mathbb C)^g$ denote the sequence $(M_n(\mathbb C)^g)_n$.  A \df{subset} $\Gamma$ of $M(\mathbb C)^g$ is a sequence $(\Gamma(n))_n$ where $\Gamma(n) \subset M_n(\mathbb C)^g$.  The subset $\Gamma$ is a \df{free set} if it is closed under direct sums and unitary similarity; that is,  if $X\in \Gamma(n)$ and $Y\in \Gamma(m),$ then 
\[
 X\oplus Y = \begin{pmatrix} \begin{pmatrix} X_1& 0\\0 & Y_1 \end{pmatrix}, \dots, \begin{pmatrix} X_g& 0\\0 & Y_g \end{pmatrix} \end{pmatrix} \in \Gamma(n+m)
\]
 and if $U$ is an $n\times n$ unitary matrix, then
\[
 U^* X  U = \begin{pmatrix} U^* X_1 U, \dots, U^* X_g U \end{pmatrix} \in \Gamma(n).
\]

  The free set $\Gamma$ is a \df{matrix convex set} (alternately \df{free convex set}) if it is also closed under simultaneous conjugation by isometries; i.e., if $X\in \Gamma(n)$ and $V$ is an $n\times m$ isometry, then  $V^* XV \in \Gamma(m)$.  In the case that $0\in \Gamma(1)$, $\Gamma$ is a matrix convex set if and only if it is closed under direct sums and simultaneous conjugation by contractions.  It is straightforward to see that a matrix convex set is \df{levelwise} convex; i.e., each $\Gamma(n)$ is a convex set in $M_n(\mathbb C)^g$. The converse is true if $\Gamma$, in addition to being a free set, is closed with respect to restrictions
to reducing subspaces. 

  A distinguished class of matrix convex domains
  are those described by a linear matrix inequality. 
  Given a positive integer $d$
  and  $A_1,\dots,A_g \in M_d(\C),$ the
  linear matrix-valued free polynomial
\[ 
  \La_A(x)=\sum_{j=1}^g A_j x_j\in M_d(\C)\otimes \bCx
\]
  is a {\bf (homogeneous) linear pencil}.\index{homogeneous linear pencil} \index{linear pencil}
  Its adjoint is, by definition,
$
  \La_A(x)^*=\sum_{j=1}^g A_j^* x_j^*.
$
 Thus
\[ 
  L_A(x) = I_d + \La_A(x) +\La_A(x)^*.
\] 
In particular, $\cD_A=\cD_{\Lambda_A}$ and 
 it is immediate that the free spectrahedron $\cD_A$ is a matrix convex set that contains a neighborhood of $0$.

\subsubsection{Free functions}
Let $\cD\subset M(\C)^g$.
 A \df{free  function}   $f$ from $\cD$ into $M(\C)^1$ is a sequence
 of functions $f[n]:\cD(n) \to M_n(\C)$ that
  \df{respects intertwining}; that is
 if $X\in\cD(n)$, $Y\in\cD(m)$, $\Gamma:\mathbb C^m\to\mathbb C^n$,
  and
 \[
  X\Gamma=(X_1\Gamma,\dots, X_g\Gamma)
   =(\Gamma Y_1,\dots, \Gamma Y_g)=\Gamma Y,
 \]
  then $f[n](X) \Gamma =  \Gamma f[m] (Y)$.
  Equivalently, $f$ respects direct sums and similarity.
The definition of a free  function naturally extends to
vector-valued functions $f:\cD\to M(\C)^h$,
matrix-valued functions $f:\cD\to M_e(\C)$ and even
 operator-valued functions. We refer the reader to
 \cite{KVV14,Voi10} for a comprehensive study of free function theory.

\subsubsection{Formal power series and free analytic functions} 
\label{sec:formal}
Here, assuming, as we always will,  its domain $\Gamma\subset M(\C)^g$ is a free  open set (meaning each $\Gamma(n)\subset M_n(\C)^g$ is open), 
a free function $f=(f[n])_{n}:\Gamma\to M(\C)$ is {\bf free analytic} if each $f[n]$ is analytic.
Very weak additional hypotheses  (e.g.~continuity \cite{HKM11b} or even
local boundedness \cite{KVV14,AM14}) on a free function imply it is analytic.

An important fact for us is that 
a formal power series with positive radius of convergence determines a free analytic function
within its  radius of convergence and (under a mild local boundedness assumptions) vice versa,
cf.~\cite[Chapter 7]{KVV14} or \cite[Proposition 2.24]{HKM12b}.  Given a positive integer $d$ and Hilbert space $H$, an
  operator-valued \df{formal  power series} $f$ 
  in $x$  is an expression
  of the form
 \begin{equation*}
   f = \sum_{m=0}^\infty \sum_{\substack{\; w\in\langle x\rangle\\ |w|=m}} f_w w= \sum_{m=0}^\infty f^{(m)},
 \end{equation*}
  where $f_w:\C^d\to H$ are linear maps  and $f^{(m)}$ is the \df{homogeneous component} of degree $m$ of $f;$ that is, the sum of all monomials in $f$ of degree $m$.
 Given $X\in M_n(\C)^g$,  define
 \[
    f(X)=\sum_{m=0}^\infty \sum_{\substack{\; w\in\langle x\rangle\\ |w|=m}}  f_w \otimes w(X),
 \]
  provided the series converges (summed in the indicated order). 
 If the norms of the coefficients of $f$ grow slowly enough, then, for $\|X_j\|$ sufficiently small,
the series $f(X)$ will converge.  For the purposes of this article, the \df{formal radius of convergence} $\tau(f)$ of a   formal power series $ f(x) =\sum f_\alpha \alpha$  with operator coefficients is 
\[
  \tau(f) = \frac{1}{ \limsup_N \big(\sum_{|\alpha|=N} \|f_\alpha\|\big)^{\frac 1N}},
\]
 with the obvious interpretations in the cases that the limit superior is either zero or infinity. 
 Similarly, the \df{spectral radius} of a tuple $X\in M_n(\C)^g$ is 
\[
 \rho(X) =\limsup_N \max \{\|X^\alpha\|^{\frac{1}{N}}: |\alpha|=N\}.
\]

A tuple of matrices $E\in\matng$ is (jointly) \df{nilpotent} if there exists an $N$ such 
  that $E^w=0$ for all words $w$ of length $|w|\ge N$.
  The smallest such $N$ is the \df{order of nilpotence} of $E$. 
In particular, if $X$ is (jointly) nilpotent, then $\rho(X)=0$. In any event,
if $X\in M(\rC)^g$ and $\rho(X)<\tau(f)$, then the series
\[
 f(X) = \sum_{N=0}^\infty \sum_{|\alpha|=N} f_\alpha \otimes X^\alpha
\]
 converges. 
Let $\SR_\tau =\{X\in M(\rC)^g: \rho(X)<\tau\}.$  \index{$\SR_\tau$} \index{$\tau(f)$} \index{$\rho(X)$}

\subsubsection{Free Rational Functions}
\label{sec:freerats}
{Free rational functions regular at $0$} (in the free variables
$x=(x_1,\dots,x_g)$)  appear in many areas
of mathematics and its applications including automata theory and
systems engineering. There are several different, but equivalent
definitions. Based on the results of \cite[Theorem 3.1]{KVV09}
and \cite[Theorem 3.5]{Vol17}) a \df{free rational function  regular at $0$} can, for the purposes of this article,
 be defined with minimal overhead as an expression of the form
\begin{equation}
\label{eq:ratr}
r(x)= c^* \big(I-\La_E(x)\big)^{-1} b
\end{equation}
where $e$ is a positive integer,  $E\in M_e(\C)^g$ and  $b,c\in\C^e$ are vectors. The expression $r$ is evaluated in the obvious fashion for
a tuple $X\in M_n(\C)^g$ so long as $I-\La_E(X)$ is invertible.
In particular, this natural domain of $r$ contains a free neighborhood
of $0$. Often in the sequel by {\it rational function} it will be clear from the context
that we mean {\it free rational function regular at $0$}. 
An exercise shows that free polynomials are (free) rational functions. Moreover,
it is true that the sum and product of rational functions are again rational.
Likewise a free  
rational function $r$ as in equation \eqref{eq:ratr}   is free analytic. It is a fundamental result that
the singularity set of $r$ coincides with the
singularity set (i.e., the free locus \cite{KV}) $\cZ_E$  of $I-\La_E$ 
(see \cite[Theorem 3.1]{KVV09}
and \cite[Theorem 3.5]{Vol17})
if $E$ is of minimal size among all representations of the form \eqref{eq:ratr} for $r$. That is, $r$ can not be extended analytically to a (open) set strictly containing the free locus $\cZ_E$.

\def\cE{\tilde \cD}
\subsection{Bianalytic maps between free spectrahedra}
\label{sec:ct}
A \df{free analytic mapping} (or simply an analytic mapping) $p$ is, for some pair of positive integers $g,\tg$, an expression of the form
\[
 p=(p^1,\dots,p^\tg),
\]
 where each $p^j$ is an analytic function in the free variables $x=(x_1,\dots,x_g)$. Given free domains $\cD$ and $\cE$, we write  $p:\cD\to \cE$ to indicate $\cD$ is a subset of the domain of $p$ and $p$ maps $\cD$ into $\cE$.  The domains $\cD$ and $\cE$ are \df{bianalytic} if there exist free analytic mappings $p:\cD\to\cE$ and $q:\cE\to \cD$ such that $p\circ q$ and $q\circ p$ are the identity mappings on $\cE$ and $\cD$ respectively.   To emphasize the role of $p$ (and $q$), we say that $\cD$ and $\cE$ are 
\df{$p$-bianalytic}. 

In this paper we introduce a small and highly structured class of birational maps we call \df{convexotonic} and to each such map $p$ describe the pairs of spectrahedra $(\cD,\cE)$ bianalytic via $p$. We conjecture these triples $(p,\cD,\cE)$ account for all bianalytic free spectrahedra 
(up to affine linear equivalence)
and establish the result under certain irreducibility hypotheses on $\cD$ and $\cE$. We start with the
definition of the \ct maps.

\def\cE{\mathcal E}
\subsubsection{\Ct maps}
\label{sssec:contonics}
A tuple $\Xi=(\Xi_1,\dots,\Xi_g)\in M_g(\C)^g$  satisfying
\beq\label{eq:cttuple}
 \Xi_k \Xi_j = \sum_{s=1}^g (\Xi_j)_{k,s} \Xi_s
\eeq
for each $1\le j,k\le g$ is \df{convexotonic}. 
    We say the rational mappings $p$ and $q$  whose entries have the form
\[
 p^i (x)=\sum_j x_j \left(I-\La_{\Xi}(x) \right )^{-1}_{j, i} 
    \qquad \text{and} \qquad  q^i (x)=\sum_j x_j \left(I +\La_{\Xi}(x) \right)^{-1}_{j,i},
\]
that is, in row form,
\begin{equation}
\label{eq:tropic}
 p(x)= x(I-\Lambda_\Xi(x))^{-1}
\qquad \text{and} \qquad q=x(I+\Lambda_\Xi(x))^{-1}
\end{equation}
are \df{\ct}.  
 It turns out (see Proposition \ref{prop:con}) the mappings $p$ and $q$ are inverses of one another, hence they are birational maps.

Given a $g$-tuple $R=(R_1,\dots,R_g)$ of $n\times n$ matrices that spans a $g$-dimensional algebra $\cR$,  we call the $g$-tuple of
$g\times g$ matrices $\Xi=(\Xi_1,\dots,\Xi_g)$ uniquely determined by 
\begin{equation*}
 R_k R_j = \sum_{s=1}^g (\Xi_j)_{k,s}R_s,
\end{equation*}
the \df{structure matrices} for $\cR$ (suppressing the obvious dependence on the choice of basis $R$). By Proposition \ref{lem:gtg}, $\Xi$ is \ct.
Moreover, if $\Xi$ is \ct, then $\mathcal X$ equal the span of $\Xi$ is an algebra of dimension
at most $g$ for which $\Xi_j$ are the structure matrices. 
See Proposition \ref{prop:con}. 

Conversely, each \ct $g$-tuple $\Xi$ as in \eqref{eq:cttuple}
(even if linearly dependent)
arises as the set of structure matrices for a $g$-dimensional
algebra. For instance, letting $R_j$ denote the  $(g+1)\times (g+1)$  matrices
\[
 R_j = \begin{pmatrix} 0 & e_j^*\\ 0 &\Xi_j \end{pmatrix}
\]
 with respect to the orthogonal decomposition $\C \oplus \C^g$ (here $e_j$ is the $j$-th standard
 basis vector for $\C^g$) as an easy computation reveals.

\Ct maps are  fundamental objects and to each are attached 
 pairs of bianalytic spectrahedra. Let
$\cR={\rm span}\{R_1,\ldots,R_g\}\subseteq M_d(\C)$ be a $g$-dimensional algebra with structure matrices $\Xi$,
 and suppose that 
  $C$ a $d\times d$ is a  unitary matrix and  a tuple  $A\in M_d(\mathbb C)^g$,  such that
 $R_j=(C-I)A_j$ for $1\le j\le g$,
and 
 \begin{equation}
 \label{eq:Astructure}
 A_k R_j = \sum_{s=1}^g (\Xi_j)_{k,s} A_s.
\end{equation}
In particular, the span
   $\mathcal A$ of the $A_j$ is a right $\cR$-module and 
if $C-I$ is invertible then
   \eqref{eq:Astructure} holds automatically.
We call the so constructed $(\cD_A,\cD_{CA})$ a \df{spectrahedral pair}
associated to the algebra $\cR$.

 \sssec{Overview of free bianalytic maps between free spectrahedra}

\begin{thm}
\label{thm:ctok}
If  $(\cD_A,\cD_{CA})$ is a spectrahedral pair
associated to a $g$-dimensional algebra $\cR$ and $C$ is unitary,
then  $\cD_A$ is bianalytic  to $\cD_{CA}$
under the \ct map $p$  whose structure matrices $\Xi$ are associated 
to the algebra $\cR$. 
\end{thm}

See \cite{HKMVarxiv} for an expanded statement and proof of Theorem \ref{thm:ctok}.

\begin{proof}
A proof appears immediately after Theorem \ref{thm:shotinthedark}.
\end{proof}

We conjecture that \ct maps 
are the only bianalytic maps between free spectrahedra.

\begin{conj}
\label{conj:main}
Up to conjugation with  affine linear maps, 
the only 
  bounded free spectrahedra $\cD_A $, $ \cD_{B}$
  that are $p$-bianalytic  
  arise as spectrahedral pairs associated to an algebra $\cR$ with  
  $p$ as the corresponding \ct map.

A weaker version of the conjecture adds the hypothesis that
  the the ranges of the $A_j$ and $B_k$ span their respective spaces.
\end{conj}

Theorem \ref{thm:main} below says   
the conjecture is true in a generic sense.
An unusual feature of Conjecture \ref{conj:main} from the
viewpoint of traditional several complex variables is 
that typical bianalytic mapping results 
would be stated up to conjugation with automorphisms
of $\cD_A$ and $\cD_B$.  Here we
are actually asserting conjugation up to affine linear
equivalence. See also Subsection \ref{ssec:ball}.

\bs

We emphasize there are few indecomposable $g$-dimensional complex algebras.  
To give a clear picture we 
have calculated the \ct maps for these algebras explicitly for $g=2$ and $g=3.$
These calculations were done in Mathematica using NCAlgebra in a notebook
 you can use after downloading from
{\url{https://github.com/NCAlgebra/UserNCNotebooks}} \cite{HOMS16}.

\begin{prop}
\label{prop:g2g3}
We list a basis $R_1, R_2$ for each of the four  2-dimensional
indecomposable  algebras over $\C$.
Then we give the associated  ``indecomposable'' \ct map and
its (\ct) inverse.

\ben[\rm(1)]
\item $R_1$ is nilpotent of order 3 and $R_2=R_1^2$
$$
 p(x_1,x_2)=\bem x_1 & x_2+x_1^2\eem \qquad 
\quad
 q(x_1,x_2)= \bem x_1 & x_2 - x_1^2\eem.
$$
\item $R_1^2 =  R_1, \ R_1 R_2=R_2$
\[
\begin{split}
p(x)&= 
\bem
(1-x_1)^{-1} x_1 &
(1 - x_1)^{-1} x_2
\eem
\qquad
q(x) =  \bem (1+x_1)^{-1}x_1  & (1+x_1)^{-1} x_2 \eem.
\end{split}
\]

\item
$R_1^2=R_1, \  R_2 R_1=R_2$
\[
\begin{split}
p(x) & = \bem x_1 (1-x_1)^{-1} & x_2(1-x_1)^{-1} \eem
\qquad
q(x)= \bem x_1(1+x_1)^{-1}& x_2(1+x_1)^{-1} \eem.
\end{split}
\]

\item
$R_1^2=R_1, \ R_1R_2=R_2, \ R_2 R_1=R_2$
\[
\begin{split}
p(x) & = \bem
x_1 (1 - x_1)^{-1} &
(1 - x_1)^{-1}  x_2 (1 - x_1)^{-1} 
\eem
\\
q(x)&= \bem
x_1 (1 + x_1)^{-1} &
(1 + x_1)^{-1}  x_2 (1 + x_1)^{-1} 
\eem.
\end{split}
\]
\een
For $g=3$ there are exactly ten plus a one parameter 
family of indecomposable \ct  maps,
since there are exactly this many corresponding indecomposable 
 $3$-dimensional algebras, see Appendix A to
the arXiv posting 
{\rm\url{https://arxiv.org/abs/1604.04952}}
of this paper.
\end{prop}

\begin{proof}
See Section \ref{sec:examples}.
\end{proof}

\begin{remark}\rm
All $g$ variable \ct maps are direct sums of \ct maps  associated to indecomposable algebras.
See Subsection \ref{subsec:decompose}. 

The composition of two \ct maps
may not be \ct (see Subsection \ref{ssec:compose}), a further indication
of the very restrictive nature of \ct maps.
\end{remark}

 \sssec{Results on free bianalytic maps under a genericity assumption}

The main result of this paper supporting Conjecture \ref{conj:main} is
Theorem \ref{thm:main} below. It says, in part, 
under certain irreducibility conditions on $A$ and $B$, if $\cD_A$ and
$\cD_B$ are $p$-bianalytic, then $p$ and its inverse $q$ are in fact
\ct.

Let $d$ be a positive integer. 
A set $\{u^1,\dots,u^{d+1}\}$ is a \df{hyperbasis} for $\C^d$ if each
$d$ element subset is a basis.  The tuple $A\in
M_d(\C)^g$ is \df{sv-generic} if there exists
$\alpha^1,\dots,\alpha^{d+1}$ and $\beta^1,\dots,\beta^d$ in $\C^g$
such that $I-\Lambda_A(\alpha^j)^*\Lambda_A(\alpha^j)$ is positive
semidefinite and has a one-dimensional kernel spanned by $u^j$ and the
set $\{u^1,\dots,u^{d+1}\}$ is a hyperbasis for $\C^d$; and
$I-\Lambda_A(\beta^k)\Lambda_A(\beta^k)^*$ is positive semidefinite,
its kernel spanned by $v^k$ and the set $\{v^1,\dots,v^d\}$ is a basis
for $\C^d$.  
Generic  tuples  $A$ satisfy this property, see Remark \ref{rem:sv=gen}.  Weaker (but still sufficient) versions of the sv-generic
condition are given in the body of the paper, see Subsection
\ref{sssec:eig}.

\begin{theorem}
 \label{thm:main}
 Suppose $A\in M_d(\C)^g$ and $B\in M_e(\C)^g$ are sv-generic and $\cD_A$ is bounded. If $p$ is a birational map between 
 $\cD_A$ and $\cD_B$ with $p(0)=0$ and $p^\prime(0)=I_g$, then
\ben[\rm (1)]
\item $d=e$; 
\item  there exists a $d\times d$ matrix $C$ such that $B$ is unitarily
 equivalent to $CA$;
\item  the tuple $R=(C-I)A$ spans an algebra $\cR$;
\item  the span
 of $A$ is a right $\cR$-module; and
\item  letting $\Xi$ denote the structure
 matrices for this module, $p$ has the \ct form of 
 \eqref{eq:tropic}; that is,\looseness=-1
 \[
 p(x)= x(I-\Lambda_\Xi(x))^{-1}.
 \]
\een
\end{theorem}

\begin{proof}
A proof appears at the end of Section \ref{sec:redo}. 
\end{proof}

We point out the normalization conditions $p(0)=0$ and $p'(0)=I$ can be enforced e.g.~by an affine linear change of variables on the range of $p$, see Section \ref{sec:normalize} for details.

The proof of Theorem \ref{thm:main}  is based on several intermediate results of independent interest.  
Subsection \ref{sec:approximate} contains results approximating free spectrahedra by more tractable free sets. Subsection \ref{sec:introsatz} 
describes  
 Positivstellensatz  for (matrix-valued) free  analytic functions  with positive real part on a free spectrahedron.

\subsection{Approximating free spectrahedra and free analytic functions}
\label{sec:approximate}
This subsection concerns approximation
of functions analytic on free spectrahedra by analytic  polynomials; 
that is,  a free Oka-Weil theorem. 
An example is the remarkable theorem of
Agler and McCarthy \cite{AM14} (see also \cite{BMV}) stated below as Theorem \ref{thm:AMoka}. 

Given a matrix-valued free analytic polynomial $Q$, the set
\[
 \cG_Q = \{X\in M(\C)^g: \|Q(X)\|<1\}
\]
 is a (semialgebraic) \df{free pseudoconvex} set.  Given $t>1$, let
\[
 K_{t\, Q} =\{X : t \|Q(X)\| \le 1\} \subset \cG_Q.
\]
A (matrix-valued) free analytic function $f$ on a free domain $\cD\subset M(\C)^g$
is \df{uniformly approximable by polynomials} on a subset $\cE\subset \cD$ if for 
each $\epsilon>0$ there is a polynomial $q$ such that $\|f(X)-q(X)\|<\epsilon$
for each $n$ and $X\in \cE(n)$.

\begin{thm}\label{thm:AMoka}
If $f$ is a bounded free analytic function on
a free pseudoconvex set $\cG_Q$, 
then $f$ can be uniformly approximated by
 analytic free polynomials on each smaller set $K_{t\, Q}$, $t>1$.
\end{thm}

\begin{proof}
This result is proved, though not stated in this form,
in Section 9 of \cite{AM14} (cf.~their proof of Corollary 9.7;
see also \cite[Corollary 8.13]{AM14}).
\end{proof}

Free spectrahedra are approximable by free pseudoconvex sets.

\begin{prop}\label{prop:approxIntro}
   If $\cD_A$ is bounded and $t>1$, then there exists 
   free analytic polynomial $Q$ such that
\[
  \cD_A  \subset \cG_Q \subset t\cD_A.
\]
 Moreover, if $\cG_Q$ is a free pseudoconvex set and $\cD_A\subset \cG_Q$, then there
is an $s>1$ such that $\cD_A \subset K_{sQ}.$
Finally, if $p$ is a free rational function analytic on $\cD_A$, then there is a $t>1$ such
 that $p$ is analytic and bounded on $t\cD_A$.
\end{prop}

\begin{proof}
A proof is given near the end of Section \ref{sec:igorcomments}.
\end{proof}

\begin{theorem}
  \label{prop:okadron}
    Suppose $A\in M_d(\C)^g$ and $\cD_A$ is bounded. If $f$ is analytic and bounded on
a free pseudoconvex set $\cG_Q$ containing
    $\cD_A$, then $f$ is uniformly approximable by polynomials on $\cD_A$.
\end{theorem}

\begin{proof}
 A proof is given at the end of Section \ref{sec:igorcomments}. 
\end{proof}

\subsection{Positivstellens\"atze 
and representations
for analytic functions}
\label{sec:introsatz}
We begin this section with  Positivstellens\"atze
and then turn to representations they imply.
 We use \df{nonnegative} and \df{positive} as synonyms for positive semidefinite and positive definite respectively.

\def\al{\alpha}
\begin{thm}[Analytic convex Positivstellensatz]\label{thm:analPossIntro}
Let  $A\in\matdg$ and $e$ be a positive integer. 
Assume $\cD_A$ is bounded and
$G:\cG_Q\to M_e(M(\C))$  is a matrix-valued free function 
analytic on a free pseudoconvex set $\cG_Q$ containing the free spectrahedron $\cD_A$.
 If $G(0)=0$ and $I+G+G^*$ 
is nonnegative on $\cD_A$,
 then there exists
\ben[\rm (1)]
\item   a Hilbert space $H$;
 \item  a formal power series 
$W=\sum_{\al\in\ax}W_{\al}\al$ with coefficients $W_\alpha:\C^e\to H \otimes \C^d$;
\item a unitary mapping $C:H\otimes \C^d\to H\otimes\C^d$ and an isometry $\sW:\C^e\to H\otimes \C^d$,
\een
such that 
the identity
 \begin{equation}\label{eq:posst+}
 I+G(x)+G(x)^* = W(x)^* L_{\IHA}(x) W(x)
\end{equation}
holds in the ring of $e\times e$ matrices over formal power series in $x,x^*$
and there exists a $\tau>0$ such that equation \eqref{eq:posst+} holds
at each $X\in \SR_\tau$.

Moreover, letting $\mathscr{E}= H\otimes \C^d,$  $\ptA= I_H\otimes A$ and $R=(C-I)\ptA$,
the functions  $G$ and $W$ are  given by 
\begin{equation}
\label{eq:Gup}
 G(x)  =  \sW^* C \big(\sum_{j=1}^g \ptA_j x_j\big)\, W(x) \\  
\end{equation}
and
\begin{equation}
\label{eq:WIntro}
  W(x)  =  \big(I_{\mathcal E}-\sum_{j=1}^g R_j x_j\big)^{-1} \sW
\end{equation}
and the coefficients $G_{x_j\alpha}$ of $G$ are given by 
\begin{equation}
\label{eq:general0+}
 G_{x_j \alpha} = \sW^* C \ptA_j R^\alpha \sW;
\end{equation}
for all words $\alpha$.
\end{thm}

\begin{proof}
See Subsection \ref{sec:proofanalpossIntro}.
\end{proof}

An analytic (not necessarily bianalytic) map $p$ maps $\cD_A$ into
$\cD_B$ if and only if $L_A(X)\succeq 0$ implies $L_B(p(X))\succeq
0$. Theorem \ref{thm:analPossIntro} 
 thus provides a
representation for $G(x)=\Lambda_B(p(x))$ with a state space
realization flavor.

\begin{cor}[Rational convex Positivstellensatz]\label{thm:analPoss}
  Let  $A\in\matdg$, $B\in \matetg$, assume that $\cD_A$ is bounded and
 $p:\cD_A\to\cD_B$ satisfies $p(0)=0$. Let $G(x)=\La_B(p(x))$. 
  If $p$ is  either a rational function or a free function analytic 
  and bounded on a free pseudoconvex set $\cG_Q$ containing
  $\cD_A$, 
  then there exists a
  Hilbert space $H$,  a formal power series
  $W=\sum_{\al\in\ax}W_{\al}\al$ with coefficients $W_\alpha:\C^e\to
  H\otimes\C^d$, a unitary  $C:H\otimes \C^d\to H\otimes \C^d$ and an 
 isometry $\sW:\C^e\to H\otimes \C^d$  such that $L_B(p(x))=I+G(x)^* +G(x)$
 and the conclusions \eqref{eq:posst+} -- \eqref{eq:general0+} of Theorem \ref{thm:analPossIntro} hold.
\end{cor}

\begin{proof}
 By Proposition \ref{prop:approxIntro}, in any case we may assume $p$ is analytic on a pseudoconvex set containing $\cD_A$.
Since $p$ is analytic in a pseudoconvex neighborhood of $\cD_A$ so is $G$; and
since $p$ maps $\cD_A$ into $\cD_B$, it follows that $I+G+G^*$ is nonnegative on $\cD_A$. An application of 
Theorem \ref{thm:analPossIntro} completes the proof.
\end{proof}

A key ingredient in the proof of Theorem \ref{thm:analPossIntro}  is Proposition \ref{prop:multi generalIntro}.
It shows that  a Positivstellensatz certificate like that of equation \eqref{eq:posst+}
 suffices to deduce the remaining conclusions of Theorem \ref{thm:analPossIntro}.\looseness=-1

 For hereditary polynomials positive on a free spectrahedron, 
the conclusion of  Theorem \ref{thm:analPossIntro} is stronger.
The weight(s) $W$ 
in the
positivity certificate \eqref{eq:hered2} are 
polynomial,
still analytic  
and we get optimal degree bounds.

\def\mt{\nu}
\begin{thm}[Hereditary Convex Positivstellensatz]
\label{thm:heredposSSIntro}
Let  $A\in M(\C)^g$, and
  let $h\in\C^{\mt\times\mt}\axs$ be an
hereditary
  matrix polynomial of degree $d$. 
 Then $h|_{\cD_A}\succeq0$ if and only if
\beq\label{eq:hered2}
  h=\sum_{k}^{\rm finite} 
h_k^*h_k +
  \sum_j^{\rm finite} f_{j}^* L_A f_{j}
\eeq
for some analytic polynomials 
$   h_{j}\in\R^{\ell\times \mt}\ax_{d+1}$,
$   f_{j}\in\R^{\ell\times \mt}\ax_{d}$. 
Moreover, if $\cD_A$ is bounded, then 
the pure sum of squares term in \eqref{eq:hered2} may be omitted, provided
the $f_j$ are allowed to have  degree $\leq d+1$.
\end{thm}

\begin{proof}
See Section \ref{sec:null} and in particular Theorem \ref{thm:heredposSS}.
\end{proof}

Theorem \ref{thm:all} shows,
under the assumption of a square \df{one term Positivstellensatz certificate} \eqref{eq:1term}
for a mapping $p:\cD_A\to\cD_B$ between free spectrahedra, that $p$ is \ct
and in particular birational between $\cD_A$ and $\cD_B$.

\begin{theorem}
\label{thm:all}
Suppose $A,B\in \matdg$, the set $\{A_1,\dots,A_g\}$ is linearly independent  and  $p=(p^1,\dots,p^g)$ 
is a free formal power series map in $x$ (no $x_j^*$)
 with $p(0)=0$ and  $p^\prime(0)=I$. 
  If there exists a $d\times d$ free formal power series $W$ such that 
\beq\label{eq:1term}
L_B(p( x ))
 = W( x )^\ast L_A( x ) W( x ),
\eeq
  then  $p$ is a \ct map 
\[p(x)= x(I-\Lambda_\Xi(x))^{-1}  \]
  as in \eqref{eq:tropic},
 determined by a module spanned by the set $\{A_1,\dots,A_g\}$  over an algebra of dimension at most $g$ with structure matrices $\Xi$.
\end{theorem}  

\begin{proof}
See Theorem \ref{thm:shotinthedark}.
\end{proof}

In the context of Theorem \ref{thm:main}, the sv-generic condition is
used to show the one term Positivstellensatz hypothesis of Theorem \ref{thm:all} holds.

\subsection{Readers guide}
\label{sec:guide}
The paper is organized as follows.  The polynomial approximation
results of Subsection \ref{sec:approximate} are proved in Section
\ref{sec:approx}. Section \ref{sec:null} contains the proof of Theorem \ref{thm:heredposSSIntro}.
In Section \ref{sec:HeredEq}, key algebraic consequences of an Hereditary
Positivstellensatz representation are collected for use in the following sections.
The proof of Theorem \ref{thm:analPossIntro} appearing in Section
\ref{sec:analPoss} uses the results of the previous three sections.
Theorem \ref{thm:all} is proved in Section \ref{sec:square}. A
somewhat more general version of Theorem \ref{thm:main} is the topic
of Section \ref{sec:redo}. Throughout much of the article the
(bi)analytic maps are assumed to satisfy the normalization $p(0)=0$
and $p^\prime(0)$ a projection. Section \ref{sec:normalize} describes
the consequences of relaxing this assumption.  Section
\ref{sec:examples} provides examples of \ct maps. In the several
complex variables spirit of classifying domains up to affine linear
equivalence, it is natural to ask if there exist matrix convex domains
that are polynomially, but not affine linearly, bianalytic.  The hard
won answer is yes. A class of examples appears in Section
\ref{sec:PQDomain}.

We thank the anonymous referee for carefully reading the manuscript and providing many helpful and thoughtful suggestions.
 
\section{Approximating Free Analytic Functions by Polynomials}
\label{sec:approx}
In this section we prove Proposition
 \ref{prop:approxIntro} and Theorem
\ref{prop:okadron} approximating
free spectrahedra with free pseudoconvex sets
and approximating free mappings analytic on free 
spectrahedra by free polynomials, respectively.

\subsection{Approximating free spectrahedra and free analytic functions using free polynomials}

\def\mB{\mathbb B} For $C>0$, let $\mathfrak F_{C}$ denote the free set of
matrices $T$ such that $C-(T+T^*)\succeq 0$ and for $M>0$, let $\mathfrak
F_{C,M}$ denote those $T\in\mathfrak F_C$ such that $\|T\|< M$.  Let $\varphi$
denote the linear fractional mapping $\varphi(z)=z(1-z)^{-1}$. In particular,
$\varphi$ maps the region $\{z: {\rm Re\,} z\le \frac 12 \}$ in the complex
plane to the set $\{z: |z|\le 1, z\ne -1\}$.  The inverse of $\varphi$ is
$\psi(w)=w(1+w)^{-1}$.  Given $\epsilon>\delta>0$ sufficiently small, the ball
$\mB_\delta(\epsilon)=\{z: |z-\epsilon| \le 1+\delta\}$ does not contain $-1$
and there exists a $K\in(1,2)$ such that $\psi(\mB_\delta(\epsilon))\subset \{z:
{\rm Re\,} z< \frac{K}{2}\}$.

\begin{lemma}
\label{lem:reAhalf}
If $2>C$ and $T\in\mathfrak{F}_C$, then $I-T$ is invertible.  Moreover,
given $M>0$ and $t>0$, there exists $2>C>1$, and $t>\epsilon>\delta>0$ such that if $T\in \mathfrak F_{C,M}$, then
 \[
  \| \varphi(T)-\epsilon \| \le 1+\delta.
 \]
\end{lemma}

\begin{proof}
A routine argument establishes the first part of the lemma. To prove the 
moreover part,
fix $M>0$ and $t>0$ and suppose $\min\{1,t\}>\epsilon>0$.
 Choose $0<\rho <1$ such that both 
\[
 \begin{split}
 \big(2(1-\rho)+\epsilon(1-\rho^2)\big) M^2  & <  \frac 12, \\
 1<C:=\frac{1+2(\rho-\frac14)\epsilon -(1-\rho^2)\epsilon^2}{1+2(\rho-\frac12)\epsilon-(1-\rho^2)\epsilon^2} &<2.
\end{split}
\]
 Let $T\in \mathfrak{F}_{C,M}$ be given. It follows that 
\begin{equation}
 \label{eq:epsrho}
 \begin{split}
  \epsilon \Big(2(1-\rho) &+ {\epsilon}  (1-\rho^2)\Big) T^* T  
    \preceq  \frac{\epsilon}{2} 
    \\
    &\preceq  \frac{\epsilon}{2} + \Big(1+2(\rho-\frac12)\epsilon-(1-\rho^2)\epsilon^2 \Big) \big (C-(T+T^*)\big )\\
    & \preceq  \frac{\epsilon}{2} + \left ( \big(1+2(\rho-\frac14)\epsilon -(1-\rho^2)\epsilon^2\big)-\big(1+2(\rho-\frac12)\epsilon-(1-\rho^2)\epsilon^2\big)(T+T^*)\right ).
\end{split}
\end{equation}
 Let $\delta = \rho \epsilon$ and observe,
\begin{equation}
 \label{eq:epsdel1}
    \epsilon \Big(2(1-\rho) + {\epsilon} (1-\rho^2)\Big) =   (1+\epsilon)^2 -(1+\delta)^2,
\end{equation}
\begin{equation}
 \label{eq:epsdel2}
 \begin{split}
   \frac{\epsilon}{2}+  1+2(\rho-\frac14)\epsilon -(1-\rho^2)\epsilon^2 & = \frac{\epsilon}{2}+ 1+ 2\rho\epsilon  -\frac{\epsilon}{2}-\epsilon^2 +(\rho\epsilon)^2\\
  & = 1+2\delta +\delta^2 -\epsilon^2 = (1+\delta)^2-\epsilon^2,
\end{split}
\end{equation}
and
\begin{equation}
 \label{eq:epsdel3}
  \begin{split} 
    1+2(\rho-\frac12)\epsilon-(1-\rho^2)\epsilon^2
     &=  1 + 2\rho\epsilon -\epsilon -\epsilon^2 + (\rho\epsilon)^2 = (1+\delta)^2 -\epsilon(1+\epsilon). 
 \end{split}
\end{equation}

  Thus, substituting $\delta=\rho\epsilon$ into equation \eqref{eq:epsrho} and using equations \eqref{eq:epsdel1}, \eqref{eq:epsdel2}, \eqref{eq:epsdel3} yields,
\[
 \Big((1+\epsilon)^2-(1+\delta)^2\Big)T^*T \preceq \Big((1+\delta)^2 -\epsilon^2\Big) -\Big((1+\delta)^2-\epsilon(1+\epsilon)\Big)(T+T^*).
\]
 Rearranging gives,
\[
 (1+\epsilon)^2 T^*T -\epsilon(1+\epsilon){(T+T^*)}+\epsilon^2 \preceq (1+\delta)^2 \big (I-(T+T^*) + T^*T \big ) 
\]
 and hence
\[
\big((1+\epsilon)T -\epsilon\big)^* \, \big((1+\epsilon)T -\epsilon\big) \preceq (1+\delta)^2\, (I-T)^*\, (I-T).
\]
 The lemma  follows from this last inequality together with $(1+\epsilon)T -\epsilon=T-\epsilon(I-T)$. 
\end{proof}

\begin{prop}
 \label{lem:uniffrakT}
  For each $M>0$ and $\rho>0$ there exists  $2>C_0>1$ such that for each $1<C<C_0$ there exists $\rho>\epsilon>\delta>0$ such that for each $\eta>0$  there is  an analytic polynomial $q$ in one variable such that  for all  $T\in\,\mathfrak{F}_{C,M}$,
\ben[\rm(1)]
\item $   \|\varphi(T)-\epsilon \| < 1+\delta$; 
 \item $ \| q(T)-\varphi(T)\| <\eta.$
\een
\end{prop}

\begin{proof}
The linear fractional map $\varphi(z)=z(1-z)^{-1}$ is analytic on  $\mathfrak{H}  =\{z\in\mathbb C: {\rm Re\,}z<1\}$. %
By Lemma \ref{lem:reAhalf} with $t=\rho,$ 
 given $M>0$ and $\rho>0$ there exists a $2>C>1$ and $\rho>\epsilon>\delta>0$ such that if $T\in \mathfrak F_{C,M}$, then 
$ \|\varphi(T)-\epsilon \| <1+\delta.$
 Let $\varphi_*(z)= \varphi(z)-\epsilon$. Its inverse is $\psi_*(w)=\psi(w+\epsilon).$ In particular, $\psi_*$ is analytic in a neighborhood of the closed ball $\mB_\delta(0)=\{z\in\,\mathbb C: |z|\leq1+\delta\}$ and for $\delta>0$ sufficiently small, $\psi_*(\mB_\delta(0))$ is a compact subset of $\mathfrak{H}$.   Thus, by Runge's Theorem, there exists a polynomial $p$ such that
\[
 \|p-\varphi_*\|_{\psi_*(\mB_\delta(0))}:=\sup\{|p(z)-\varphi_*(z)|: z\in \psi_*(\mB_\delta)\} <\eta.
\]
Hence, 
\[
 \|p\circ\psi_* - z\|_{\mB_\delta(0)} <\eta.
\]
Now let $T\in\mathfrak{F}_{C,M}$ be given. The matrix $S= \varphi(T)-\epsilon$ has norm at most $1+\delta$ %
and hence 
\[
 \|(p \circ\psi_* )(S) - S\| \le \eta. 
\]
 Equivalently,
\[
 \| p(T) - \varphi_*(T)\| \le \eta.
\]
 Choosing $q=p+\epsilon$, completes the proof.
\end{proof}

 \begin{cor}
 \label{cor:varphibound}
 There exists a $2>C_0>1$ such that for each $M>0$ and $C_0>C>1$, the set $\{\|\varphi(T)\|: T\in \mathfrak{F}_{C,M}\}$ is bounded.
\end{cor}

\begin{proof}
  By Proposition \ref{lem:uniffrakT}, there exists $\epsilon,\delta>0$ such that $\|\varphi(T)-\epsilon\|< 1+\delta$ for $T\in \mathfrak{F}_{C,M}$. Hence,
\[
 \|\varphi(T)\| \le 1+\epsilon +\delta
\]
 for $T\in\mathfrak{F}_{C,M}.$
\end{proof}

\begin{lemma} \label{lem:FGF}
 For each $M>0$  and $2>C>1$ there exists an analytic ($2\times 2$ matrix) polynomial $s$ in one variable such that 
\[
 \mathfrak{F}_{1,M} \subset \cG_s=\{T: \|s(T)\|<1\} \subset \mathfrak{F}_{C,M}.
\]
\end{lemma}

\begin{proof}
 Choose $\rho>0$ such that 
\[
 \frac{M^2+1}{\rho} < C-1
\]
 and let
$
 R^2= M^2+1 +\rho+\rho^2.
$
 In particular,
\begin{equation}
 \label{eq:rhoR}
 \frac{R^2-\rho^2}{\rho} = \frac{M^2+1+\rho}{\rho} <C.
\end{equation}
 Let $s_1(x)=\frac{x+\rho}{R}$,  $s_2(x)= \frac{x}{M}$ and $s=s_1\oplus s_2$.
 Thus, $\|s(T)\|<1$ if and only if $\|T+\rho\|<R$ and $\|T\|<M$. 
 Suppose $T\in \mathfrak{F}_{1,M}$. Then automatically $\|s_2(T)\|<1$. Using $T+T^* \preceq I$, estimate
\[
 (T+\rho)^*(T+\rho) =  T^*T + \rho(T+T^*)+ \rho^2  \preceq  M^2 +\rho +\rho^2 < R^2.
\]
 Thus $\|s_1(T)\|<1$ and the first inclusion of the lemma is  proved.

 Now suppose $\|s(T)\|<1$. Equivalently $\|T\|<M$ and  $\|T+\rho\|<R$. Hence,
\[
\begin{split}
 0& \preceq  R^2 - (T+\rho)^* (T+\rho)   =  R^2 - T^*T -\rho (T+T^*)  -\rho^2   \\
 & \preceq  \rho \Big(\frac{R^2-\rho^2}{\rho} -(T+T^*)\Big)  \preceq   \frac{1}{\rho} \big (C-(T^*+T) \big ),
\end{split}
\] 
 where equation \eqref{eq:rhoR} was used to obtain the last inequality. Thus, $\|s(T)\|<1$ implies $T\in \mathfrak{F}_{C,M}$ and the proof is complete.
\end{proof}

\begin{lem}\label{lem:approx} If $\cD_A$ is bounded and $t>1$, then there exists a matrix-valued free polynomial $Q$ such that 
\[
 \cD_A \subset \cG_Q \subset t\cD_A.
\]
\end{lem}

\begin{proof}
Since $\cD_A$ is bounded, there is an $M>0$ such that $t\|\Lambda_A(X) \|\le M$ for all $X\in\cD_A$.   By Proposition \ref{lem:uniffrakT},
there exists a $2>C>1$ and a sequence of $q_k$ polynomials converging uniformly to $\varphi(z)$ on $\mathfrak F_{C,M}$. Passing to a subsequence if needed, we can assume 
\[
 \|q_k(T)-\varphi(T)\| <\frac{1}{k}
\]
 for $T\in \mathfrak F_{C,M}$. 
 Writing
\begin{multline}\label{eq:unnamed}
 2 \big(q_k(T)^*q_k(T) -\varphi(T)^* \varphi(T) \big) \\ =   \big(q_k(T)-\varphi(T)\big)^* \big(q_k(T)+\varphi(T)\big) +  \big(q_k(T)+\varphi(T)\big)^* \big(q_k(T)-\varphi(T)\big),
\end{multline}
 and using $\varphi(T)$ is uniformly bounded on $\mathfrak F_{C,M}$ (see Corollary \ref{cor:varphibound}), there is a constant $\kappa$ (independent of $k$ and $T$) such that 
\[
 q_k(T)^* q_k(T)-\varphi(T)^*\varphi(T) \preceq \frac{\kappa}{k}.
\]
 Hence,
\[
 I+\frac{\kappa}{k} - q_k(T)^* q_k(T) \succeq I-\varphi(T)^* \varphi(T).
\]
  Thus, if $T\in \mathfrak F_{C,M}$ and $I-\varphi(T)^*\varphi(T)\succeq 0,$ then $I-(1+\frac {\kappa}{k})^{-1} q_k(T)^* q_k(T) \succeq 0$.

Now, given a monic linear pencil $L_A= I+\Lambda_A +\Lambda_A^*$, let 
\[
  Q_k = \left(1+\frac{\kappa}{k}\right)^{-\frac 12} q_k \circ \Lambda_A. 
\]
  If $X\in \cD_A$, then $T=\Lambda_A(X)\in\mathfrak{F}_{1,M}$. Hence $I-Q_k(X)^* Q_k(X)\succeq 0$; that is, 
$
 \cD_A \subset K_{Q_k},
$
 in the notation $K_{Q_k}:=\{X: \|Q_k(X)\|\le 1\}$ of \cite{AM14}. Moreover, since $q_k(T)$ converges to $\varphi(T)$, 
\[
 \cD_A = \bigcap_k^\infty  K_{Q_k}.
\]
 Choose $s$ as in Lemma \ref{lem:FGF} so that
$ \mathfrak{F}_{1,M}\subset \{T:\|s(T)\| \le 1\} \subset \mathfrak{F}_{C,M}.$
 Thus,
\[
 \cD_A \subset \{ X:\|  s(\Lambda_A(X))\| <1\}.
\]
 Consequently, letting 
\[
 \hat{Q}_k = \begin{pmatrix} Q_k &  0 \\ 0 & s\circ \Lambda_A\end{pmatrix},
\]
we have
\[
 \cD_A \subset \{X:\|\hat{Q}_k(X)\| \le 1\}.
\]

 We now turn to showing, given $t>1$, there is a $k$ such that  $\{X:\|\hat{Q}_k(X)\| \le 1\} \subset t\cD_A$.
The estimate  \eqref{eq:unnamed} works reversing the roles of $q_k$ and $\varphi$ giving the inequality
\beq\label{eq:unnamed2}
 \varphi(T)^*\varphi(T)-q_k(T)^*q_k(T) \preceq \frac{\kappa}{k}
\eeq
 for $T\in\mathfrak F_{C,M}$. Now suppose $X\in M(\C)^g$ and $I- \hat{Q}_k(X)^* \hat{Q}_k(X)\succeq 0$. Let, as before $T=\Lambda_A(X)$. It follows that
$\|s(T)\|\le 1$ and hence $T\in \mathfrak{F}_{C,M}$. 
We can thus apply \eqref{eq:unnamed2} to conclude  
\[
\begin{split}
 0 & \preceq 1- Q_k(X)^* Q_k(X) 
= 1- \left(1+\frac{\kappa}k \right)^{-1} q_k(T)^* q_k(T)\\
& \preceq  1+\frac{\kappa}k\left(1+\frac{\kappa}k \right)^{-1}  - 
\left(1+\frac{\kappa}k \right)^{-1}
\varphi(T)^* \varphi(T).
 \end{split}
\]
Let $\tau_k= 1+\frac{2\kappa}{k}$.  This last inequality implies
\[
 (I-T)^{-*}T^*T (I-T)^{-1} \le \tau_k.
\]
 A bit of algebra shows this inequality is equivalent to
\[
 \tau_k - \tau_k(T+T^*) + (\tau_k-1)T^*T  \succeq 0.
\]
 Since
$\tau_k\to1$, for all sufficiently large  $k,$ 
\[
 t> 1+ \frac{\tau_k-1}{\tau_k} M^2.
\]
 Using $T^*T \preceq M^2$ , it follows that
\[
 \tau_k(t- (T+T^*)) \succeq \tau_k +  (\tau_k-1)M^2 -  \tau_k(T+T^*)  \succeq \tau_k \big(I-(T+T^*)\big) + (\tau_k-1)T^*T \succeq 0.
\]
Thus $X\in t \cD_A$.  Summarizing, for sufficiently large $k$,
\[
 \cD_A \subset \{X:\|\hat{Q}_k(X)\| \le 1\} \subset t\cD_A.
\qedhere
\]
\end{proof}

\begin{lem}\label{lem:approx2} Let $A\in M_d(\C)^g$. If $\cD_A$ is bounded and 
$\cG_Q$ is a free pseudoconvex set such that 
$\cD_A \subset \cG_Q$, then there is $s>1$ such that
\beq\label{eq:KG}
 \cD_A \subset K_{sQ} \subset \cG_Q.
\eeq
\end{lem}

\begin{proof}
By definition, $\cD_A \subset \cG_Q$
is equivalent to $\|Q(X)\|<1$ on $\cD_A$. 
For each $M$ the set $\cD_A(M)$ is compact (as $\cD_A$ is bounded and closed). Thus,
for each $M$ there is an $0<r_M<1$ such that $\|Q(X)\|\le r_M$ on $\cD_A(M)$.

For $C\in\R_{>0}$, we have $\|Q(X)\|\le C$ on $\cD_A$  if and only if $C^2-Q^*Q\succeq0$ on $\cD_A$
if and only if  $C^2-Q^*Q\succeq0$ on $\cD_A(N)$ 
for $N:=N(\deg Q,g,d)$ large enough (\cite[Remark 1.2]{HKM12}) if and only if
 $\|Q(X)\|\le C$ on $\cD_A(N)$. Since  $\|Q(X)\|\le r_N <1$ on $\cD_A(N)$, it follows that
$\|Q(X)\|\le r_N<1$ on $\cD_A$.
So  \eqref{eq:KG} holds with $s=\frac{1}{r_N}$.
\end{proof}

\subsection{Rational functions analytic on $\cD_A$}
\label{sec:igorcomments}

In this subsection we show a rational function $p$ without singularities on $\cD_A$ is analytic and bounded on  $t\cD_A$ for some $t>1$. Hence, by Lemma \ref{lem:approx}, $p$ is analytic and bounded on a free pseudoconvex set  $\cG_Q$ containing $\cD_A$.

\begin{lem}\label{lem:rat}
Suppose  $\cD_A$ is bounded and let $r$ be  an analytic noncommutative rational function with no singularities
on $\cD_A$. Then there is a $t>1$ such that $r$ is bounded with 
no singularities on $t\cD_A$. 
\end{lem}

\begin{proof}
Since $r$ is analytic on $\cD_A$ and $\cD_A$ contains $0$, we can consider its
minimal  realization,
\[
r(x)= c^* \big(I-\La_E(x)\big)^{-1} b
\]
for some $e\times e$ tuple $E\in M_e(\C)^g$ and vectors $b,c\in\C^e$. The singularity set of $r$ coincides with the
singularity set (i.e., the free locus \cite{KV}) $\cZ_E$  of $I-\La_E$ 
(see \cite[Theorem 3.1]{KVV09}
and \cite[Theorem 3.5]{Vol17}).

We claim that $\cZ_E\cap\cD_A=\emptyset$ if and only if
$\cup_{1\le e^\prime\le e} \left ( \cZ_E(e)\cap\cD_A(e)\right ) =\emptyset$.
To prove the claim, 
suppose $X\in\cZ_E\cap\cD_A(m).$  Then
for some nonzero $v=\sum_{j=1}^e e_j \otimes v_j$,
where the $e_j$ are standard unit vectors in $\C^e$ and $v_j\in \C^m$, 
\[
\begin{split}
0&=(I-\La_E)(X)v = (I\otimes I - \sum_k E_k \otimes X_k)
(\sum_{j=1}^e e_j \otimes v_j) 
= \sum_{j=1}^e e_j \otimes v_j - \sum_{j,k} E_k e_j \otimes X_k v_j.
\end{split}
\] 
Let $P$ denote the orthogonal projection $\C^m\to \cV ={\rm span}\{v_1,\ldots,v_e\}$ and let $e^\prime = \dim\, \cV.$ Then $P^*XP\in\cZ_E$. Indeed,
for any $u=\sum_i e_i\otimes u_i\in\C^d\otimes\cV,$ 
\[
\begin{split}
u^*(I-\La_E)(P^*XP)v&= (\sum_i e_i\otimes u_i)^*(I\otimes I - \sum_k E_k \otimes P^*X_kP)(\sum_{j=1}^e e_j \otimes v_j) \\
&= \sum_iu_i^*v_i - (\sum_i e_i\otimes u_i)^*(\sum_k E_k \otimes P^*X_kP)(\sum_{j=1}^e e_j \otimes v_j) \\
& = u^*v - \sum_{i,j,k} e_i^*E_ke_j u_i^*P^*X_kPv_j =  u^*v - \sum_{i,j,k} e_i^*E_ke_j u_i^*X_kv_j \\
&=(\sum_i e_i\otimes u_i)^*(I\otimes I - \sum_k E_k \otimes X_k)(\sum_{j=1}^e e_j \otimes v_j) \\
&= u^*(I-\La_E)(X)v =0.
\end{split}
\] 
This calculation shows $P^*XP\in\cZ_E(e')\cap\cD_A(e').$ Hence $\cup_{1\le e^\prime \le e} \left ( \cZ_E(e^\prime)\cap \cD_A(e^\prime) \right ) \ne \emptyset$.
The reverse implication is evident and so the claim is proved.

Since each $\cD_A(e^\prime)$ is compact and disjoint from the closed
set $\cZ_E(e^\prime)$, there exists $t>1$ such that
$t\cD_A(e^\prime)\cap \cZ_E(e^\prime)=\emptyset$ for each $1\le e^\prime\le e$. But 
now using 
$t\cD_A=\cD_{\frac1t A}$, 
the above claim  proves there is a $t>1$ such that
  $r$ has no singularities on $t\cD_A$; that is $\cZ_E \cap \cD_{tA} =\emptyset$. 

We now argue that in fact $r$ is bounded on $t\cD_A$. First observe that if $X_n\in t\cD_A$
and $\|r(X_n)\|$ grows without bound, then so does $\|(I-\La_E(X_n))^{-1}\|$. Hence, there
is a sequence $\gamma_n$ of unit vectors such that $(\|(I-\La_E(X_n))\gamma_n\|)_n$ tends 
to zero. By the argument above, we can replace $X_n$ with $Y_n=V_n^* X_n V_n$ where $V_n$
includes an $e$-dimensional space containing $\gamma_n$ and assume that the $Y_n\in \cD_{tA}(e)$.
By compactness of $\cD_{tA}(e)=t\cD_A(e)$, and passing to a subsequence if needed, 
without loss of generality $Y_n$ converges to some $Y\in \cD_{tA}(e)$
and $\gamma_n$ to some unit vector $\gamma$. It follows that $(I-\La_E(Y))\gamma=0$, and
we have arrived at the contradiction that $Y\in t\cD_{tA}$ and $Y$ is a singularity of
$I-\La_E(x)$.
\end{proof}

The ingredients are now in place to prove Propositions \ref{prop:approxIntro} and \ref{prop:okadron}.

\begin{proof}[Proof of Proposition~\ref{prop:approxIntro}]
The first statements are immediate from Lemmas \ref{lem:approx} and \ref{lem:approx2}. 
Lemma \ref{lem:rat} finishes off the proof.
\end{proof}

\begin{proof}[Proof of Theorem~\ref{prop:okadron}]
Suppose $f$ is analytic and bounded on some $\cG_Q$ containing $\cD_A$. 
By Proposition \ref{prop:approxIntro}, there is $s>1$
with $\cD_A\subset K_{sQ}\subset \cG_Q$. 
By the free Oka-Weil Theorem \ref{thm:AMoka}, $f$ can be uniformly approximated by
polynomials on $K_{sQ}$ and thus on $\cD_A$.
\end{proof}

\section{Hereditary Convex Positivstellensatz}
 \label{sec:null}
 In this section we present a strengthening of the Convex Positivstellensatz \cite{HKM12},
characterizing hereditary polynomials nonnegative on free spectrahedra. In the obtained sum of squares certificate 
all weights will be analytic.

\def\axs{\langle x,x^*\rangle}
\def\ps{\mathcal D}
\def\gtupn{{\matn}^g}
\def\al{\alpha}
\def\be{\beta}

\def\mt{\nu}
\def\ps{\mathcal D}
\def\her{{\rm her}}
\def\eps{\epsilon}

\def\matn{M_n(\C)}

\def\gtupn{\matn^g}

\def\ga{\gamma}
\def\de{\delta}
\def\la{\lambda}
\def\msG{\mathscr G}
\def\msV{\mathscr V}
\def\msY{\mathscr Y}
\def\fF{\mathscr F}
\def\sW{\mathscr W}
\def\sY{\mathscr Y}

\def\Cmt{\C^{\mt\times\mt}}

\def\cX{\mathscr H}
 \def\La{\Lambda}
\def\sV{\mathcal V}
\def\ptA{A}
\def\ptAp{A^\prime}
\def\matdgc{M_d(\mathbb C)^g}
\def\matdhc{M_d(\mathbb C)^\tg}
\def\matngc{M_n(\mathbb C)^g}
\def\fcH{\mathfrak{H}}


\newcommand{\Pu}{P_{12}}
\newcommand{\Pd}{P_{21}}
\newcommand{\Po}{P_{22}}
\newcommand{\abs}[1]{\lvert#1\rvert}
\newcommand{\inpr}[1]{\left\langle#1\right\rangle}
\newcommand{\ip}[1]{\left( #1 \right)}
\newcommand{\adj}[1]{{#1}^*}

 Fix a symmetric $q\in \C^{\ell\times\ell}\axs,$  let
\[
 \ps_q(n) : =\{X\in\gtupn  :   q(X) \succeq 0 \}
\]
for positive integers $n$ and let $\ps_q = (\ps_q(n))_n$.
 Given $\alpha,\beta\in\N$, set 
\begin{equation}
 \label{eq:Malbeta}
  M_{\al, \beta}^{\mt,\her}(q):=
\Big\{\sum_j^{\rm finite} \varphi_j^*\varphi_j+ \sum_i^{\rm finite}  \psi_i^* q \psi_i : 
   \  \psi_i \in \C^{\ell\times \mt}\ax_\beta,\, \varphi_j\in\C^{\mt\times \mt}\ax_\alpha \Big\} 
  \  \subseteq \ \R^{\mt\times \mt}\axs_{\max \{2\al, 2\beta +a\}},
\end{equation}
 where $a=\deg(q)$.
 Observe that $M_{\al,\beta}^{\mt,\her}(q)$ is a proper subset of the quadratic module $M_{\al,\beta}^{\mt}(q)$ as defined
 in \cite{HKM12}. We emphasize  that $\varphi_j,\psi_i$ are assumed to be analytic in 
 \eqref{eq:Malbeta} defining
$ M_{\al,\beta}^{\mt,\her}(q)$. 
Obviously, if $f\in M_{\al,\beta}^{\mt,\her}(q)$ then $f|_{\ps_q}\succeq0$.

For notational convenience,  let $\Sigma^{\mt,\her}_{\al}$ denote
the  cone of sum of squares obtained from $M_{\al, \alpha}^{\mt,\her}(q)$
with $q=1$.

We call $M_{\al,\beta}^{\mt,\her}(q)$ the \df{truncated hereditary quadratic module} defined by $q$.
 We often abbreviate $  M_{\al, \beta}^{\mt,\her}(q)$   to 
$  M_{\al, \beta}^{\mt}$.   If $q(0)=I$ ($q$ is \df{monic}), then
 $\ps_q$ contains an \df{free neighborhood of $0$};
 i.e., there exists $\eps>0$ such that for each $n\in\N$, if
  $X\in\matn^{g}$ and $\|X\|<\eps$,
  then $X\in \ps_q$.  
 Likewise $\ps_q$ is called \df{bounded} provided there is a number $N\in\N$
 for which all $X\in \ps_q$ satisfy $ \|X\| <N$. The following
 theorem is, using the notations above, a restatement of Theorem \ref{thm:heredposSSIntro}.

\begin{thm}[Hereditary Convex Positivstellensatz]
\label{thm:heredposSS}
 Suppose $L\in\C^{\ell\times \ell}\axs$ 
is a monic linear pencil  and $h\in\C^{\mt\times\mt}\axs$ is a 
symmetric hereditary
  matrix polynomial.
 If $\deg(h)= d$, then
\beq\label{eq:posss1}
h(X)\succeq 0 \text{ for all } X\in \ps_L\quad\iff\quad
   h\in M_{d+1,d}^{\mt,\her}(L).
   \eeq
If, in addition, the set $\ps_L$ is  bounded, then
the  right-hand side of this equivalence
is  further
equivalent to
\beq\label{eq:posss2}
h \in
\Big\{  \sum_j^{\rm finite} \psi_{j}^* L \psi_{j} : 
   \psi_{j}\in\R^{\ell\times \mt}\ax_{d+1} \Big\} = : \mathring M_{d+1}^{\mt,\her}(L).
\eeq
\end{thm}

\subsection{Proof of Theorem \ref{thm:heredposSS}}
The proof consists of 
two main steps: 
a separation argument together with a partial
Gelfand-Naimark-Segal (GNS) construction.

\subsubsection{Step 1: Towards a separation argument.}
Let $\Cmt\axs^\her$ denote the vector space of all hereditary $\mt\times\mt$ matrix polynomials.

\begin{lemma}\label{lem:closed}
$M_{\al,\be}^{\mt,\her}(L)$ is a closed convex cone in $\Cmt\axs^\her_{\max\{2\al,2\be+1\}}$. 
\end{lemma}

\begin{proof}
The proof is the same as for the 
corresponding 
free non-hereditary setting \cite[Proposition 3.1]{HKM12}; its main ingredient is
 Carath{\'e}odory's theorem on convex hulls \cite[Theorem I.2.3]{Ba02}.
\end{proof}

 \subsubsection{Step 2: A GNS construction}
\label{subsubsec:flathank}
 Proposition \ref{prop:gns} below,
 embodies the  well known connection, through the Gelfand-Naimark-Segal (GNS)
 construction,
 between operators and
 positive linear functionals. 
It is adapted here to hereditary matrix polynomials.

 Given a Hilbert space $\cX$ and 
 a positive integer $\mt$,  let $\cX^{\oplus \mt}$
 denote the orthogonal direct sum of $\cX$ with itself $\mt$ times.
Let  $L$ be a monic $\ell\times \ell$ linear pencil
 and abbreviate 
\begin{equation*}
 \index{ $M^{\mt}_{k+1} := M^\mt_{k+1,k}(L)$}
    M^{\mt}_{k+1}:= M^{\mt,\her}_{k+1, k}(L).
\end{equation*}

\begin{prop}
 \label{prop:gns}
  If $\la:\Cmt\axs^\her_{2k+2}\to \C$ is a symmetric linear functional  
  that is nonnegative on $\Sigma^{\mt,\her}_{k+1}$ and  positive  on
  $\Sigma^{\mt,\her}_k\setminus\{0\}$, 
  then there exists a tuple $X=(X_1,\dots,X_g)$ of
   operators on a Hilbert space $\cX$ of
  dimension at most $\mt \sigma_\#(k)=\mt \dim \R\ax_k$ and a vector 
  $\ga\in \cX^{\oplus \mt}$ such that
\begin{equation}
 \label{eq:LorX}
  \la(f)= \langle f(X)\ga,\ga\rangle 
\end{equation}
  for all $f\in \Cmt\axs^\her_{k}$, where $\langle\textvisiblespace, \textvisiblespace\rangle$
  is the inner product on $\cX$.  
  Further, if
  $\la$ is nonnegative on $M^{\mt}_{k+1}$, then $X\in\ps_L$. 

  Conversely, if $X=(X_1,\dots,X_g)$ is a tuple of
   operators on a Hilbert space $\cX$ of
  dimension $N$,  
   $\ga$ is a vector in $\cX^{\oplus \mt},$ and $k$ is a positive
  integer, then the linear
  functional  $\la:\Cmt\axs^\her_{2k+2}\to\C$ defined by
\[
  \la(f)= \langle f(X)\ga,\ga\rangle 
\]
  is nonnegative on $\Sigma^\mt_{k+1}$. 
  Further, if $X\in\ps_L$, then $\la$ is nonnegative also on $M^{\mt}_{k+1}$. 
\end{prop}

\begin{proof}
  First suppose that $\la:\Cmt\axs^\her_{2k+2} \to \C$ is 
  nonnegative on $\Sigma^{\mt,\her}_{k+1}$ and  positive
  on $\Sigma_k^{\mt,\her}\setminus\{0\}$.   Consider the 
 symmetric bilinear form,
  defined on the vector space $K=\C^{\mt \times 1}\ax_{k+1}$
  (row vectors of length $\mt$ whose entries are analytic polynomials
  of degree at most $k+1$)  by,
\beq\label{eq:bform}
  \langle f,h\rangle = \la(h^* f).
\eeq
 From the hypotheses, this form is positive semidefinite.

 A standard use of Cauchy-Schwarz inequality shows that the set of
 null vectors
$$ 
 \cN:= \{ f \in K :  \langle f , f \rangle = 0 \}
$$
is a vector subspace of $K$. 
Whence one can endow the quotient
$\tilde \cX := K /\cN$ with the induced positive 
 definite bilinear form
 making it a Hilbert space.  Further,  because
 the form \eqref{eq:bform} is positive definite on 
 the subspace $\cX=\C^{\mt\times 1}\ax_k$, 
 each equivalence class in that set has a unique representative which is 
 a $\mt$-row of analytic polynomials of degree at most $k$. 
 Hence we can consider $\cX$ as a subspace of $\tilde\cX$ with dimension
 $\mt \sigma_\#(k)$.

 Each $x_j$ determines a multiplication operator on $\cX$. 
 For $f=\begin{pmatrix} f_1 & \cdots & f_{\mt}\end{pmatrix}\in \cX$, let 
\[
  x_j f = \begin{pmatrix} x_j f_1 & \cdots & x_j f_{\mt} \end{pmatrix}
\in\tilde\cX
\]
 and define
$X_j : \cX \to \cX$ by 
\[
  X_j f = P x_j f, \quad f \in \cX, \; 1 \leq j \leq g,
\]
 where $P$ is the orthogonal projection from $\tilde\cX$ onto $\cX$ 
 (which is only needed on the degree $k+1$ part of $x_jf$). 
 From the positive definiteness of the bilinear form \eqref{eq:bform}
 on $\cX$, 
one easily sees that each $X_j$ is well defined.

 Let $\gamma \in \cX^{\oplus \mt}$ denote the vector
 whose $j$-th entry, $\gamma_j$ has the empty word
 (the monomial 1)
 in the $j$-th entry and zeros elsewhere.  
 Finally, given words $v_{s,t} \in\ax_{k}$ 
  and $w_{s,t}\in\ax_{k}$ for $1\le s,t \le \mt$,
  choose $f \in \Cmt\axs^\her_{k} $ to  have $(s,t)$-entry 
  $w^*_{s,t} v_{s,t}$. 
  In particular, with $e_1,\dots,e_{\mt}$
  denoting the standard orthonormal basis for $\mathbb R^{\mt}$,
  \[f= \sum_{s,t=1}^\mt w_{s,t}^* v_{s,t} e_s e_t^*.\] Thus, \[
 \begin{split}
  \langle f(X)\gamma,\gamma\rangle 
    &=  \sum \langle f_{s,t}(X) \gamma_t,\gamma_s\rangle 
     =  \sum \langle w_{s,t}^*(X)v_{s,t}(X)\gamma_t, \gamma_s \rangle 
      = \sum \langle v_{s,t}(X)\gamma_t, w_{s,t}(X) \gamma_s \rangle \\
     & = \sum \langle v_{s,t} e_t^*, w_{s,t} e_s^* \rangle 
      = \sum  \la( w_{s,t}^* v_{s,t} e_s e_t^* ) 
      = \la\big(\sum (w_{s,t}^* v_{s,t} e_s e_t^*)\big)      = \la(f). 
 \end{split}
\]
  Since any $f\in  \Cmt\axs^\her_{k}$ 
  can be written as a linear combination
  of words of the form $w^*v$  with
  $v,w\in\ax_{k}$   as was done above,
  equation \eqref{eq:LorX} is established.

  To prove the further statement, suppose $\la$ is nonnegative
  on $M^{\mt}_{k+1}$. 
Write $L=I+\Lambda+\Lambda^*$, where $\Lambda$ is the homogeneous linear analytic part of $L$.  
  Given
\[
  \psi =\begin{pmatrix} \psi_1 \\ \vdots \\ \psi_\ell \end{pmatrix}
    \in \cX^{\oplus \ell},
\]
  note that
\begin{equation*} 
 \begin{split}
   \langle L(X) \psi, \psi \rangle 
   & = \langle (I+\La(X)+\La(X)^*)\psi,\psi  \rangle =
\langle (I+\La(X))\psi,\psi  \rangle + \langle \psi,\La(X)\psi \rangle    \\
& =    \langle \psi +\sum A_j P x_j \psi, \psi\rangle  + \langle \psi,\sum A_j P x_j \psi\rangle 
   = \langle \psi +\sum A_j  x_j \psi, \psi\rangle  + \langle \psi,\sum A_j  x_j \psi\rangle \\
& = \langle (I +\La(x)) \psi, \psi\rangle  + \langle \psi, \La(x) \psi\rangle 
   =  \la\big(\psi^* (I+\La_A(x))\psi\big)  + \la (\psi^* \La(x)^* \psi)  \\
&   = \la\big(\psi^* (I+\La(x)+\La(x)^*)\psi\big) 
   = 
\la ( \psi^* L \psi)    \ge 0.
 \end{split}
\end{equation*}
Hence, $L(X) \succeq 0$.

  The proof of the converse  is routine  and is not
  used in the sequel. 
\end{proof}

\subsubsection{Step $3$: Conclusion} 
Let us first prove \eqref{eq:posss1}.
The reverse implication %
is obvious. To prove  the forward implication 
assume $h\not\in M_{d+1,d}^{\mt,\her}(L)$. 
By Lemma \ref{lem:closed} then there is a linear functional
$\la: \Cmt\axs^\her_{2d+2}\to\C$ satisfying
\[
\la(h)<0, \qquad \la(M_{d+1,d}^{\mt,\her})\subseteq \R_{\geq0}.
\]
By adding a small multiple of a linear functional that is strictly
positive on $\Sigma^{\mt,\her}_{d+1}\setminus\{0\}$
(see e.g.~\cite[Lemma 3.2]{HKM12} for its existence), we may assume moreover that
\[
\la \big(\Sigma^{\mt,\her}_{d+1}\setminus\{0\}\big)\subseteq\R_{>0}.
\]
Now Proposition \ref{prop:gns} applies: 
 there exists a tuple $X=(X_1,\dots,X_g) \in \cD_L$ of
$\mt\sigma_\#(d) \times \mt\sigma_\#(d) $ matrices, and a
vector $\ga$  such that
\eqref{eq:LorX} holds
  for all $f\in \Cmt\axs^\her_{d}$. Hence
  \[
  0 > \la(h) = \langle h(X)\ga,\ga\rangle.
  \]
 Thus $h(X)\not\succeq0$. 
  
Finally, \eqref{eq:posss2} follows from the Hahn-Banach theorem. Namely, if $\cD_L$ is bounded, then
$I=\sum_j V_j^*L(X)  V_j$ for some $V_j$, see e.g.~\cite[Section 4.1]{HKM12}.  
  \hfill 
 \qedsymbol

\section{Positivity Certificates for Analytic Mappings}
 \label{sec:HeredEq}
This section chronicles consequences of a Positivstellensatz certificate 
like that of equation \eqref{eq:posst+}. 
Proposition \ref{prop:multi generalIntro} is the principal result.

Given a $g$-tuple $\ptA$ of operators on a Hilbert space $\mathcal E$, a positive integer $e$ and
a formal powers series $W(x)=\sum W_\alpha \alpha$ with coefficients $W_\alpha:\C^e\to \mathcal E$,
and $G=\sum G_\alpha \alpha$ with  coefficients $G_\alpha:\C^e\to\C^e$  and $G(0)=G_\emptyset =0$, the identity
\begin{equation}
\label{eq:multi general0}
 I_e+G(x)+G(x)^* = W(x)^* L_{\ptA}(x) W(x)
\end{equation}
is interpreted as holding in the ring of (matrices over) formal power series in $x,x^*$. Equivalently, 
for words $\alpha,\beta$ and $1\le j\le g$,
\begin{equation}
 \label{eq:preiso1alt}
\begin{split}
     W_\beta^*\ptA_k^* W_{x_j\alpha} + W_{x_k\beta}^* \ptA_j W_{\alpha} + W_{x_k\beta}^* W_{x_j\alpha} & =   0\\
    W_{\emptyset}^* ( \ptA_k W_{\alpha} + W_{x_k \alpha}) & = G_{x_k\alpha} \\
   W_{\emptyset}^* W_\emptyset &=  I.
\end{split}
\end{equation}

\begin{prop}
 \label{prop:formalveval}
 Suppose $e$ is a positive integer, 
 $G$ is an $e\times e$ matrix-valued free analytic function,
 $\cE$  is a (not necessarily finite-dimensional) Hilbert
 space, $W$ is a formal power series with coefficients $W_\alpha:\C^e\to \cE$
 and  $\ptA$ is a $g$-tuple of operators on $\cE$.

  The following are equivalent.
\ben[\rm(i)] 
 \item \label{it:fe1} Equation  \eqref{eq:multi general0} holds in the ring of formal power series. Equivalently, the equations \eqref{eq:preiso1alt} hold.
 \item \label{it:fe2} For all nilpotent $X\in M(\rC)^g$,
\beq\label{eq:GX0}
 I+G(X)+G(X)^* = W(X)^* L_{\ptA}(X) W(X).
\eeq
\een

In addition, if $W$ and $G$ have positive formal radii of convergence at least $\tau>0$, then items \eqref{it:fe1} and \eqref{it:fe2} are equivalent to
\begin{enumerate}[\rm(i)]
\setcounter{enumi}{2}
 \item\label{it:fe3}  Equation \eqref{eq:GX0} holds for all $X\in \SR_\tau$.
\end{enumerate}
\end{prop}

Before beginning the proof of Proposition \ref{prop:formalveval}, we
first state and prove a routine lemma.
Fix $N$ a positive integer. Consider the truncated Fock Hilbert space
$\mathscr F_N$ with orthonormal basis $\{\alpha\in\ax : |\alpha|\le
N\}$. Let $S$ (we suppress the dependence on $N$) denote the tuple of
shifts determined by $S_jw =x_j w$ if the length of the word $w$ is
strictly less than $N$ and $S_j w =0$ if the length of the word $w$ is
$N$. In particular, $S$ is nilpotent of order $N$.

\begin{lemma}
\label{lem:faith} 
Given Hilbert spaces $H$ and $K$ and operators
$F_{\alpha,\beta}:H\to K$ parameterized over words $\alpha,\beta$,
of length at most $N$, if 
\[
 \sum_{|\alpha|,|\beta|\le N} F_{\alpha,\beta}\otimes S^{\beta} S^{*\alpha} = 0,
\]
 then $F_{\alpha,\beta}=0$ for all $\alpha,\beta$.
\end{lemma}

\begin{proof}
We argue  by induction on the length of $\alpha$.  In the case $\alpha =\emptyset$, 
evaluating at vectors of the form $h\otimes \emptyset$ with  $h\in H$ gives
\[
 0= \sum_{|\alpha|,|\beta|\le N} F_{\alpha,\beta}h \otimes S^{\beta} S^{*\alpha}\emptyset  = \sum_{|\beta|\le N} F_{\emptyset,\beta} h \otimes \beta.
\]
 Hence $F_{\emptyset,\beta}=0$ for all $|\beta|\le N$. Now suppose $0\le n<N$ and $F_{\alpha,\beta}=0$ for all $|\alpha|\le n$ and $|\beta|\le N$.
Let a word $\gamma$ with length $n+1$ be given. Evaluating at vectors of the form $h\otimes \gamma$ and using the induction hypothesis gives,
\[
 \begin{split}
 0= & \sum_{|\alpha|,|\beta|\le N} F_{\alpha,\beta}h \otimes S^{\beta} S^{*\alpha}\gamma \\
  = &  \sum_{n<|\alpha|\le N,|\beta|\le N} F_{\alpha,\beta}h \otimes S^{\beta} S^{*\alpha}\gamma \\
  = & \sum_{|\beta|\le N} F_{\gamma,\beta} h\otimes \beta
\end{split}
\]
Hence $F_{\gamma,\beta}=0$ for all $|\beta|\le N$.
\end{proof}

\begin{proof}[Proof of Proposition~\ref{prop:formalveval}]
 Suppose item \eqref{it:fe2} holds. Thus for all nilpotent tuples $X$,
\[
 I + G(X)+G(X)^* - W(X)^* L_A(X)W(X) =0.
\]
In this case Lemma \ref{lem:faith} implies the identities of 
equation \eqref{eq:preiso1alt} hold. Hence item \eqref{it:fe2} implies item
\eqref{it:fe1}. That item \eqref{it:fe1} implies \eqref{it:fe2} is
evident.  

Under the added hypotheses on the radii of convergence, item \eqref{it:fe3} implies \eqref{it:fe2}.
It remains to prove the converse. Accordingly, suppose $X\in \SR_\tau$.
 Let $R$ and $L$ denote the values of the right and left hand side of \eqref{eq:GX0}
evaluated at $X$ respectively and let $G^{(N)}$ and $W^{(N)}$ denote the $N$-th partial sums of the respective series $G(X)$ and $W(X)$.  
 Given $\epsilon>0$ there is an $N$ such that 
\[
\begin{split}
 \|I+G^{(N)}(X)+G^{(N)}(X)^* - R\| &<\epsilon,\\
 \|W^{(N)}(X)^* L_{\ptA}(X)W^{(N)}(X) - L\|&<\epsilon.
 \end{split}
\]
 Use \eqref{eq:preiso1alt} to compute
\begin{multline}\label{eq:psCalc2}
 I+G^{(N)}(X)+G^{(N)}(X)^* - W^{(N)}(X)^*L_{\ptA}(X) W^{(N)}(X) \\
 = -\sum_{k=1}^g  \sum_{|\beta|=N}\sum_{|\alpha|\le N} W_\beta^* \ptA_k^* W_\alpha \otimes X^{*\beta}X_k^* X^\alpha 
   - \sum_{j=1}^g  \sum_{|\beta|\le N}\sum_{|\alpha|=N} W_\beta^* \ptA_j W_\alpha \otimes X^{*\beta}X_j X^\alpha \\
   =
-\Big(  \sum_{|\beta|=N}  W_\beta \otimes X^{\beta}\Big)^* \La_\ptA(X)^* \Big(  \sum_{|\al|\leq N}  W_\al \otimes X^{\al}\Big)
   -
\Big(  \sum_{|\beta|\leq N}  W_\beta \otimes X^{\beta}\Big)^* \La_\ptA(X) \Big(  \sum_{|\al|=N}  W_\al \otimes X^{\al}\Big)
       .
\end{multline}
The norm of each of the two summands in the last line of \eqref{eq:psCalc2}
is at most
\beq\label{eq:psCalc}
  \sum_{k=1}^g \|\ptA_k\otimes X_k\|  \big(\sum_{|\beta|=N} \|W_\beta\|\, \|X^\beta\| \big)
 \big(\sum_{|\al|\leq N} \|W_\al\|\, \|X^\al\| \big)
  .
\eeq
By hypothesis the second factor in \eqref{eq:psCalc}
tends to $0$ with $N$ and the 
first and third factor are uniformly bounded on $\SR_\tau$.
Thus the left hand side of \eqref{eq:psCalc2} tends to zero with $N$ and the proof is complete. 
\end{proof}

With the notations already introduced, let
\[
 \rg(\ptA,W) =\{\ptA_j W_\alpha h: 1\le j\le g, \, \alpha, \, h\in \C^e\}.
\]

\begin{prop}
 \label{prop:multi generalIntro}
   Suppose $e$ is a positive integer and $\mathcal E$ is a separable Hilbert space.  If
\ben[\rm (a)]
  \item  $\ptA$ is a $g$-tuple of operators on $\mathcal E$; 
  \item   $W$ is a formal power series with coefficients $W_\alpha :\rC^e \to \mathcal E$; and
  \item   $G$ is a formal power series with $G(0)=0$ and $e\times e$ matrix coefficients $G_\alpha$  such that  
 equation \eqref{eq:multi general0} holds,
 \een
   then $\mathscr{W}=\Wem:\rC^e\to \mathcal E$ is an isometry and
  there exists a contraction $C:\mathcal E\to\mathcal E$ that is isometric on $\rg(\ptA,W)$ 
such that 
\[ 
\begin{split}
 G(x) =  \sW^* C \big(\sum_{j=1}^g \ptA_j x_j\big) W(x),
          \end{split}
\]
 where, letting $R=(C-I_{\mathcal E}) \ptA$, the function $W$ is given as in equation \eqref{eq:WIntro}
\[ 
W(x)  = \big(I_{\mathcal E}-\sum_{j=1}^g R_j x_j\big)^{-1} \sW.
\] 

 Moreover, if $\mathcal E$ is finite dimensional, then $C$ can be chosen unitary. In any case, 
choosing an auxiliary separable infinite dimensional Hilbert space $\cEp$ and a tuple $\ptAp$ acting on $\cEp$ and letting 
\[
 \tcE= \mathcal E\oplus \cEp, \ \ \ \tptA =\begin{pmatrix} \ptA&0\\ 0 & \ptAp \end{pmatrix}, \ \ \ 
 \tsW = \begin{pmatrix} \sW \\ 0\end{pmatrix}:\C^e\to \tcE, \ \ \ \tilde{W} = \begin{pmatrix} W \\ 0\end{pmatrix},
\]
we have
\[
I_e+G(x)+G(x)^* = \tilde{W}(x)^* L_{\tptA}(x) \tilde{W}(x),
\]
and there is a unitary mapping $\tilde{C}:\tcE\to\tcE$ such that, letting $\tilde{R} = (\tilde{C}-I_{\tcE})\tptA$,
\[
 \tW(x) =  \big(I_{\mathcal E}-\sum_{j=1}^g \tilde{R}_j x_j\big)^{-1} \tsW, \ \ \
 G(x) =  \tsW^* \tilde{C} \big(\sum_{j=1}^g \tptA_j x_j\big) \tilde{W}(x).
\]

In particular, 
\ben[\rm(1)]
 \item \label{it:recure}
  the following recursion formula holds,
\[
 W_{x_j\alpha} = (C-I_{\mathcal{E}})\ptA_j W_\alpha, \ \ \  {\tW}_{x_j\alpha} = (\tilde{C}-I_{\tcE})\tptA_j {\tW}_\alpha ;
\]
\item \label{it:Bj}
   the $x_j$ coefficient of $G(x)$ is
\[
  G_{x_j} = \sW^* C \ptA_j\sW = \tsW^* \tilde{C}\tptA_j \tsW;
\]
 \item \label{it:Balpha}
   more generally, $G_{x_j\alpha}$, the $x_j \alpha$ coefficient of $G,$ is
\begin{equation}
\label{eq:tildegeneral0}
 G_{x_j \alpha} = \sW^* C \ptA_j R^\alpha \sW = \tsW^* \tilde{C} \tptA_j {\tilde{R}}^\alpha \tsW.
\end{equation}
\een

 Conversely, given a tuple $\ptA$ on a Hilbert space $\mathcal{E}$ a contraction 
 $C:\mathcal{E}\to\mathcal{E}$ that is isometric on the range of $A$ and an isometry  $\sW:\rC^d\to \mathcal E$, 
  defining $R=(C-I) \ptA,$ and  $W$ and $G$ as in equations \eqref{eq:WIntro} and \eqref{eq:Gup}, the identities of  \eqref{eq:preiso1alt} hold. 
\end{prop}

\begin{proof}
  Completing the square in the first equation of \eqref{eq:preiso1alt} gives,
\begin{equation}
 \label{eq:preiso2w}
   (\ptA_kW_\beta + W_{x_k\beta})^* \, (\ptA_j W_\alpha + W_{x_j\alpha}) = W_\beta^* \ptA_k^* \ptA_j W_\alpha.
\end{equation}

 Fix, for the moment, a positive integer $N$. Recall  $W_\alpha:\rC^e\to\cE$ and 
 $\ptA_j:\cE\to\cE.$   Let $\mathcal K_N = \oplus_{|\alpha|\le N} \rC^e$, the Hilbert space direct sum of $\rC^e$ over the set of words  
 of length at most $N$ in the variables $x=(x_1,\dots,x_g)$.  Finally, let $\mathcal L_N :=\oplus_{j=1}^g \mathcal K_N$
   and note that $h\in\mathcal L_N$ takes the form $h=\oplus_j \oplus_{|\alpha|\le N} h_{j,\alpha}=\oplus h_{j,\alpha}$.  Let
\[
 \mathcal E_N = \Big\{ \sum_{j,|\alpha|\le N} \ptA_jW_\alpha \, h_{j,\alpha}: h_{j,\alpha} \in \C^e  \mbox{ for } 1\le j\le g, \ \ |\alpha|\le N\Big\}
  \subset \rg(\ptA) \subset \mathcal E.
\]
 The subspaces $\mathcal E_N$ are nested increasing and $\rg(\ptA,W)=\cup_N \mathcal E_N$.

  Define $X_N,Y_N:\mathcal L_N\to \mathcal E$,
\[
\begin{split}
   X_N ( \oplus_{j=1}^g \oplus_{|\alpha|\le N} h_{j,\alpha} ) & = \sum_{j,|\alpha|\le N} (\ptA_jW_\alpha +W_{x_j\alpha}) h_{j,\alpha}\\
    Y_N (\oplus_{j=1}^g \oplus_{ |\alpha| \le N}  h_{j,\alpha} ) &= \sum_{j,|\alpha|\le N}  (\ptA_jW_\alpha) h_{j,\alpha}.
    \end{split}
\]
 Note that the range of $Y_N$ is $\mathcal E_N$.   Equation \eqref{eq:preiso2w} implies that
\begin{equation}
 \label{eq:lurk}
   X_N^* X_N = Y_N^* Y_N:\mathcal L_N \to \mathcal L_N. 
\end{equation}
In particular, if $Y_N h=0$, then $X_N h=0$. Hence, 
$Y_N h \mapsto X_N h$ is a well-defined map
$C_N:\mathcal E_N
\to \mathcal E$.  Further,
equation \eqref{eq:lurk} implies that $C_N$ is an isometry. Since $\mathcal E_N$ 
is finite-dimensional,  $C_N$ can be extended to a unitary $C_N:\mathcal E\to \mathcal E$.  
Thus, there is a unitary  mapping $C_N:\mathcal E\to \mathcal E$ such that
\[
   X_N=C_N Y_N. 
\]
Moreover, for $N\ge M$,
\[
 X_M =C_N Y_M.
\]

Since $(C_N)_N$ is a sequence of unitaries on $\mathcal E$, a subsequence $(C_{N_j})_j$ 
converges in the weak operator topology (WOT) to a contraction operator $C$.
(In the case $\mathcal E$ is finite dimensional, $C$ is unitary.) Fix $M$. For $N_j\ge M$,
\(
 C_{N_j}Y_M=X_M.
\)
 Hence, for a vector $h\in \mathcal L_M$ and a vector $e\in\mathcal E$,
\[
  \langle C Y_Mh,e\rangle =  \lim_j \langle C_{N_j} Y_Mh,e\rangle = \langle X_M h,e\rangle.
\]
 Thus, $C Y_M=X_M$ for all $M$.   In particular, $C$ is an isometry on $\rg(A,W)$ and, combined with 
 the second and third identities in equation \eqref{eq:preiso1alt},  for each $j,\alpha$
\begin{equation}
 \label{eq:pre-recursive-general}
\begin{split} 
  \ptA_jW_\alpha +W_{x_j\alpha} & =  C \ptA_jW_\alpha \\
     W_{\emptyset}^* ( \ptA_k W_{\alpha} + W_{x_k \alpha}) & = G_{x_k\alpha} \\
   W_{\emptyset}^* W_\emptyset &=  I.
\end{split}
\end{equation}
The first identity in equation \eqref{eq:pre-recursive-general} is equivalent to
 \begin{equation}
  \label{eq:isorecursive general}
     W_{x_j\alpha} = (C-I)\ptA_j W_\alpha.
 \end{equation}

 Now suppose $C:\mathcal E\to\mathcal E$ is a contraction that is isometric on $\rg(\ptA,W)$ and the identities of
  equation \eqref{eq:pre-recursive-general} hold.  In particular, equation \eqref{eq:isorecursive general}
 also holds.  For notational ease, if not consistency, let $\sW=\Wem$ and $W_\ell = W_{x_\ell}$. In particular,
 from $\sW$ is an isometry. Moreover,  it follows
  from equation \eqref{eq:isorecursive general} with $\alpha=\emptyset$, that $ W_j = (C-I)\ptA_j \sW$
   for each $j$.  Thus $W_j=R_j\sW$, where $R_j=(C-I)\ptA_j$.

  For each $k$ an application of equation \eqref{eq:isorecursive general}
  with $\alpha = x_j$ yields
\[
  W_{x_k x_j} = (C-I)\ptA_k W_j = R_k R_j \sW,
\]
 for  $j,k =1, \dots g.$  Induction on the length  of words gives,
\[
 W_\alpha = R^\alpha \sW
\]
 where  $R=(R_1,\dots,R_g)$.  Hence, 
\begin{equation}
\label{eq:W general}
 W(x) = (I-\sum R_j x_j)^{-1} \sW.
\end{equation}

  Now using the second and third equations of \eqref{eq:preiso1alt} together with
   \eqref{eq:W general} gives
\[
 \sW^* (I+\sum \ptA_k x_k)  (I-\sum R_\ell x_\ell)^{-1} \sW = I +G(x). 
\]
 Hence, 
\[
\begin{split}
 G(x)&=  \sW^* \Big[ \big (I+\sum \ptA_k x_k\big )  \big (I-\sum_{\ell=1}^g  R_\ell x_\ell\big )^{-1}  - I\Big] \sW \\
  & =  \sW^*\big(\sum_{k=1}^g (\ptA_k+R_k)x_k\big) \,  \big(I-\sum_{\ell=1}^g R_\ell x_\ell\big)^{-1} \sW \\
 & =  \sW^* C\big(\sum \ptA_k x_k\big)\, \big(I-(C-I)\sum \ptA_\ell x_\ell \big )^{-1}  \sW  \\
         & =  \sW^* C\big(\sum \ptA_k x_k\big)\, \big(I-\sum R_\ell x_\ell \big )^{-1} \sW.
\end{split}
\]
  At this point we have proved that if $W$ and $G$ solve equation \eqref{eq:multi general0}, then
   there exists a contraction $C$ that is  isometric   on $\mathcal E$ such that $W$ and $G$ have the claimed form.
   Further, in the case $\mathcal E$ is finite dimensional,  $C$ can be chosen unitary.

   Now let $\tcE$, $\tptA$, $\tsW$ and $\tilde{W}$ be given as in the statement of the proposition. In particular
   equation \eqref{eq:tildegeneral0} holds.  Further, 
   Let  $\rg(\tptA,\tilde{W})=\rg(\ptA,W)\oplus (0)\subset \tcE$. The orthogonal complements of 
   $\rg(\tptA,\tilde{W})$ and $C(\rg(\tptA,\tilde{W}))$ in $\tcE$ are  infinite dimensional separable Hilbert spaces. Hence
   there exists a unitary operator $\tilde{C}:\tcE\to\tcE$ such that $\tilde{C}(h\oplus 0)=Ch\oplus 0$ for $h\in \rg(\ptA,0)$.
    Thus $\tilde{C}$ is unitary and $\tilde{C}$, 
   $\tilde{W}$ and $\tptA$ together satisfy the analog of equation \eqref{eq:pre-recursive-general} and
   hence the conclusions of the proposition.

To prove the converse,  given a  tuple $\ptA=(\ptA_1,\dots,\ptA_g)$ of operators on the Hilbert space $\cE$,
 and a contraction $C:\mathcal E\to\mathcal E$ that is isometric on the range of $A$ and an isometry
  $\sW:\rC^d\to\cE$,  let 
  $R=(C-I)\ptA$ and   $W(x) = (I-\La_R(x))^{-1}\sW$ and define $G$ by equation \eqref{eq:Gup},
\[
  G(x) =  \sW^* C\Lambda_\ptA(x) (I-\Lambda_R(x))^{-1} \sW.
\]
By construction, $W_\alpha = R^\alpha \sW$. Moreover, 
for each word $\alpha$ and $1\le j\le g$,
\[
 (C-I)\ptA_j W_\alpha =R_j R^\alpha \sW = W_{x_j\alpha}.
\]
Hence,
\[
 C\ptA_j W_\alpha = \ptA_j W_\alpha + W_{x_j\alpha}.
\]
Since $C$ is isometric on the range of $\ptA$, given $1\le k\le g$
 and a word $\beta$, 
\[
 W_\beta^* \ptA_k^* \ptA_j W_\alpha = W_\beta^* \ptA_k^* \ptA_j W_\alpha +  W_\beta^* \ptA_k^* W_\alpha + 
  W_{x_k\beta}^* \ptA_j W_\alpha  + W_{x_\beta}^* W_{x_j\alpha}. 
\]
Thus the first of the identities of equation \eqref{eq:preiso1alt} hold. The third identity holds
since $\sW=\Wem$ is an isometry and, as $\sW=\Wem$, the second
identity holds by the choice of $G$ and the proof is complete.
\end{proof}

\begin{remark}\rm
 We note that the proof of the converse of Proposition \ref{prop:multi generalIntro} would, under some
  convergence assumptions, follow from the following formal calculation starting from the formula
 for $G$ of equation \eqref{eq:Gup}. Using $R_j= (C-I)\ptA_j$  gives
\[
 I-\La_R(x)  = I+\Lambda_\ptA(x) -C\Lambda_\ptA(x). 
\] 
  Let 
\[
 H(x) = C\Lambda_\ptA(x)  (I-\Lambda_R(x))^{-1} = \Lambda_{C\ptA}(x)  (I-\Lambda_R(x))^{-1}
\]
 and note that
\[
 \sW^* \big(I +H(x)+H^*(x)\big)\sW =I + G(x) +G^*(x).
\]
 Now,
 \begin{multline*}
   I+H(x)+H(x)^* \\
   =  (I-\La_R(x))^{-*} [ (I-\La_R(x))^* (I-\La_R(x)) + (I-\La_R(x))^* C \Lambda_\ptA(x) + \Lambda_\ptA(x)^* C^* (I-\La_R(x))] (I-\La_R(x))^{-1} \\
   = (I-\La_R(x))^{-*} [ (I-\Lambda_R(x)+ C\Lambda_\ptA(x))^*  (I-\Lambda_R(x)+C\Lambda_\ptA(x)) - \Lambda_\ptA(x)^* C^* C \Lambda_\ptA(x)] (I-\La_R(x))^{-1} \\
   = (I-\La_R(x))^{-*} [ \Psi(x)^* \Psi(x) - \Lambda_\ptA(x)^*  \Lambda_\ptA(x)] (I-\La_R(x))^{-1} \\
  =  (I-\La_R(x))^{-*} [ I+\Lambda_\ptA(x)^* + \Lambda_\ptA(x)] (I-\La_R(x))^{-1},
 \end{multline*}
 from which it follows that
\[
 \sW^* L_R(x)^{-*} \big(I+\Lambda_\ptA(x)+\Lambda_\ptA^*(x)\big) L_R(x)^{-1} \sW = I +G(x)+G^*(x).\qedhere
\]
\end{remark}

\subsection{Polynomials correspond to nilpotent $R$}

\begin{cor}
 \label{cor:poly-nil}
  Suppose, in the context of Proposition {\rm\ref{prop:multi generalIntro}}, that $W_\emptyset$ is $e\times D$.
  If 
  \ben[\rm (a)]
   \item $G$ is a polynomial; 
   \item\label{it:control} $\Span \{ R^\omega \Wem h:h\in\rC^d, \omega \in\ax\} =\rC^D$; and
   \item\label{it:observe} $\bigcap \{ \ker(\Wema C \ptA_j R^w) :  w\in\ax,\ 1\le j\le g\}=(0),$
  \een
   then the tuple $R$ is nilpotent. 
   In particular, if $D=d$ and $G$ is a polynomial, then $R$ is nilpotent. 
\end{cor}

  In the language of systems theory, the hypotheses of items \eqref{it:control} and \eqref{it:observe} are that  the system $(R, \Wem, \{\Wema C \ptA_j\})$
  is \df{controllable} and \df{observable} respectively.

\begin{proof}
Since
$W_\emptyset^\ast C \Lambda_\ptA(x) ( I- \sum R_j x_j )^{-1}
    W_\emptyset
$ is a polynomial, there exists a positive integer $N$ such that 
\[
  W_\emptyset^\ast C \ptA_i R^\omega \Wem = 0
\]
 for all words $\omega$ for which $|\omega|\ge N$. 
Hence, if $|\xi|\ge N$, then for words $\alpha,\beta$,
\[
  0=   W_\emptyset^\ast C \ptA_i R^\omega \Wem
= W_\emptyset^\ast C \ptA_i R^\alpha R^\xi R^\beta \Wem
\]
 Conditions \eqref{it:control} and \eqref{it:observe} now imply that $R^\xi = 0$. 
\end{proof}

\begin{remark}\rm
In any case, $W$ is a polynomial if and only if  $R^\alpha W_\emptyset =0$ for $|\alpha|$ large enough.  
 Of course if $W_\emptyset$ is square, then  it is invertible. Thus, in this case,  the $R_j$ 
 are jointly nilpotent if and only if  $W$ is a polynomial. 
\end{remark}

\def\ptA{A}
\section{Extending the Hereditary Positivstellensatz to  Analytic Functions}
\label{sec:analPoss}
In this section we prove Theorem \ref{thm:analPossIntro}, 
extending the Hereditary Convex Positivstellensatz
(Theorem \ref{thm:heredposSS}) to analytic and rational maps
between free spectrahedra. The proof combines Theorems 
 \ref{thm:heredposSS} and  \ref{prop:okadron} and Proposition \ref{prop:multi generalIntro}.

\begin{lemma}
\label{lem:betterpl}
Suppose $e$ is a positive integer, $G:\cG\to M_e(M(\C))$ is analytic on a pseudoconvex set $\cG$ containing $\cD_A$ and $I+G+G^*$ is nonnegative on $\cD_A$ 
and $G(0)=0$.
If  $(G_\ell)$ is a sequence  of $e\times e$ matrix  polynomials converging uniformly to $G$ on $\cD_A$, then there exists a sequence of polynomials $(Q_k)$ converging uniformly to $G$ on $\cD_A$ such that $Q_k(0)=0$ and  $I+Q_k +Q_k^*$ is nonnegative on $\cD_A$.
\end{lemma}

\begin{proof}
 Note that $(G_\ell(0))_\ell$ converges to $0$ since $0\in \cD_A$ and $G(0)=0$. Let $H_\ell = G_\ell-G_\ell(0)$. In particular, $H_\ell$ converges uniformly to $G$ on $\cD_A$ and $H_\ell(0)=0$. Choose a sequence $(t_k)_k$ such that $0<t_k<1$ and $\lim t_k =1$. Note that, for $X\in \cD_A$, 
\[
I+t_k(G(X)+G(X)^*) = (1-t_k)I + t_k (I+G(X)+G(X)^*) \succeq (1-t_k)I.
\]
 For each $k$ there is an $\ell_k$ such that $H_{\ell_k}(X)$ is uniformly sufficiently close to $G$ so that 
\[
 I+t_k (H_{\ell_k}(X)+H_{\ell_k}(X)^*) \succeq I+t_k(G(X)+G(X)^*) -(1-t_k)I \succeq 0.
\]
 Hence, the sequence   $(Q_k= t_k H_{\ell_k})$,  converges uniformly to $G$ on $\cD_A$  and satisfies $Q_k(0)=0$ and $I+Q_k+Q_k^*$ is nonnegative on $\cD_A$.
\end{proof}

\subsection{Proof of Theorem~\ref{thm:analPossIntro}}
\label{sec:proofanalpossIntro}
By Lemma \ref{lem:betterpl} and Theorem \ref{prop:okadron} (using, in particular,
 the boundedness assumption on $\cD_A$), without loss of generality 
there is a sequence $(G_\ell)_\ell$  of polynomials converging uniformly to $G$ on $\cD_A$ 
and   such that $I+G_\ell+G_\ell^*$ is nonnegative on $\cD_A$ and  $G_\ell(0)=0$. 
 By Theorem \ref{thm:heredposSS} (again using the boundedness of $\cD_A$), there is a separable infinite-dimensional Hilbert space $H$ such that for each $\ell$ there exists a polynomial $W_\ell$ with coefficients $W_{\ell,\alpha}:\mathbb C^e\to H\otimes \mathbb C^d$ such that
\begin{equation}
 \label{eq:ratWLW+}
 I+G_\ell(x)+G_\ell(x)^*  = {W_\ell}^*(x) L_{\IHA}(x) {W_\ell(x)}.
\end{equation}
  Applying  Proposition \ref{prop:multi generalIntro}, there exists a contraction $C_\ell$ on $H\otimes \mathbb C^d$
 and an isometry  $\sW_\ell:\mathbb C^e\to H\otimes \mathbb C^d$ such that, with
 $R_\ell = (C_\ell-I)[\IHA]$, 
\[
 W_\ell(x) = (I-\Lambda_{R_\ell}(x))^{-1} \sW_\ell.
\]
Moreover, from the identity $W_{\ell,x_j\alpha}= (C_\ell-I)[\IHA_j] W_{\ell,\alpha}$ of item \eqref{it:recure}  of Proposition \ref{prop:multi generalIntro},
\[
 \| W_{\ell,x_j\alpha}\| \le 2 \max\{\|A_1\|,\dots,\|A_g\|\} \, \|W_{\ell,\alpha}\|.
\]
 Thus, using the fact that $\sW_\ell = W_{\ell,\emptyset}$ is an isometry and hence has norm one, $\|W_{\ell,\alpha}\|$ has a uniform bound depending only on the length of the word $\alpha$ (independent of $\ell$).

Observe that for each $\alpha$, the dimension of the range of $W_{\ell,\alpha}$ is at most $e$. Hence, 
for a fixed $N$, there is a constant $D_N$ such that for each $\ell$ the dimension of the span of
\[
  H_{N,\ell}=\bigvee_{|\alpha|\le N} \rg(W_{\ell,\alpha})
\]
 is at most $D_N$. (Indeed one can take $D_N$ to be $de$ times the number of words of length at most $N$.) It follows that, given $N$, for each $\ell$ there exists a subspace $H_{N,\ell}$ of $H$ of dimension $D_N$ such that the ranges of $W_{\ell,\alpha}$ all lie in $H_{N,\ell}\otimes \mathbb C^d$.  For technical reasons that will soon be apparent,  choose a basis $\{e_1,e_2,\dots\}$ for $H$ and inductively construct subspaces $\mathcal H_{N,\ell}$ of $H$ of dimension $2D_N$ such that $\mathcal H_{N,\ell}$ contains both $H_{N,\ell}$ and $\mbox{span}(\{e_1,\dots,e_{D_N}\})$ and such that $\mathcal H_{N,\ell}\subset \mathcal H_{N+1,\ell}$.  In particular, $H=\oplus_{N=-1}^\infty (\mathcal H_{N+1,\ell}\ominus \mathcal H_{N,\ell})$, where $\mathcal H_{-1,\ell}=\{0\}$.   Set $D_{-1}=0$ and  let $E_m=2(D_m - D_{m-1})$.  Letting $K_m = \mathbb C^{E_m}$ and  $K$ denote the Hilbert space $\oplus_{m=0}^\infty K_m$, it follows that for each $\ell$ there is a unitary mapping $\rho_\ell:H\to K$ such that $\rho_\ell(\mathcal H_{N,\ell})= \oplus_{m=0}^N K_m$.  We have,
\[
 W_\ell(x)^* (\rho_\ell \otimes I_d)^* [I_K \otimes L_A(x)] (\rho_\ell \otimes I_d) W_\ell(x) = W_\ell(x)^* [I_H\otimes L_A(x)] W_\ell(x).
\]
 Hence, we can replace $W_\ell(x)$ with $(\rho_\ell \otimes I) W_\ell(x)$ in  \eqref{eq:ratWLW+} and thus, given a word $\alpha$ of length $N$,  assume that $W_{\ell,\alpha}$ maps into $\oplus_{m=0}^N K_m$ independent of $\ell$.    

For a fixed word $\alpha$, the set $\{W_{\ell,\alpha}:\ell\}$ maps into a common finite-dimensional Hilbert space and is, in norm, uniformly bounded.  Hence, by passing to a subsequence, we can assume for each word $\alpha$ the sequence $W_{\ell,\alpha}$ converges to some $W_\alpha$ in norm. Let $W$ denote the corresponding formal power series. We will argue that
\[
I+G(x)+G(x)^* = W(x)^*L_{\IHA}(x)W(x)
\]
in the sense explained as follows.
 Since $G_\ell(0)=0$, it follows that  $W_{\ell,\emptyset}^* W_{\ell,\emptyset} =I$ for each $\ell$. Hence
\begin{equation}
\label{eq:preisoalt0+}
  W_\emptyset^* W_\emptyset = I.  
\end{equation}
Likewise,  given $\alpha$ and $j$, for every $\ell$,
\[
  W_{\ell,\emptyset}^* (\IHA_j) W_{\ell, \alpha} +W_{\ell, x_j \alpha} = (G_\ell)_{x_j\alpha},
\]
 the coefficient of the $x_j\alpha$ term of $G_\ell$.
 From what has already been proved, the left hand side above converges to $(\IHA_j) W_\alpha +W_{x_j\alpha}.$
  Since $G_\ell$ converges uniformly to $G$ on $\cD_A$, 
 the sequence $((G_\ell)_{x_j\alpha})$ converges to $G_{x_j\alpha}$, the $x_j\alpha$ coefficient of $G$.    Thus,
\begin{equation}
\label{eq:preisoalt1+}
 W_\emptyset^* (\IHA_j) W_\alpha + W_{x_j\alpha} = G_{x_j\alpha}.
\end{equation}

 Moreover, also by construction,  for each $\alpha,\beta$ and $j,k$,
\[
  W_{\ell,\beta}^*(\IHA_k)^* W_{\ell, x_j\alpha} + W_{\ell, x_k\beta}^* (\IHA_j) W_{\ell, \alpha} + W_{\ell, x_k\beta}^* W_{\ell,x_j\alpha} = 0.
\]
 Hence,
\begin{equation}
\label{eq:preisoalt2+}
W_\beta^*(\IHA_k)^* W_{x_j\alpha} + W_{x_k\beta}^* (\IHA_j) W_{\alpha} + W_{x_k\beta}^* W_{x_j\alpha}  = 0.
\end{equation}
Equations \eqref{eq:preisoalt0+}, \eqref{eq:preisoalt1+} and \eqref{eq:preisoalt2+} together show the equations of \eqref{eq:preiso1alt} holds
 in the ring of formal power series.
 Thus, equation \eqref{eq:ratWLW+} holds.  Hence Proposition \ref{prop:multi generalIntro} applies and there exists a 
 contraction $C:H\otimes\C^d\to H\otimes \C^d$ that is isometric on $\rg(A,W)$ such that equations \eqref{eq:Gup} and \eqref{eq:WIntro} 
 hold. 

 To complete the proof, in the notation of Proposition \ref{prop:multi generalIntro}, choose $\mathcal{E}^\prime = H\otimes\C^d$ and make the identification
 $\tcE=\mathcal{E}\oplus\mathcal{E}^\prime = (\C^2\otimes H)\otimes \C^d$.  Likewise, let $A^\prime =A$ and make the identification 
 $\tilde{A} = I_{\C^2\otimes H} \otimes A$.   The moreover portion of Proposition \ref{prop:multi generalIntro} produces a unitary $\tilde{C}$
 and isometry $\tilde{\sW}$ satisfying  equations \eqref{eq:posst+}  and \eqref{eq:Gup}.  
Finally, from the formulas for $G$ and $\tilde{W}$, there
 series have positive radii of convergence say both at least $\tau>0$. Hence equation \eqref{eq:posst+} holds for $X\in \SR_\tau$ by Proposition
\ref{prop:formalveval}.

\section{Consequences of a One Term Positivstellensatz}
\label{sec:square}
In this section, we consider the consequences of a one term square
Positivstellensatz. In particular, a one term Positivstellensatz
produces a \ct map.  Accordingly, suppose $p=(p^1, \ldots, p^g)$ where
each $p^j$ is a free formal power series in $x=(x_1,\dots,x_g)$ such
that $p(0)=0$ and $p^\prime(0)=I$. Further assume $A,B\in M_d(\mathbb
C^g)$ and $W$ is a formal power series 
 with coefficients in $M_d(\C)$ (square matrices) 
 satisfying 
\begin{equation}
 \label{eq:WLAW}
 L_B(p(x)) = W(x)^* L_A(x) W(x)
\end{equation}
in the sense that the relations of equation \eqref{eq:preiso1alt} hold
with $G(x)=\Lambda_B(p(x))$.  Thus the sizes of $A$ and $B$ are the
same and both $L_A$ and $L_B$ are pencils in $g$ variables.  As we
will see, under this assumption (that $W$ is square), equation
\eqref{eq:WLAW} implies $p$ is a \ct map and imposes rigid structure on 
the triple $(p,A,B)$.

Proposition \ref{prop:multi generalIntro} produces $d\times d$ unitary
matrices $C$ and $\sW$ such that, with $R=(C-I)A$,
\begin{equation}
 \label{eq:WG}
\begin{split}
 W(x) & =   (I-\La_R(x))^{-1}\sW\\
 \La_B(p(x)) & =  \sW^* C (\sum_{j=1}^g A_jx_j)W(x).
\end{split}
\end{equation}
Before continuing, we pause to collect some consequences of these relations.

\begin{lemma}
\label{lem:hells kitchen}
 Let $d,e$ and $g\le \tg$ denote positive integers.
 Suppose $p= (p^1, \ldots,  p^\tg)$ and each $p^t$ is a formal power series. Further suppose
 $p(0)=0$ and $p(x)=( x, 0) +h(x)$, where $h$ consists of higher (two and larger) degree terms. Write 
\[
 p^t(x) = \sum_{j} \sum_{\alpha} p^t_{x_j \alpha} x_j\alpha.
\]
If
\ben[\rm(a)]
 \item  $A\in M_d(\C)^g$ and $B\in M_e(\C)^{\tg}$;
 \item  $C$ is a $d\times d$ unitary matrix and $\sW:\C^e\to\C^d$ is an isometry;
 \item with $R=(C-I)A$, $W$ and $\La_B(p(x))$ are as in equation \eqref{eq:WG};
\een
then 
  \ben[\rm (1)]
 \item \label{it:Bjagain}  $B_j = \sW^* C\ptA_j \sW$ for $1\le j\le g$; 
\item \label{it:LambdaBp}
$\displaystyle
 \Lambda_B(p(x)) = \sum_t B_t p^t(x) =  \sum_{k=1}^g \sum_{\alpha} \sW^* C \ptA_k  R^\alpha \sW x_k \alpha 
$;
 \item \label{it:bill}
 for each word $\omega$ and $1\le k\le g$, 
\begin{equation*}
  \sW^*  C\ptA_k R^{\omega} \sW   =  \sum_{j=1}^\tg p^j_{x_k\omega} \; B_j; %
 \end{equation*}
\item \label{it:billtgisg} in the case $\tg=g$, for all words $\omega$ and $1\le k\le g$,
\begin{equation}
 \label{eq:billgtisg}
  \sW^*  C\ptA_k R^{\omega} \sW   =  \sum_{j=1}^g p^j_{x_k\omega} \; B_j = \sum_{j=1}^g p^j_{x_k\omega} \; \sW^* C\ptA_j\sW.
 \end{equation}
\een
\end{lemma}

\begin{proof}
 The result follow by comparing power series expansions terms and using the normalization hypotheses on $p$.
\end{proof}

In the case $e=d$ and $\tg=g$, Lemma \ref{lem:hells kitchen} implies $B=\sW^* CA\sW,$ where
$\sW=W(0)=\Wem$ is unitary. Further, since $\sW$ is unitary, equation
\eqref{eq:billgtisg} of Lemma \ref{lem:hells kitchen} gives $A_k
(C-I)A_j$ is in the span of $A_1,\ldots, A_g$ for all $j,k;$ that is, for
each $1\le j\le g$ there is a $g\times g$ matrix $\xij$ (described
explicitly in terms of the coefficients of $p$) such that for all
$1\le k\le g$,
\begin{equation}\label{eq:AZA}
  A_k (C-I)A_j = \sum_{s=1}^g (\xij)_{k,s} A_s.
\end{equation} 
 The structure inherent in equation \eqref{eq:AZA} is analyzed in the next  subsection.

\subsection{Lurking algebras}
   \label{sec:secretdevils}

\begin{prop}
\label{prop:con}
 If $\Xi=(\Xi_1,\dots,\Xi_g)$ is a \ct tuple, then $\mathcal X$, the span of $\{\Xi_1,\dots, \Xi_g\}$ is an algebra whose structure matrices are the $\Xi_j$; that is,
 for all $1\le k\le g$ and words $\alpha$,
\begin{equation}
\label{eq:con}  
   \Xi_k \Xi^{\alpha} = \sum_s (\Xi^{\alpha})_{k,s}\Xi_s. 
\end{equation}
Moreover,  the associated \ct rational mappings  of equation \eqref{eq:tropic}
$$p(x)= x(I-\Lambda_\Xi(x))^{-1}
\qquad and \qquad q=x(I+\Lambda_\Xi(x))^{-1},$$
     are inverses of one another. 
 \end{prop}

\begin{proof}
 Inducting on the length of $\alpha$ in equation \eqref{eq:con} and using the relation of equation \eqref{eq:cttuple} at the third equality, gives
\[
\begin{split}
 \Xi_k \Xi^{\alpha x_\ell} & =  \Xi_k \Xi^{\alpha} \Xi_\ell\\
  & =  \sum_t (\Xi^{\alpha})_{k,t} \Xi_t \Xi_\ell \\
  & =  \sum_t \sum_s (\Xi^{\alpha})_{k,t} (\Xi_\ell)_{t,s} \Xi_s \\
  & =  \sum_s (\Xi^{\alpha x_\ell})_{k,s} \Xi_s.
\end{split}
\]

To prove the maps of equation \eqref{eq:tropic} are inverses of one another, 
 expand $q^t$ in a series gives to obtain
\[
 q^t(x)=\sum_j x_j \big(I +\La_{\Xi}(x) \big)^{-1}_{j.t} 
=  
\sum_{ j,\alpha\in\ax}  (-1)^{|\alpha|}\;  (\Xi^\alpha)_{j,t}  x_j  \alpha.
\]
Using equation \eqref{eq:con}  at the fourth equality below  obtains,
\[
 \begin{split}
 I-\Lambda_\Xi (q(x))& =  I-\sum_{t=1}^g \Xi_t q^t 
  =  I-\sum_t \Xi_t \sum_{j,\alpha} (-1)^{|\alpha|} (\Xi^\alpha)_{j,t} x_j\alpha \\
 & =  I-\sum_{j,\alpha} (-1)^{|\alpha|} \big(\sum_t  (\Xi^\alpha)_{j,t} \Xi_t\big) x_j \alpha
=  I-\sum_{j,\alpha} (-1)^{|\alpha|}  \Xi_j \Xi^\alpha   x_j \alpha  \\
& =    I-\sum_{j,\alpha} (-1)^{|\alpha|} \Xi^{x_j \alpha} x_j\alpha
 = I +\sum_{|\beta|>0} (-1)^{|\beta|} \Xi^\beta \beta \\
  & =  \sum_{\beta} (-1)^{|\beta|} \Xi^\beta \beta 
  =  (I+\Lambda_\Xi(x))^{-1}.
 \end{split}
\]

  Thus, 
\[
 \begin{split}
   p\circ q(x)& =   q(x) \, [I-\La_{\Xi}(q(x))]^{-1}  
    =  q(x)\big( (I+\Lambda_\Xi(x))^{-1} \big)^{-1}\\
   & =   x \  (I+\Lambda_\Xi(x))^{-1} \, (I+\Lambda_\Xi(x))  =  x
 \end{split}
\]
 and it follows that  $q$ is a right inverse for $p$. By symmetry, it is also a left inverse establishing 
 item  \eqref{it:ratspq}. 
\end{proof}

An algebra $\mathscr A$ has \df{order of nilpotence} $N\in\mathbb N$ if the product of any $N$ elements
of $\mathscr A$ is $0$ and $N$ is the  smallest natural number with this property.
Proposition \ref{prop:AZA} below explains how \ct maps naturally arise from the
algebra-module structure of equation \eqref{eq:AZA}.

\begin{prop}
 \label{lem:gtg}
 Let $R$ and $E$ are $g$-tuples of matrices of the same size $d$.
 Let $\mathscr E=\{E_1,\dots,E_g\}$ and let $\mathscr B$ denote the 
 span of $\mathscr E$. Suppose $\mathscr E$ is linearly independent 
 and there exists (a necessarily uniquely determined tuple)
 $\Xi\in M_g(\C)^g$ such that 
\begin{equation}
\label{eq:defXi}
 E_k R_j = \sum_{s=1}^g (\Xi_j)_{k,s} E_s.
\end{equation}
Then
\begin{enumerate}[(i)]
\item \label{i:gtg1}
 for each $1\le j\le g$ and word $\alpha$, 
\begin{equation}
 \label{eq:EjRa}
 E_k R^\alpha  = \sum_{s=1}^g \Xi^\alpha_{k,s} E_s;
\end{equation}
\item \label{i:gtg2}
if $\Xi$ is nilpotent of order $\nu$, then $\nu\le g$.  Moreover,
if  $R^\alpha=0$, then $\Xi^\alpha=0$ and hence if $R$ is nilpotent of order $\mu$,
 then $\nu\le \mu;$
\item \label{i:gtg3}
  if there is a $d\times d$ matrix $G$ such that
$R=GE$, then $\Xi$ is convexotonic and moreover,
\[
 \Xi_k \Xi_j =  \sum_{s=1}^g (\Xi_j)_{k,s} \Xi_s.
\]
\end{enumerate}

In particular, if $J\in M_d(\C)^g$ is linearly independent and spans an algebra,
then the tuple $\Psi\in M_g(\C)^g$ (uniquely) determined by
\[
 J_k J_j = \sum_{s=1}^g (\Psi_j)_{k,s} J_s
\]
is convexotonic.
\end{prop}

\begin{proof}
 From the hypotheses,
 for each word $\alpha$ and $1\le t\le g$ the matrix  $E_t R^\alpha$ has a unique representation of the form
\begin{equation}
 \label{eq:EtRa}
  E_t R^\alpha  = \sum_{k=1}^g (\Xi_\alpha)_{t,k} E_k,
\end{equation}
 for some $g\times g$ matrix $\Xi_\alpha$. 
 We now argue that $\Xi_\alpha = \Xi^\alpha$ by induction on the length of the word $\alpha$, the case of length $1$
 holding by the choice of $\Xi$.
 Accordingly, suppose $\Xi_{\alpha} = \Xi^{\alpha}$. 
 Applying $R_u$ on the right of equation \eqref{eq:EtRa}  gives
\[
 \begin{split}
   \sum_{s=1}^g (\Xi_{\alpha x_u})_{t,s} E_s& =  E_t R^\alpha R_u 
   =  \sum_{k=1}^g (\Xi_\alpha)_{t,k} E_k R_u 
  =  \sum_{k=1}^g   (\Xi_\alpha)_{t,k} \sum_{s=1}^g  (\Xi_u)_{k,s} E_s 
  =  \sum_{\ell=1}^g  (\Xi_\alpha \Xi_u)_{t,s} E_s. \\
  \end{split}
\]
 By linear independence of $\{E_1,\dots,E_g\}$ and the induction hypothesis,  
\[
 \Xi_{\alpha x_u} = \Xi_\alpha \Xi_u=\Xi^{\alpha} \Xi_u = \Xi^{\alpha \, x_u},
\]
verifying equation \eqref{eq:EjRa} and completing the proof of item \eqref{i:gtg1}.

Now  suppose $\Xi$ is nilpotent of order $\nu$. Thus, if $|\alpha|\ge \nu$, then   $\Xi^\alpha=0$ and therefore  $E_t R^\alpha =0$ for each $t$. 
For positive integers $k$, let $\cE_k$  denote the subspace of $\cE$ spanned
by $\{E_t R^\alpha: |\alpha|\ge k, \, 1\le t \le g\}$. Thus $\cE_{\nu}=\{0\}$ and
since $\cE_k \supset \cE_{k+1}$ and $\cE$ has dimension $g$, it follows that
$\cE_g=\{0\}$.  Hence $E_t R^\alpha =0$ for all $t$ and $|\alpha|\ge g$ and therefore, 
using the independence of $\{E_1,\dots,E_g\}$ once again,
$\Xi^\alpha=0$. Hence $\nu\le g$.  

Conversely, if $R^\alpha=0$, then $\Xi^\alpha=0$. Hence if $R$ is nilpotent of order 
$\mu$, then $\Xi$ is nilpotent of order $\nu\le \mu.$

Now suppose  there is a $d\times d$ matrix $G$ such that
 $R=GE$. To prove that the tuple $\Xi$ is \ct,  fix $1\le k \le g$ and compute
   the product $E_k R_j R_\ell$ in two different ways.
 On the one hand,  using equation \eqref{eq:defXi} twice,
\[
 \begin{split}
 E_kR_jR_\ell  = &(E_k G) E_j R_\ell \\
  &=   \sum_{t=1}^g  (\Xi_\ell)_{j,t} E_k G E_t = \sum_{t=1}^g (\Xi_\ell)_{j,t} E_k R_t\\
  &  =  \sum_{t=1}^g  (\Xi_\ell)_{j,t} \sum_{s=1}^g (\Xi_t)_{k,s} E_s 
   =  \sum_{s=1}^g  \sum_{t=1}^g (\Xi_\ell)_{j,t} (\Xi_t)_{k,s}E_s.
 \end{split}
\]
 On the other hand, using the already established equation \eqref{eq:EjRa} with $\alpha = x_j\, x_\ell$, 
\begin{equation*}
   E_kR_jR_\ell  =  \sum_s (\Xi_j \Xi_\ell)_{k,s} E_s.
\end{equation*}
 For a fixed $k$, the independence of the set $\{E_1,\dots,E_g\}$ implies
\[
  \sum_t (\Xi_\ell)_{j,t} (\Xi_t)_{k,s} =  (\Xi_j \Xi_\ell)_{k,s}
\]
 for each $1\le k,s\le g$  and thus,
\[
  \sum_t  (\Xi_\ell)_{j,t} \Xi_t = \Xi_j \Xi_\ell.
\]
Hence $\Xi$ is a \ct tuple and item \eqref{i:gtg3} is proved.

To prove  the last part of the proposition, choose, in item \eqref{i:gtg3}, 
$E=R=J$ and $G=I$.
\end{proof}

\begin{cor}
 \label{cor:boundcontpoly}
   Suppose $\Xi$ is a convexotonic $g$-tuple with associated convexotonic maps $p$ and $q$ as in equation \eqref{eq:tropic}. If  the tuple
     $\Xi$ is nilpotent, then its order of nilpotency is at most $g$. Further $\Xi$ is nilpotent if and only if $p$ and $q$
 are polynomials. In this case the  order of nilpotence of $\Xi$
    is the same as the degrees of $p$ and $q.$  In particular, the degrees of $p$ and $q$ are at most $g$. Finally, there are examples where
    the degree of $p$ and $q$ are $g$.
\end{cor}

\begin{proof}
 As described in Subsection \ref{sssec:contonics}, there exists a tuple $R=(R_1,\dots,R_g)$ such that $\{R_1,\dots,R_g\}$ is linearly
 independent and spans an algebra with structure matrices $\Xi$.  Hence, choosing $E=R$ in Proposition \ref{lem:gtg}, 
 it follows that if $\Xi$ is nilpotent, then its order of nilpotency is at most $g$. The remainder of the corollary follows
immediately from the form of $p$ and $q$ and the bound on the order of nilpotency of $\Xi$.
\end{proof}

\begin{example}\rm
\label{ex:degp=g}
 Given $g$, let $S$ denote a (square) matrix nilpotent of order $g+1$ and 
 let $R_j=S^j$.  Let $R$ denote the tuple $(R_1,\dots,R_g)$. 
 On $\rC^g$ with its standard orthonormal basis $\{e_1,\dots,e_g\}$, 
  define $S e_j = e_{j-1}$ for $j\ge 2$ and $S e_1=0$. Thus $S$ is 
  the truncated backward shift.  The structure matrices $\Xi_j$  for the algebra generated by $R$ are then
  $\Xi_j = S^j$.  In this case the \ct  polynomial $p$ associated to $\Xi$ is
\[
 p = x (I-\Lambda_\Xi(x))^{-1} = (p^1,\dots,p^g),
\]
 where
\[
 p^m = \sum \prod_{\sum j_k=m} x_{j_k}.
\]
 In particular, $p^m$ has degree $m$ and hence $p$ has degree $g.$ 
\end{example}

\begin{prop}
 \label{prop:AZA}
   Let $A=(A_1,\dots,A_g)\in M_d(\C)^g$ be given and assume
 that $\{A_1,\dots,A_g\}$ is linearly independent.   Suppose $C$ is a $d\times d$ matrix such that,
 for each $1\le j\le g$ there exists a matrix $\xij$ such that for each $1\le k\le g$ equation \eqref{eq:AZA} holds.
  Let  $R= (C-I)A $ and let $\Xi = (\Xi_1,\dots,\Xi_g)$.  Then:
 \ben[\rm (1)]
  \item \label{it:Ralg}
     the span $\mathcal R$ of $\{R_1,\dots,R_g\}$ is an algebra;
  \item \label{it:Amodule}
       the span $\mathcal M$ of $\{A_1,\dots,A_g\}$
   is a right $\mathcal R$-module  and
\begin{equation}
\label{eq:Amodule}
   A_k R^\alpha = \sum_{t} (\Xi^\alpha)_{k,t} A_t;
\end{equation}
  \item \label{it:Xialg} the tuple  $(\Xi_1,\dots,\Xi_g)$ is \ct;
 \item \label{it:ratspq}
    the \ct rational mappings $p$ and $q$ associated to $\Xi$ by equation \eqref{eq:tropic} are inverses of one another; 
   \item \label{it:Rnil}
     if  $R^\alpha =0$, then  $\Xi^\alpha =0$ and conversely, if $\Xi^\alpha =0$, then $R_j R^\alpha =0$ for all $1\le j\le g$; 
   \item \label{it:Rnil-more}
   $\mathcal R$ is nilpotent if and only if $\mathcal X$, the span of $\Xi$, is nilpotent. In this case, letting $\mu$ and $\nu$ denote the orders of nilpotency of $\mathcal R$ and $\mathcal X$ respectively,  $\mu\le \nu \le \mu+1,$ and $\mu \le \min\{\dim(\mathcal R)+1,g\}.$
\een
\end{prop}

\begin{proof}
  From Proposition \ref{lem:gtg}\eqref{i:gtg1}
\begin{equation}
 \label{eq:AkRa}
 A_k R^\alpha = \sum_{j=1}^g \Xi^\alpha_{k,j} A_j.
\end{equation}
 Multiplying \eqref{eq:AkRa}  on the left by $(C-I)$ gives,
\begin{equation*}
 \label{eq:RkRa}
 R_k R^\alpha = \sum_{j=1}^g \Xi^\alpha_{k,j} R_j.
\end{equation*}
 Thus the set $\{R_1,\dots,R_g\}$ spans an algebra  $\mathcal R$  and equation \eqref{eq:AkRa} says the span $\mathcal M$  of the set $\{A_1,\dots,A_g\}$ is a module over $\mathcal R$. At this point items \eqref{it:Ralg} and \eqref{it:Amodule} have been established.

Item \eqref{it:Xialg} follows from Proposition \ref{lem:gtg}\eqref{i:gtg3} by 
choosing $E=A$ and $G=(C-I)$.
Item \eqref{it:ratspq} is contained in Proposition \ref{prop:con}.
Item \eqref{it:Rnil} follows from equation \eqref{eq:AkRa} and
the independence of $\{A_1,\dots,A_g\}$ and, for the converse, equation
\eqref{eq:RkRa}.  Turning to the proof of item \eqref{it:Rnil-more},
 that $\mathcal R$ is nilpotent if and only if
 $\mathcal X$, as well as the inequalities $\mu\le \nu \le \mu+1$
 follow from item \eqref{it:Rnil}.  The inequality $\nu\le g$,
and hence $\mu\le g$ is a consequence of Proposition \ref{lem:gtg}\eqref{i:gtg2}.
The other inequality $\mu\le \dim(\mathcal R)+1$ is automatic.
as is most of item \eqref{it:Rnil-more}. 
\end{proof}

\subsection{The \ct map $p$ and its inverse $q$}
\label{subsubsec:ScottRulz}
The following theorem is the main result of this section. Its proof relies on
Proposition \ref{prop:AZA}.

\begin{thm}
\label{thm:shotinthedark}
Suppose $A,B$ are $g$-tuples of matrices of the same size $d$,
 $\{A_1,\dots,A_g\}$ is linearly independent  and $p=(p^1, \ldots, p^g)$ 
where each $p^j$ is a formal power series,
$p(0)=0$ and $p'(0)=I$. If there exists a $d\times d$ matrix-valued
free formal power series $W$ such that equation \eqref{eq:WLAW} and
the identities of equation \eqref{eq:preiso1alt} with
$G(x)=\Lambda_B(p(x))$ hold, then 
\ben[\rm(1)]\itemsep=5pt
\item \label{it:formW}
 there exists a uniquely determined $d\times d$ unitary matrix $\fU$ and a unitary $d\times d$ matrix $C$
such that, with $R=(C-I)A$, the functions  $G$ and $W$ are given as in 
 equations \eqref{eq:Gup} and \eqref{eq:WIntro} 
 and $B=\fU^* C A \fU$ (meaning $B_j=\fU^* CA_j\fU$ for $j=1,2,\dots,g$);
\item \label{it:Ralg-old}
there is a \ct tuple $\Xi$  satisfying  \eqref{eq:AZA} (equivalently
\eqref{eq:Amodule}). In particular,
the set of matrices  $\{R_j=(C-I)A_j: j=1,\ldots,g\}$ spans an algebra $\cR$;
\item \label{it:formp}
 letting $p$ and $q$ denote the \ct mappings of equation \eqref{eq:tropic} associated to $\Xi$,
 we have $p:\cD_A \to \cD_B$ is bianalytic with inverse $q$;
\item \label{it:Rnil-old} 
  $p$ is a polynomial if and only if the algebra $\mathcal{X}$  spanned by  $\{\Xi_j:1\le j\le g\}$
  is nilpotent and in this case $q$ is also a polynomial and the degrees of $p$ and $q$
  and the order of nilpotence of $\Xi$ are all the same and at most $g$, and there are examples
  where this degree is $g$.
\een

 Conversely, if $A=(A_1,\dots,A_g)$  is a linearly independent tuple of $d\times d$ matrices and $C$ is a $d\times d$
 matrix that is unitary on the span of the ranges of the $A_j$ such that 
  for each $j,k$ the matrix $A_k (C-I)A_j$ is in $\mathcal M=\Span(\{A_1,\dots,A_g\})$, then
  $\mathcal R=\Span(\{R_j = (C-I)A_j: 1\le j\le g\})$ is an algebra and $\mathcal M$ 
 is a right module over the algebra $\mathcal R$.  Let $\Xi=(\Xi_1,\dots,\Xi_g)$ denote the structure matrices 
  for the module $\mathcal M$ over the algebra $\mathcal R$
  The tuple $\Xi$ is convextonic.  Given a unitary $\fU$ and letting  $B=\fU^* CA\fU$, the function
 $W$ given by equation \eqref{eq:WIntro} and  
 the convexotonic mapping $p(x)= x(I-\Lambda_\Xi(x))^{-1}$   together
  satisfies the identities of equations \eqref{eq:WLAW} and  \eqref{eq:preiso1alt} and 
 hence item \eqref{it:formp} holds. %
\end{thm}

Before turning the proof of Theorem \ref{thm:shotinthedark}, we indicate  
how to use it to deduce Theorem \ref{thm:ctok}.

\begin{proof}[Proof of Theorem~\ref{thm:ctok}]
 Suppose $A\in M_d(\C)^g$, $C\in M_d(\C)$ is a unitary matrix,
 $R=(C-I)A$ is linearly independent and $\Xi\in M_g(\C)^g$
 is a tuple such that  equation \eqref{eq:Astructure}
 holds; that is $(\cD_A,\cD_{CA})$ is a spectrahedral pair. 
 An application of the converse of Theorem \ref{thm:shotinthedark}
 with $\fU=I_d$, says the convexotonic map $p=x(I-\Lambda_\Xi(x))^{-1}$
 satisfies item \eqref{it:formp} of Theorem \ref{thm:shotinthedark}.
\end{proof}

\begin{proof}[Proof of Theorem~\ref{thm:shotinthedark}]
   To prove item \eqref{it:formW} apply Proposition \ref{prop:multi generalIntro} to equation \eqref{eq:WLAW}
 and use the finite dimensionality of $\mathcal{E}=\C^d$ in the present context  to obtain a $d\times d$ unitary
   matrix $C$ such that
  $W$ and $G$ have the form given in equations \eqref{eq:WIntro} and \eqref{eq:Gup} %
   with $\fU=\Wem$.  Since $\fU$ is an isometric mapping from $\C^d$ to $\C^d$ it is unitary. 
 Thus, by Lemma \ref{lem:hells kitchen},   equation \eqref{eq:AZA} holds. (See the discussion surrounding equation \eqref{eq:AZA}.)
 Hence, by Proposition \ref{prop:AZA}, item  \eqref{it:Ralg-old}  holds.
 In fact, since $C$ and $\sW$ are unitary,  from equation \eqref{eq:billgtisg} of Lemma \ref{lem:hells kitchen},
\begin{equation}
 \label{eq:billoncemore}
  A_k R^{\omega}  =  \sum_{j=1}^g   p^j_{x_k\omega} \; A_j.  
\end{equation}
 Comparing  equations \eqref{eq:Amodule} (or \eqref{eq:AZA}) and \eqref{eq:billoncemore} shows  $p^j_{x_k \omega} = (\Xi_\omega)_{k,j}$ and hence
\[
 \begin{split}
 x(I-\Lambda_\Xi(x))^{-1} 
 & =  \begin{pmatrix} \sum_{j,\alpha}  (\Xi^\alpha)_{1,j} x_j \alpha & \dots &  \sum_{j,\alpha}  (\Xi^\alpha)_{g,j} x_j \alpha \end{pmatrix}\\
 & =  \begin{pmatrix} \sum_{j,\alpha} p^j_{x_1 \alpha} x_j \alpha & \dots &  \sum_{j,\alpha} p^j_{x_g \alpha} x_j\alpha \end{pmatrix} =  p(x).
\end{split}
\]
Thus  item \eqref{it:formp} holds. %

   The converse statements of the theorem are established by verifying that, with 
  the choices of $A,B,\fU,C$ and $W$ and finally $p$,  equation \eqref{eq:WLAW}
  holds.
\end{proof}

\subsection{Proper analytic mappings}
 In this section we apply Theorem \ref{thm:shotinthedark} to the case of a  mapping $p= (p^1, \ldots, p^\tg)$ in $g$ ($g<\tg$) variables $x$ where each $p^j$ is a formal power series.  

\begin{prop}
 \label{prop:gvtg}
 Suppose
\ben[\rm (a)]
  \item  $A\in \matdg$ and $B\in \matdtg$; 
  \item  $p(0)=0$; 
  \item  $p^\prime(0)= ( I \;\; 0);$ and
  \item  the set $\{B_1,\dots,B_\tg\}$ is linearly independent.
\een
  If there exists a matrix-valued formal power series $W$ with coefficients from $M_d(\mathbb C)$ such that
\[
 L_B(p(x)) = W(x)^* L_A(x) W(x),
\]
and the identities of equation \eqref{eq:preiso1alt} with $G(x)=\Lambda_B(p(x))$ hold, then there exists a $\hatg$ and a \ct  $\hatg$-tuple $\Xi$ of $\hatg\times \hatg$ matrices such that $P(x,0_{\tau})=(p(x),0_{\sigma})$, where $P(x,y)$ is the \ct  rational function in the variables $(x_1,\dots,x_g,y_1,\dots, y_\tau)$ (and where $\tau = \hatg-g$ and $\sigma=\hatg-\tg$) associated to $\Xi$,
\[
 P(x,y) = \begin{pmatrix} x & y \end{pmatrix} (I-\Lambda_\Xi(x,y))^{-1}.
\]
\end{prop}

\begin{proof}
The strategy is to reduce to the case $\tg=g$.
 From $L_B(p(x)) = W(x)^* L_A(x)W(x)$ and Proposition \ref{prop:multi generalIntro}, it follows that there exists $d\times d$ unitary matrices $C$ and $\sW$ such that, with $R=(C-I)A$, the formal power series $W$ is the rational function $W(x) = (I-\Lambda_R(x))^{-1} \sW$.  Further, by Lemma \ref{lem:hells kitchen}, for $1\le j\le g$,
\[
 B_j = \sW^* C A_j \sW,
\]
  and generally $\sW^* CA_jR^\omega \sW$ is a linear combination of $\{B_1,\dots,B_{\tg}\}$.  Thus,
\[
 A R^\omega  \in \mathfrak{B},
\]
 where $\mathfrak{B}$ denotes the span of $\{C^* \sW B_1\sW^* ,\dots,C^* \sW B_{\tg}\sW^* \}$.  (In particular, $C^*\sW B_j \sW^*=A_j$ for $1\le j\le g$.)
Let $\mathcal E =\{E_1,\dots,E_\hatg\}$ be any linearly independent  subset of $M_d(\C)$ such that  $E_j=C^* \sW B_j\sW^*$ for $1\le j\le \tg$ and 
\[
 E_k (C-I)E_j \in {\rm span}\ \mathcal{E}.
\]
 In particular $\hatg\ge\tg$.  
Let $F= \sW^* C E \sW$, set  $S=(C-I)E$  and let $Y(x,y)= (I-\Lambda_S(x,y))^{-1}\sW$.  By the converse portion of Proposition \ref{prop:multi generalIntro}, 
\[
 Y(x,y)^* L_E(x,y) Y(x,y) = L_F(P(x,y)),
\]
 for some power series $P$. Indeed, by Theorem \ref{thm:shotinthedark}, $P$ is the \ct rational function associated to $\Xi$
(and is a bianalytic map between the free
spectrahedra determined by $E$ and $F$).   Observe that $Y(x,0)=W(x)$, $L_E(x,0)=L_A(x)$.  Hence 
\[
 L_F(P(x,0))=  Y(x,0)^* L_E(x,0)Y(x,0) = W(x)^* L_A(x) W(x) = L_B(p(x)).
\]
Since $F_j=B_j$ for $1\le j\le \tg$, the linear independence assumption implies $P(x,0)=p(x)$. 
\end{proof}

\section{Bianalytic Maps} \label{sec:redo} 
Suppose $A\in\matdg$ and $B\in \mateg$, the domains $\cD_A$ and $\cD_B$ are bounded,
 $p:\cD_A\to\cD_B$ is an analytic mapping such that $p(0)=0$, $p^\prime(0)=I$ 
and $p$ maps the boundary of $\cD_A$ into the boundary of $\cD_B$.
Equivalently, $p$ is proper and thus  bianalytic \cite{HKM11b}. 
In this section we will see, up to mild assumptions on $A$ and $B$, 
that $d=e$ and the hypothesis of Theorem \ref{thm:shotinthedark} are met and hence $p$ is \ct.

\subsection{An irreducibility condition}
In this subsection we introduce irreducibility conditions on tuples $A$ and $B$ that ultimately allow the application of Theorem \ref{thm:shotinthedark}.

\subsubsection{Singular vectors}

The following is an elementary linear algebra fact.

\begin{lemma}
 \label{lem:elementary largest}
   Suppose $T$ is an $M\times N$ matrix of norm one and let $\mathscr E$ and $\mathscr E_*$ denote the eigenspaces corresponding to the (largest) eigenvalue $1$ of $T^*T$ and $TT^*$ respectively. Thus, for instance,
\[
 \mathscr E = \{x\in\rC^N: T^*Tx=x\}.
\]
\ben[\rm (1)]
 \item The dimensions of $\mathscr E$ and $\mathscr E_*$ are the same.
 \item The mapping $x\mapsto Tx$ is a unitary map from $\mathscr E$ to $\mathscr E_*$ with inverse $y\mapsto T^* y$.
 \item Letting
\[
 J= \begin{pmatrix} I & T\\ T^* & I \end{pmatrix},
\]
 the kernel of $J$ is the set $\{-Tu\oplus u: u\in \mathscr E\}.$
\een
\end{lemma}

\begin{proof}
  Simply note, if $T^*Tx=x$, then $TT^*(Tx)=Tx$ and conversely if $TT^*y=y$, then $T^*T(T^*y)=T^*y$ to prove the first two items. To prove the last item, observe that vectors of the form $-Tu\oplus u$ are in the kernel of $J$. On the other hand, if $v\oplus w$ is in the kernel of $J$, then $v+Tw=0$ and $T^*v +w=0$. From the first equation $T^*v+T^*Tw=0$ and from the second $T^*Tw=w$. Thus $w\in \mathscr E$ and $v\oplus w = -Tw\oplus w$.
\end{proof}

\def\laambda{\lambda}
\def\gaamma{\eta}

\begin{lemma}
 \label{lem:even better}
  Suppose $d,e,g\le \tg$ are positive integers and 
\ben[\rm (a)] 
 \item $A\in \matdg$, $B\in \matetg$; 
 \item $H$ is a Hilbert space;
 \item $C$ is a bounded linear operator on $H\otimes \rC^d$;
 \item $\sW:\C^e\to H\otimes \C^d$  is an isometry; 
 \item $p=(p^1,\ldots,p^\tg)$ is a free analytic mapping 
$\cD_A\to\cD_B$
 with $p(0)=0$ and linear term $\ell$  such that
 \begin{equation}
 \label{eq:even better}
 L_B(p(x)) = W(x)^* L_{\IHA}(x) W(x),
\end{equation}
 where 
\begin{equation}
 \label{eq:even betterer}
 W(x) = (I-\Lambda_R(x))^{-1} \sW
\end{equation}
 and $R=(C-I)A$; and 
\item $\alpha\in \matng$ and the largest eigenvalue of
  $\Lambda_A(\alpha) \Lambda_A(\alpha)^*$ and $\Lambda_A(\alpha)^*
  \Lambda_A(\alpha)$ is $1$; the eigenspaces of $\Lambda_A(\alpha)
  \Lambda_A(\alpha)^*$ and $\Lambda_A(\alpha)^* \Lambda_A(\alpha)$
  corresponding to the eigenvalue $1$ are one-dimensional, spanned by
  the unit vectors $u_1,u_2$ in $\rC^{nd}$ respectively; and
 \item  $v_1\in\rC^{ne}$ is a unit vector and  $\Lambda_B(\ell(\alpha))\Lambda_B(\ell(\alpha))^* v_1= v_1$.
\een
 Let $v_2= -\Lambda_B(\ell(\alpha))^* v_1$ (note that $v_2$ is a unit vector) and write, for $j=1,2$ and $\{e_1,\dots,e_n\}$ a basis for $\rC^n$,  
\[
 u_j = \sum_{k=1}^n u_{j,k} \otimes e_k \in \rC^d\otimes \rC^n =\rC^{nd}
\]
 and similarly for $v_j.$ Then  there is a  unit  vector $\laambda\in H$ (depending on $\alpha$, $u_j$ and $v_j$) such that, 
\ben[\rm (1)]
 \item $\Lambda_A(\alpha) u_2 = -u_1$ and $\Lambda_A(\alpha)^* u_1 =-u_2;$ 
 \item $\Lambda_B(\ell(\alpha))v_2=-v_1$ and $\Lambda_B(\ell(\alpha))^* v_1=-v_2$; 
 \item  $\sW v_{2,k}= \laambda \otimes u_{2,k}$ for each $1\le k\le n$; 
 \item  $\sW v_{1,k}=C(\laambda\otimes u_{1,k})$ for each $1\le k\le n$; and
 \item if $A=B$ and $\ell(x)=x$,   then, without loss of generality, $v_1=u_1$ and $v_2=u_2$.
\een
\end{lemma}

Note that if $X\in M(\C)^g$  is of sufficiently small norm
or nilpotent, then we may substitute $X$ for $x$ in equation \eqref{eq:even better} by 
using the formulas for $G$ and $W$ in Proposition \ref{prop:multi generalIntro}.
Moreover, in this case we can evaluate $W(x)$ from \eqref{eq:even betterer} at $X$ as 
\[
\begin{split}
W(X)& =  \big(I\otimes I_n - [(C-I)\otimes I_n] \Lambda_{\IHA}(X)\big)^{-1} (\sW\otimes I_n) \\
   & =  \big(I\otimes I_n - [(C-I)\otimes I_n][I_H\otimes \Lambda_{A}(X)]\big)^{-1} (\sW\otimes I_n)
\end{split}
\]
 rather than appealing to convergence of a series expansion for $W$.

\begin{proof}[Proof of Lemma \ref{lem:even better}]
Let
\[
 S=\begin{pmatrix} 0 & 1 \\ 0 & 0 \end{pmatrix}
\]
  and let $X=S\otimes \alpha$. Thus $X$ has size $2n$.   
Conjugating $L_A(X)$ by the permutation matrix that implements the unitary equivalence of $A\otimes S\otimes \alpha$ with $S\otimes A\otimes \alpha$ shows, up to this unitary equivalence, 
\[
 L_A(X) = \begin{pmatrix} I & \Lambda_A(\alpha) \\ \Lambda_A(\alpha)^* & I \end{pmatrix}.
\]
 Thus the assumptions on $\Lambda_A(\alpha)$ and Lemma \ref{lem:elementary largest} imply that $L_A(X)$  is positive semidefinite with a nontrivial kernel 
 spanned by 
\[
 u  = \sum_{j=1}^2 e_j\otimes u_j = \begin{pmatrix} u_1 \\ u_2 \end{pmatrix} \in \rC^2 \otimes (\rC^n\otimes \rC^d).
\]
In particular, if $z$ is in the kernel of $I_H\otimes L_A(X)$, then there is a vector $\laambda\in H$ such that $z=\laambda\otimes u$.    Also note, $\|u_1\|=\|u_2\|$ and we assume both are unit vectors. 

Since, by assumption $p(x)=\ell(x) +h(x)$, where $\ell$ is linear and $h$ consists of higher order terms and $X$ is (jointly) nilpotent (of order $2$), 
\[
 p(X) = \begin{pmatrix} 0 & \ell(\alpha) \\ 0  & 0 \end{pmatrix}.
\]
Thus,
\[
 L_B(p(X)) = \begin{pmatrix} I & \Lambda_B(\ell(\alpha)) \\ \Lambda_B(\ell(\alpha))^* & I \end{pmatrix}.
\]
 Since $L_A(X)$ is positive semidefinite, equation \eqref{eq:even better} implies $L_B(p(X))$ is positive semidefinite. Moreover, 
 the hypotheses imply that the vector $v= \sum_{j=1}^2 e_j\otimes v_j$ satisfies $L_B(p(X))v=0$. 
 Another application of equation \eqref{eq:even better} shows $W(X)v$ is in the kernel of $L_{\IHA}(X)$; i.e., $W(X)v=\laambda\otimes u,$ for some vector $\laambda\in H$ with $\|\laambda\|=\|W(X)v\|$.  Hence, 
\[
 \begin{split}
\laambda\otimes (\sum_{j=1}^2 u_j\otimes e_j)
    & =  \laambda\otimes u 
     = W(X)v 
     = (I-\sum_{i=1}^g R_i \otimes X_i)^{-1} (\sW\otimes I_n \otimes I_2) v. 
 \end{split}
\]
Multiplying  by $(I-\sum_{i=1}^g R_i\otimes X_i)$ on the left yields
\begin{equation*}
 \begin{split}
   \sum_{j=1}^2 ([\sW\otimes I_n] v_j) \otimes e_j  & =  [I-[(C-I)\otimes I_{2n}](I\otimes \Lambda_{\IHA}(X))](\laambda\otimes u) \\
   & =  \big (\laambda\otimes u_1 -[(C-I)\otimes I_n] (\laambda\otimes \Lambda_A(\alpha) u_2)\big)\otimes e_1 +  \laambda\otimes u_2 \otimes e_2.
 \end{split} 
\end{equation*}
 It follows that $(\sW\otimes I_n) v_2 =\laambda\otimes u_2$ and, since $\sW$ is an isometry, $\|\laambda\|=\|v_2\|$. Further, 
\[
 (\sW\otimes I_n)  v_1 = \laambda\otimes u_1 -[(C-I)\otimes I_n] \big( \laambda \otimes \Lambda_A(\alpha) u_2 \big).
\]
 Using $\Lambda_A(\alpha) u_2 = -u_1$ gives $[\sW\otimes I_n] v_1 = [C\otimes I_n] (\laambda\otimes u_1).$

To complete the proof observe that 
\[
  (\sW\otimes I_n)v_2 =  \sum_{k=1}^n \sW v_{2,k} \otimes e_k.
\] 
 Thus, $\sW v_{2,k} = \laambda\otimes u_{2,k}.$ Similarly,
\[
 (C\otimes I_n) (\laambda\otimes u_1) = (C\otimes I_n) (\sum_{k=1}^n \laambda \otimes u_{1,k}\otimes e_k) 
   = \sum_{k=1}^n C(\laambda\otimes u_{1,k})\otimes e_k.
\]
 Thus, $\sW v_{1,k} = C(\laambda\otimes u_{1,k})$ for each $1\le k\le n$.
\end{proof}

\subsubsection{The Eig-generic condition}\label{sssec:eig}
We now introduce some refinements of the notion of sv-generic
we saw in the introduction.
 A subset $\{b_1,\dots,b_{\ell+1}\}$ of a finite-dimensional vector space $V$ is a \df{hyperbasis} if each subset of $\ell$ vectors is a basis.  In particular, if $\{b_1,\dots,b_\ell\}$ is a basis for $V$ and $b_{\ell+1} = \sum_{j=1}^\tg c_j b_j$ and $c_j\ne 0$ for each $j$, then $\{b_1,\dots,b_{\ell+1}\}$ is a hyperbasis and conversely each hyperbasis has this form.
  Given a tuple $A\in\matdg$, let 
  \[\ker(A) = \bigcap_{j=1}^g \ker(A_j). \index{$\ker(A)$}
  \]   
Given a positive integer $m$, let $\{e_j: 1\le j\le m\}$ denote the standard basis for $\rC^m$.

\begin{definition}\rm
 \label{def:generic-weak}
  The tuple $A\in \matdg$  is \df{weakly eig-generic} if there exists an $\ell\le d+1$ and, for $1\le j\le \ell$,   positive integers $n_j$  and tuples $\alpha^j\in \matnjg$ such that
\ben[\rm (a)]
  \item \label{it:oneD}
   for each $1\le j\le \ell$, the eigenspace corresponding to the
   largest eigenvalue of $\Lambda_A(\alpha^j)^*\Lambda_A(\alpha^j)$ has dimension one and hence is spanned by a vector  $u^j = \sum_{a=1}^{n_j}  u^j_a\otimes e_a$; and 
 \item \label{it:span} the set  $\mathscr U =\{u^j_a: 1\le j\le \ell, \, 1\le a \le n_j\}$ contains a hyperbasis for  $\ker(A)^\perp =\rg(A^*)$. 
\een
 The tuple is \df{eig-generic} if it is weakly eig-generic and   $\ker(A)=(0)$. Equivalently, $\rg(A^*)=\C^d$.

  Finally, a tuple $A$ is \df{$*$-generic} (resp.~\df{weakly $*$-generic}) if there exists an $\ell\le d$ and tuples $\beta^j$ such that
the kernels of
   $I-\Lambda_A(\beta^j)\Lambda_A(\beta^j)^*$ have dimension one and are spanned by vectors $\mu^j = \sum \mu^j_a \otimes e_a$ for which the set $\{\mu^j_a: j, a\}$ spans $\rC^d$ (resp.~$\rg(A)=\ker(A^*)^\perp$). 
\end{definition}

\begin{remark}\rm
 \label{rem:special-eigs}
 It is illustrative to consider two special cases of the weak
 eig-generic condition.  First suppose $n_j=1$ for all $1\le j\le
 \ell$.  The
 kernel of $I-\Lambda_A(\alpha^j)^* \Lambda_A(\alpha)$ is spanned by a
 a single (non-zero) vector $u^j\in\rC^d$ and the set $\{u^1,\dots,
 u^{\ell}\}$ is a hyperbasis for $\ker(A)^\perp$. Hence $\ell-1$ is the dimension of $\ker(A)^\perp$.
  If we also assume
 $\ker(A)^\perp=(0)$ and there exists $\beta^j$ for $j=1,\dots,n$ such
 that $I-\Lambda_A(\beta^j)\Lambda_A(\beta^j)^*$ is positive definite
 with one-dimensional kernel spanned by $v^j$ and moreover
 $\{v^1,\dots,v^d\}$ is a basis for $\C^d$, then $A$ is sv-generic as
 defined in the introduction.

 For the second case, suppose, for simplicity, that $\ker(A)=(0)$. If
 there exists an $\alpha^1\in \matng$ such that
 $I-\Lambda_A(\alpha^1)^* \Lambda_A(\alpha^1)$ is positive
 semidefinite with a one-dimensional kernel spanned by
\[
 u^1 = \sum_{k=1}^n u_k^1 \otimes e_k \in \rC^d \otimes \rC^n
\]
and if the set $\{u^1_k: 1\le k\le n\}$ spans $\rC^d,$ then $A$ is
eig-generic. To prove this statement, 
suppose, without loss of generality, $\{u^1_k:1\le k\le
g\}$ is a basis for $\rC^d$.  Now take a unitary 
matrix $T$ such that $T_{k,1}\ne 0$ for each $1\le k\le d$ and
$T_{k,1}=0$ for each $d+1\le k\le n$. Let $\alpha^2=T\alpha^1 T^*$.
It follows that $I-\Lambda_A(\alpha^2)^*\Lambda_A(\alpha^2)$ is 
positive semidefinite with a one-dimensional kernel spanned by the vector
$u^2= (I_d \otimes T) u^1$ and further
\[
 \begin{split}
 u^2 =& \sum_{k=1}^n  u^1_k \otimes T e_k 
  =  \sum_{j=1}^n \sum_{k=1}^n u_k^1 \otimes T_{k,j}e_j   
= \sum_{j=1}^n \left(\sum_{k=1}^n T_{k,j}u_k^1 \right) \otimes e_j.
\end{split}
\]
Thus, in view of the assumptions on $T$,
\[
 u^2_1 = \sum_{k=1}^n T_{k,1}u^1_k = \sum_{k=1}^d T_{k,1} u^1_k.
\]
Since $T_{k,1}\ne 0$ for $1\le k\le d$, the set $\{u^1_1,\dots, u^1_g,u^2_1\}$ is a hyperbasis for $\C^d$
 and the tuple $A$ is eig-generic.
\end{remark}

\begin{remark}\label{rem:sv=gen}\rm
Let us explain why sv-genericity is a generic property in the standard algebraic geometric sense. 
First notice that for a generic
tuple 
$A\in M_d(\C)^g$,
the real-valued polynomial 
\[
p(\alpha)=\det\big(I-\Lambda_A(\alpha)^*\Lambda_A(\alpha)\big)
=
\det \begin{pmatrix} I & \Lambda_A(\alpha)\\
\Lambda_A(\alpha)^* & I \end{pmatrix}
\]
 is irreducible
and changes sign on $\R^{2g}$;
here we consider $p(\alpha)$ as a real polynomial
in the real and imaginary parts of the
complex variables $\alpha\in\C^g$.
This fact is easily established by simply giving a tuple $A$ with this property. 
As a consequence, \cite[Theorem 4.5.1]{BCR98} implies that 
each polynomial vanishing on the zero set of $p$
must be a multiple of $p$.

If $A$
 is not sv-generic, it fails one of the two properties in its definition. 
(It suffices to show this while omitting the positive semidefiniteness condition.)
Assume $A$ fails the first property. Then for every choice of $d+1$ vectors $\alpha^j\in\C^d$ for
which $I-\Lambda_A(\alpha^j)^*\Lambda_A(\alpha^j)$ is singular with a one-dimensional kernel spanned by $u^j$, the set $\{u^1,\ldots,u^{d+1}\}$ is not a hyperbasis. Observe that in this case
$u^j$ can be chosen to be a column of the
adjugate matrix of $I-\Lambda_A(\alpha^j)^*\Lambda_A(\alpha^j)$.

The latter condition can be expressed by saying that one of the
$d\times d$ minors
of the matrix $\begin{pmatrix} u^1 & \cdots & u^{d+1}\end{pmatrix}$,
whose columns $u^j$ are columns of 
the adjugate of 
$I-\Lambda_A(\alpha^j)^*\Lambda_A(\alpha^j)$,
vanishes. Equivalently, the product $q$ of all these $d\times d$ minors  vanishes on the zero set of $p$. 

But, as follows from the first paragraph, on a generic set of $A$s,
this means that $q$ is a multiple of $p$.
However, it is easy to find examples of $A$ for
which this fails. 
The argument is similar if $A$
 fails the second property of the definition of sv-generic. Hence being sv-generic is a generic property.
 \end{remark}

   \begin{remark}\label{rem:generic+}\rm
The one-dimensional kernel assumption is key for the eig-generic property,
and has been successfully analyzed in the two papers \cite{KSV,KV}.
Namely, if the tuple $A\in M_d(\C)^g$ is minimal w.r.t.~the size
   needed to describe the free spectrahedron $\cD_A$, then 
$\dim\ker L_A(X)=1$ for all $X$ in an open and dense subset of the
boundary $\partial\cD_A(n)$ provided $n$ is large enough.
\end{remark}

\subsection{The structure of bianalytic maps}
In this section our main results on bianalytic maps between free
spectrahedra appear as Theorem \ref{thm:one-sided} and Corollary
\ref{cor:main}. We begin by collecting consequences of the eig-generic
assumptions.

\def\cI{\mathcal I} 

\begin{lemma}
 \label{lem:eig-generic in action}
   Suppose 
\ben[\rm (a)] 
 \item $A\in \matdg$, $B\in \matetg$; 
 \item $H$ is a Hilbert space,  $C$ is an isometry on $H\otimes \rC^d$ and
\begin{equation*} 
 W(x) = (I-\Lambda_R(x))^{-1} \sW,
\end{equation*}
 where   $R=(C-I)[\IHA]$
and $\sW:\C^e\to H\otimes\C^d$ is an isometry;
 \item \label{it:nils}
   $p=(p^1,\ldots,p^\tg)$ is a free analytic mapping
     $\cD_A\to\cD_B$    such that $p(0)=0,$ and 
 \begin{equation*}
      L_B(p(x)) = W(x)^* L_{\IHA}(x) W(x),
\end{equation*}
 in the sense that 
\[
 L_B(p(X)) = W(X)^* L_{\IHA}(X)W(X)
\]
 for each nilpotent $X\in M(\C)^g$;
\item \label{it:pproper}
  $p$ maps the boundary of $\cD_A$ into the boundary of $\cD_B$; and 
\item \label{it:alphas}
 there is a positive integer $\ell$ and, for $1\le j\le \ell,$ tuples $\alpha^j$ in $\matnjg$
  such that $I-\Lambda_A(\alpha^j)^*\Lambda_A(\alpha^j)$ is positive
  definite with a one-dimensional kernel spanned by
\[
 u_2^j = \sum_{k=1}^{n_j} u^j_{2,k} \otimes e_k;
\]
\een

\begin{enumerate}[\rm (1)]
 \item \label{it:action dlee} If  $A$ is eig-generic {\rm \big (}resp.~weakly eig-generic{\rm \big )}, then $d\le \dim(\rg(B^*))\le e$  {\rm \big (}resp.~$\dim(\rg(A^*) \le  \dim(\rg(B^*)${\rm \big )};
 \item \label{it:action d=e} If $e=d$ {\rm \big (}resp.~$\dim(\rg(A^*)) = \dim(\rg(B^*))${\rm \big )} and the tuples $\alpha^j$ and unit vectors $u_2^j$ validate the 
  eig-generic {\rm \big (}resp.~weak eig-generic{\rm \big )} assumption for $A$, then there exists a unit vector $\laambda\in H$ and unit  vectors
\[
 v_2^j =\sum_{k=1}^{n_j}  v^j_{2,k} \otimes e_j
\]
 in the kernel of $I-\Lambda_B(p(\alpha^j))^* \Lambda_B(p(\alpha^j))$ such that if  $\cI\subset \{(j,k): 1\le j\le l, \, 1\le k \le n_j\}$ 
and $\{u^j_{2,k}: (j,k)\in\cI\}$ is a hyperbasis for $\C^d$ {\rm \big (}resp. $\rg(A^*)${\rm \big )},
 then  $\{v^j_{2,k}:(j,k)\in\cI\}$ is a hyperbasis for $\C^d$ {\rm \big (}resp. $\rg(B^*)${\rm \big )}, and  for all
 $(j,k)\in\cI,$ %
\[
 \sW v^j_{2,k} = \laambda\otimes u^j_{2,k}; 
\]
\item \label{it:action d=e alt} If $e=d$ and $A$ is eig-generic {\rm \big (}resp.~$\dim(\rg(A^*)) = \dim(\rg(B^*)$ and $A$ is weakly eig-generic{\rm \big )}, then there exists a unit vector $\laambda\in H$ and a  $d\times d$ unitary $M$ {\rm \big (}resp.~a unitary map $M$ from  $\rg(B^*)$ to $\rg(A^*)$) such that $\sW=\laambda\otimes M$ {\rm \big (}resp.~$\sW v= \laambda \otimes Mv$ for $v \in \rg(B^*)${\rm \big )}; and 
\item \label{it:action one-term} If  $A$ is eig-generic and $*$-generic  and $e=d$ {\rm \big (}resp.~$A$ is weakly eig-generic and weakly $*$-generic,  $\dim(\rg(A^*)) = \dim(\rg(B^*))$ and $\dim(\rg(A)) = \dim(\rg(B))${\rm \big )}, then there is a vector $\laambda\in H$ and $d\times d$ unitary matrices $M$ and $Z$ such that $\sW= \laambda\otimes M$ and $C(\laambda\otimes I_d) = \laambda\otimes Z$ {\rm \big (}resp.~a unitary map $M$ and  an isometry $N$ from $\rg(B^*)$ to $\rg(A^*)$ and from $\rg(B^*)\cap \rg(B)$ into $\rg(A)$ respectively such that $\sW v= \laambda \otimes Mv$ for $v\in \rg(B^*)$ and $C (\laambda\otimes Nv) = \laambda \otimes Mv$ for $v\in \rg(B^*)\cap \rg(B)${\rm \big )}.  
\end{enumerate}
\end{lemma}

\begin{remark}\rm
Note that the hypotheses on $\alpha^j$ and $u_2^j$ imply that each $u^j_{2,k}\in \rg(A^*)$. Likewise $v^j_{2,k}\in \rg(B^*)$. The  eig-generic hypothesis in item \eqref{it:action dlee} can be relaxed to {\it  $\{u^j_{2,k}:j,k\}$ spans $\C^d$ (respectively $\rg(A^*)$)}, rather than that {\it it contains a hyperbasis.}
\end{remark}

\begin{proof}
  We begin with some calculations preliminary to proving all items
  claimed in the lemma. Let $\alpha^j$ and $u^j_2$ be as described in item \eqref{it:alphas} (but do not necessarily assume that $\{u^j_2:j\}$ contains a hyperbasis yet).
  Let $\cJ=\{(j,k): 1\le j\le \ell,\, 1\le k \le n_j\}$ and, as in the proof 
of Lemma \ref{lem:even better}, let
\[
 S= \begin{pmatrix} 0 & 1 \\ 0 & 0 \end{pmatrix}.
\]
For $1\le j\le \ell$, let $X^j = S\otimes \alpha^j$.  The hypotheses
imply $X^j$ is in the boundary of $\cD_A$. By item \eqref{it:pproper}, $p(X)$ is in the
boundary of $\cD_B$. Observe that $p(X)=\ell(X)$, where $\ell$ is the
linear part of $p$. Thus (up to unitary equivalence)
\[
 L_B(p(X)) = \begin{pmatrix} I & \Lambda_B(\ell(\alpha^j)) 
    \\ \Lambda_B(\ell(\alpha^j))^* & I \end{pmatrix}
\]
is positive semidefinite and there exist unit vectors
$v^j_i\in\rC^{ne}$ such that $v^j = e_1\otimes v_1^j + e_2\otimes
v_2^j$ lies in the kernel of $L_B(p(X))$. Hence, by Lemma
\ref{lem:elementary largest},
$\Lambda_B(\ell(\alpha^j))\Lambda_B(\ell(\alpha^j))^* v_1^j=v_1^j$ and
$v_2^j=-\Lambda_B(\alpha^j)^* v_1^j$.  Consequently, Lemma
\ref{lem:even better} applies. In particular, writing
\[
 v_i^j = \sum_{k=1}^{n_j} v^j_{i,k} \otimes e_k,
\]
 there exist unit vectors $\laambda_j \in H$ such that
\begin{equation*}
 \sW v^j_{2,k} = \laambda_j  \otimes u^j_{2,k}
\end{equation*}
 for $(j,k)\in\cJ$.

Fix  $\cI\subset\cJ$ and let  $\mathscr U_\cI = \{u^j_{2,k}: (j,k)\in\cI\}\subset \rg(A^*)\subset \rC^d$.  Suppose
\begin{equation}
 \label{eq:sums to zero}
0= \sum_{(j,k)\in\cI} c^j_k v^j_{2,k}.
\end{equation}
 Applying $\sW$ to equation \eqref{eq:sums to zero}
\[
 0 = \sum_{(j,k)\in\cI} c^j_k \laambda_j\otimes u^j_{2,k}.
\]
 Given a vector $\gaamma \in H$, applying the operator $\gaamma^* \otimes I$ yields 
\begin{equation}
 \label{eq:gamma0}
 0 = \sum_{(j,k)\in\cI} c^j_k \gaamma^*\laambda_j u^j_{2,k} =\sum_{(j,k)\in\cI} a^j_k u^j_{2,k},
\end{equation}
where $a^j_k = c^j_k \gaamma^* \laambda_j$.   

Suppose  $\mathscr U_\cI$ is  
linearly independent and $(p,m)\in \cI$. %
Choosing $\gaamma =\laambda_p$ in equation \eqref{eq:gamma0} 
gives   $a^p_m=c^p_m \laambda_p^* \laambda_p =0$.
Thus $c^p_m=0$ 
 for each $(p,m)\in\cI$.  Hence 
$\mathscr{V}_\cI :=\{v^j_{2,k}: (j,k)\in\cI\}\subset \rg(B^*) \subset \rC^e$  is linearly independent
and in particular %
the cardinality of $\cI$ is at most  $\dim(\rg(B^*))$. 
If $\mathscr U_\cJ = \{u^2_{j,k}:j,k\}$ spans  $\rg(A^*)$, 
then choosing $\mathscr U_\cI$ a basis for $\rg(A^*)$  shows
 $\dim(\rg(A^*) \le \dim(\rg(B^*).$ %
Further if $\dim(\rg(A^*))=d$, then $d\le \rg(B^*)\le e$.

To prove item \eqref{it:action dlee}, simply observe if $\alpha^j$ and $u^2_j$ 
validate the eig-generic (resp. weak eig-generic) hypothesis, then $\mathscr U_\cJ$ does span  $\C^d$ (resp. $\rg(A^*)$)
and hence $d\le e$ (resp. $\dim(\rg(A^*))\le \dim(\rg(B^*))$).

To prove item \eqref{it:action d=e}, 
suppose the tuples $\alpha^j$ and the
vectors $u^j$  validate the (resp.~weakly)  eig-generic assumption. 
Thus, there is an $\cI\subset\cJ$ such that $\mathscr U_\cI$ is a 
hyperbasis for $\rC^d$ (resp.~$\rg(A^*)$).  Since a hyperbasis for $\rC^d$ contains $d+1$ 
(resp.~$\dim(\rg(A^*)) +1$) elements, 
the cardinality of $\cI$ is $d+1$ (resp.~$\dim(\rg(A^*)) +1$). 
Assuming $d=e$  (resp.~$\dim(\rg(A^*)) = \dim(\rg(B^*))$), 
the set $\{v^j_k: (j,k)\in\cI\}  $ is then a set of $e+1$ (resp.~$\dim(\rg(B^*)) +1$)
elements in $\rC^{e}$ (resp.~$\rg(B^*)$), so is  linearly dependent. 
On the other hand, if $\cI^\prime$ is a subset of $\cI$ of cardinality $d$, then
$\mathcal U_{\cI^\prime}$ is a basis for $\C^d$ (resp. $\rg(A^*)$) and
hence $\mathscr V_{\cI^\prime} =\{v^2_{j,k}:(j,k)\in \cI^\prime\}$ is a basis
for $\C^d$ (resp. $\rg(B^*)$). Hence $\mathscr V_\cI$ is a hyperbasis for $\C^d$ (resp. $\rg(B^*)$). 
Hence, there exists, for $(j,k)\in\cI,$ scalars  $c^j_k$ none of which are $0$
such that  equation \eqref{eq:sums to zero} holds.  
Given  $(p,m)\in\cI$, an  application of equation \eqref{eq:gamma0}
with a (nonzero) vector $\gaamma$ orthogonal to $\laambda_p$ gives $a^p_m=0$ and hence, again by the hyperbasis property, $a^j_k=0$ 
 for all $(j,k)\in\cI$. 
 Since $c^j_k\ne 0$, it follows that $\gaamma$ is orthogonal to
  each $\laambda_j$ and consequently  the unit vectors  $\laambda_j$ are all colinear. 
By multiplying $v^j$ by a unimodular constant as needed, it may be assumed that there is a unit vector $\laambda\in H$ such that $\laambda_j=\laambda$ for all $j$. With this re-normalization, for $(j,k)\in \cI$, 
\begin{equation}
\label{eq:Wvjk}
 \sW v^j_k = \laambda\otimes u^j_{2,k},
\end{equation}
 completing the proof of item \eqref{it:action d=e}.

Turning to the proof of item \eqref{it:action d=e alt}, it  follows immediately from equation \eqref{eq:Wvjk} that
that for each $v\in \rC^d$ (resp.~$v \in \rg(B^*)$) there is a 
 $u\in \rC^d$ (resp.~$ u \in \rg(A^*)$) such that
\begin{equation*}
 \sW v = \laambda\otimes u,
\end{equation*}
  since   $\{v^j_{2,k}: (j,k)\in\cI\}$ spans $\rC^d$ (resp.~$\rg(B^*)$).
  Hence, by linearity and since $\sW$ is an isometry, 
 there is a unitary mapping $M:\C^d\to \C^d$ (resp. $M: \rg(B^*) \to \rg(A^*)$)
 such that $\sW=\lambda\otimes M$ (resp. $\sW|_{\rg(B^*)} = \laambda\otimes M$).

For the proof of item \eqref{it:action one-term}, assuming $A$ is $*$-generic (resp.~weakly $*$-generic),
 for $1\le j\le \ell$,
there exists tuples $\beta^j$ of sizes $n_j$ and vectors
\[
 \mathrm{u}_1^j = \sum_{k=1}^{n_j} \mathrm{u}^j_{1k}\otimes e_k
\]
satisfying the $*$-generic (resp.~weak $*$-generic) condition for
$A$. That is, $I-\Lambda_A(\beta^j)\Lambda_A(\beta^j)^*$ is positive
semidefinite with one-dimensional kernel spanned by $\mathrm{u}_1^j$
and the set of vectors $\{\mathrm{u}^j_{1k}: 1\le j\le \ell, \, 1\le k
\le n_j \}$ spans $\rC^d$ (resp.~$\rg(A)$).  By Lemma \ref{lem:even
  better}, there exist vectors
\[
 \mathrm{u}^j_2 = \sum_{k=1}^{n_j} \mathrm{u}^j_{2k}\otimes e_k
\] 
 such that $I-\Lambda_A(\beta^j)^* \Lambda_A(\beta^j)$ has a   one-dimensional kernel spanned by $\mathrm{u}_2^j$.  On the other hand, the tuples
\[
 X^j=\begin{pmatrix} 0 & \beta^j \\ 0 & 0 \end{pmatrix}
\]
 lie in the boundary of $\cD_A$. Hence, as before $p(X^j)$ lies in the 
  boundary of $\cD_B$. Thus
\[
 L_B(p(X^j)) = 
  \begin{pmatrix} I & \Lambda_B(\ell(\beta^j)) \\ \Lambda_B(\ell(\beta^j))^* & I \end{pmatrix}
\]
is positive semidefinite and has a kernel. Hence, there exists vectors 
$\mathrm{v}^j= \mathrm{v}^j_1 \oplus \mathrm{v}^j_2$  such that  $L_B(p(X^j))\mathrm{v}^j=0$. 
By  Lemma \ref{lem:elementary largest}  these vectors are related by
\begin{equation*}
\begin{aligned}
  \Lambda_A(\beta^j)^* \mathrm{u}_1^j& =  -\mathrm{u}_2^j ,  \\
  \Lambda_A(\beta^j)\mathrm{u}_2^j & =   -\mathrm{u}_1^j, \\
  \end{aligned} \qquad
  \begin{aligned}
  \Lambda_B(\ell(\beta^j))^* \mathrm{v}_1^j& =  - \mathrm{v}_2^j, \\
 \Lambda_B(\ell(\beta^j)) \mathrm{v}_2^j& =  - \mathrm{v}_1^j.  \\
\end{aligned}
\end{equation*}
Write 
\[
 \mathrm{v}^j_i = \sum_{k=1}^{n_j}\mathrm{v}_{i,k}^j \otimes e_k.
\]
 By Lemma \ref{lem:even better},  for each $j$  there exists a vector $\tau_j \in H$ such that 
\[
 \sW \mathrm{v}^j_{2,k} = \tau_j \otimes \mathrm{u}^j_{2,k}
\qquad\text{and}\qquad
 \sW \mathrm{v}^j_{1,k} = C \tau_j \otimes \mathrm{u}^j_{1,k}
\]
for each $1\le k\le n_j$.  Now suppose further that $A$ is %
weakly eig-generic and
$\dim(\rg(B^*))=\dim(\rg(A^*))$.  In this case, by the already proved item \eqref{it:action d=e alt}, there is a unit vector
$\laambda$ and unitary mapping  $M:\rg(B^*)\to \rg(A^*)$ such that $\sW \mathrm{v} =
\laambda\otimes M\mathrm{v}$ on $\rg(B^*)$. Since $\mathrm{v}^j_{2k}
\in \rg(B^*)$ and $\sW$ and $C$ are isometries, it follows that
$\tau_j=\rho_j \laambda$ for some scalar $\rho_j\ne 0$. Hence,
\begin{equation}
 \label{eq:C*sW}
 \sW \mathrm{v}^j_{1,k} = C \laambda \otimes \rho_j \mathrm{u}^j_{1,k}
\end{equation}
for each $1\le j\le \ell$ and $1\le k\le n_j$.  Since
$\{\mathrm{u}^j_{1,k}:j,k\}$ spans $\rg(A)$ and both $C$ and $\sW$ are unitary, equation \eqref{eq:C*sW}
implies there is an isometry $Z:\rg(A)\to \rg(B)$ such that
\begin{equation*}
 C(\laambda \otimes \mathrm{u}) = \sW Z\mathrm{u}
\end{equation*}
for $\mathrm{u}\in \rg(A)$. In particular, $\dim(\rg(A)) \le
\dim(\rg(B))$.  Hence, if $\rg(A)=\rC^d$ (as is the case if $A$ is
$*$-generic) and $e=d$, then $\rg(B)=\rC^d$ and $Z$ is onto. In the
case that $A$ is only weakly $*$-generic, we have assumed the
dimensions of $\rg(A)$ and $\rg(B)$ are the same and so again $Z$ is
onto. So in either case, $Z$ is unitary.  In particular, given
$\mathrm{v}\in \rg(B)\cap \rg(B^*)$, there is a $\mathrm{u} \in
\rg(A)$ such that $Z\mathrm{u}=\mathrm{v}$ and
\[
 C(\laambda\otimes Z^*\mathrm{v}) = \sW \mathrm{v}.
\]
On the other hand, as $\mathrm{v}\in \rg(B^*)$, we have 
$\sW \mathrm{v} =\laambda\otimes M\mathrm{v}$. Hence,
\begin{equation}
 \label{eq:Clambda}
 C(\laambda\otimes Z^* \mathrm{v}) =  \laambda \otimes M\mathrm{v}.
\end{equation}
Hence, letting $N$ denote the restriction of $Z^*$ to 
$\rg(B)\cap \rg(B^*)$ the desired conclusion follows.

We now take up the case $A$ is eig-generic and $*$-generic  and $e=d$.  
In this case, $M$ is a $d\times d$ unitary matrix by item \eqref{it:action d=e alt}.
Moreover, as noted above $d=\dim(\rg(A))\le \dim(\rg(B))\le d$. On the other hand,
from item \eqref{it:action dlee}, $d\le \dim(\rg(B^*))\le d$ and hence $\rg(B^*)=\C^d$. 
It follows that $Z:\C^d\to \C^d$ is unitary  and letting 
$\mathrm{u}= Z^*\mathrm{v}$ in equation \eqref{eq:Clambda} gives,
\[
 C(\laambda\otimes \mathrm{u}) = \laambda \otimes M N \mathrm{u}
\]
and the proof of item \eqref{it:action one-term} is complete.
\end{proof}

\begin{remark}\rm
 \label{rem:weakC}
   In the context of item \eqref{it:action one-term}, the dimension of $\rg(B)\cap\rg(B^*)$ is at most the dimension of $\rg(A)\cap \rg(A^*)$. In the case that these dimensions coincide, the identity  $C^*(\laambda\otimes Mv) = \laambda\otimes Nv$ for $v\in \rg(B)\cap\rg(B^*)$ implies there is a unitary mapping $Z$ of  $\rg(A)\cap\rg(A^*)$ such that $C^*(\laambda\otimes z) = \laambda \otimes Zz$ for $z\in \rg(A)\cap\rg(A^*)$; i.e., $C^*=I\otimes Z$ on $\C\laambda\otimes \rg(A)\cap \rg(A^*)$.
 \end{remark}

\begin{thm}
 \label{thm:one-sided}
  Suppose
\ben[\rm (a)]
  \item $A,B\in \matdg$;
  \item $\cD_A$ is bounded; 
  \item $p$ is a  mapping from $\cD_A$ into $\cD_B$ that is analytic and 
 bounded on a free pseudoconvex set $\cG_Q$ containing $\cD_A$;
  \item $p$ maps the boundary of $\cD_A$ into the boundary of $\cD_B.$
\een
If $A$ is eig-generic and  $*$-generic and $p(x)=x+f(x)$, where $f$ consists of terms of degree two and higher, then 
there exists a $d\times d$ matrix-valued analytic function ${\tt W}$ such that
\[
 L_B(p(x))= {\tt W}(x)^* L_A(x) {\tt W}(x)
\]
and thus the conclusions of Theorem \ref{thm:shotinthedark}  hold. In particular,
 there exist $d\times d$ unitary matrices $C,\sW$ such that, 
$B=\sW^* CA\sW$ and $(\cD_A,\cD_B)$ is a spectrahedral pair
with associated \ct map $p$. 
\end{thm}

\begin{cor}
 \label{cor:main}
     Suppose
 \ben[\rm (a)]
  \item $A\in \matdg$  and $B\in \mateg$; 
  \item $\cD_A$ is bounded; 
  \item $A$ is eig-generic and  $*$-generic; 
  \item $p$ is an analytic mapping  $\cD_A\to\cD_B$ analytic  and bounded on a
    free pseudoconvex set $\cG_P$ containing $\cD_A$
     with $p(x)=x+f(x)$ where $f$ consists of terms of degree two and higher; 
 \item $r$ is an analytic mapping $\cD_B\to\cD_A$ analytic   and bounded
   on a free pseudoconvex set $\cG_Q$ containing $\cD_B$, with $r(x)=x+h(x)$ where $h$ consists of
  terms of degree two and higher;
 \item $p$ maps the boundary of $\cD_A$ into the boundary of $\cD_B$ and $r$ maps the boundary of $\cD_B$ to the boundary of $\cD_A.$
\een
If $B$ is eig-generic, then $d=e$. In any case, if $d=e$, then  the maps
$p$ and $r$ are bianalytic (\ct)
 between $\cD_A$ and $\cD_B$ and between $\cD_B$ and $\cD_A$
respectively. Moreover,  for each the conclusions of
Theorem \ref{thm:shotinthedark} hold. In particular,
 there exist $d\times d$ unitary matrices $C,\sW$ such that, 
$B=\sW^* CA\sW$ and $(\cD_A,\cD_B)$ is a spectrahedral pair
with associated \ct  map $p$. 
\end{cor}

\begin{proof}
By Lemma \ref{lem:eig-generic in action}
 \eqref{it:action dlee} the assumption $A$ is eig-generic implies $d\le e$. 
  If $B$ is assumed eig-generic, then reversing the roles of $A$ and $B$ and using 
  $r$ in place of $p$,  implies $e\le d$.  Thus in any case $d=e$ and 
Theorem  \ref{thm:one-sided} applies to complete the proof.
\end{proof}

\begin{proof}[Proof of Theorem~\ref{thm:one-sided}]
  We begin by considering the case, as in Remark \ref{rem:one-sided} below,
 that $\tg\ge g$ and $B\in \matdtg$, leaving the special case $\tg=g$ for later.
  Since $p$ maps $\cD_A$ into $\cD_B$ and $p(0)=0$, by the Analytic
  Positivstellensatz (here is where the hypothesis $\cD_A$ is bounded
  is used), Corollary \ref{thm:analPoss}, there exists a Hilbert space
  $H$, a unitary mapping $C$ on $H\otimes\rC^d$ and an isometry
  $\sW:\rC^e\to H\otimes \rC^d$ such that
\[
 L_B(p(x)) = W(x)^* \big(I_H\otimes L_A(x)\big) W(x),
\]
where
\[
 W(x) = (I-\La_R(x))^{-1}\sW
\]
and $R=(C-I)(\IHA)$. Moreover, 
\[
 L_B(p(X))=W(X)^* L_{\IHA}(X) W(X)
\]
 holds for nilpotent $X\in M(\C)^g$ and the identities of equation \eqref{eq:preiso1alt} hold with $G(x)=\Lambda_B(p(x))$.

Lemma  \ref{lem:eig-generic in action}~\eqref{it:action one-term} implies there
 is a vector $\laambda\in H$ and $d\times d$ unitary matrices $M$ and
 $N$ such that $\sW =\laambda\otimes M$ and $C(\laambda\otimes I) =
 \laambda\otimes N$. To complete the proof, let
\[
 {\tt W}(x) = [\laambda^*\otimes I]W(x).
\]
 Importantly,  ${\tt W}$ is a square ($d\times d$) matrix-valued analytic function. Further,
\[
 \begin{split}
 L_B(p(x)) & =  W(x)^* (I_H\otimes L_A(x))W(x)
  =  {\tt W}(x)^* [\laambda^*\otimes I] (I_H\otimes L_A(x))[\laambda \otimes I] {\tt W}(x)\\
 & =  {\tt W}(x)^* L_A(x) {\tt W}(x).
\end{split}
\]
If $\tg=g$ then Theorem \ref{thm:shotinthedark} applies and concludes the proof.
\end{proof}

\begin{remark}\rm
\label{rem:one-sided} 
 If, in the setting of Theorem \ref{thm:one-sided}
 or Corollary \ref{cor:main}, 
the assumptions are relaxed to
 $B\in \matdtg$ with
 $\tg\geq g$, then we can conclude that $p$ satisfies the conclusions
 of Proposition \ref{prop:gvtg}.
\end{remark}

This section concludes with a proof of Theorem \ref{thm:main}.

\begin{proof}[Proof of Theorem~\ref{thm:main}]
 The assumption that $A$ and $B$ are both sv-generic immediately imply both are eig-generic and $*$-generic. 
 By assumption, $p$ has inverse $r$ and $r$ is a bianalytic map from $\cD_B$ to $\cD_A$.  Thus
 the hypotheses of Corollary \ref{cor:main}  are validated by those of Theorem \ref{thm:main}
 and the result thus follows from Corollary \ref{cor:main}. 
\end{proof}

\section{Affine Linear Change of Variables}
\label{sec:normalize}
This section describes the effects of change of variables by way of
pre and post composition with an affine linear map on an analytic
mapping between free spectrahedra.

   Suppose  $A=(A_1,\dots,A_g)\in M_d(\C)^g$  determines a bounded LMI domain 
 $\cD_A$, $B=(B_1,\dots,B_{\tg})\in M_e(\C)^\tg$ and 
  $p: \cD_{A}  \to  \cD_{B}$  is analytic with $p(0)= b$.
  
  We first, in Subsection \ref{sec:dponeone},
  turn our attention to conditions on $A$ and $B$ that
  guarantee  $p^\prime(0)$ is one to one. 
  Next, in Subsection  \ref{sec:affinechange}, assuming $p^\prime(0)$ is one to one, we
  apply a  linear transforms on the range of $p$
  placing $p$ into the canonical form $p(x) =(x,0) + h(x)$, where $h(x)$ consists of higher order terms ($h(0)=0$ and $h^\prime(0)=0$).
  In Subsection  \ref{sec:chgstruc} we consider
  an affine linear change of variables on the domain of $p$.

\subsection{Conditions guaranteeing $p'(0)$ is one to one}
\label{sec:dponeone}
  Natural hypotheses on a mapping $\cD_A$ to $\cD_B$ via $p$ lead to the conclusion that  $p^\prime(0)$ is one-one.  

\begin{lemma}
 \label{lem:AvsUA}
   Suppose $A\in \matdg$ and $B\in \matetg$ and $p:\cD_A\to\cD_B$ is analytic  and $p(0)$ is in the interior of $\cD_B$. If 
 \ben[\rm (a)]
   \item $p$ is proper; and
   \item $\cD_{A}$ is bounded, 
 \een
   then $p^\prime(0)$ is one-one. 
\end{lemma}

\begin{proof}
Let $b$ denote the constant term of $p$.  Thus $b$ is a row vector of length $\tg$ with entries from $\mathbb C$  and $q(x)=p(x)-b$ satisfies $q(0)=0$.
Let $\mathfrak B= L_B(b)$. In particular, $\mathfrak B$ is an $e\times e$ matrix 
 (since $B\in \matetg$). It is also  positive definite, since $\mathfrak{B}=L_B(p(0))$ and $p(0)$ is in the interior of $\cD_B.$
Let $\mathfrak{H}$ denote the positive square root of $\mathfrak{B}$ and 
define $F=\mathfrak{H}^{-1} B \mathfrak{H}^{-1}$.  Thus, $F\in \matetg$ is a $\tg$ tuple of $e\times e$ matrices and 
\[
 \mathfrak{H}^{-1} L_B(p(x)) \mathfrak{H}^{-1} = L_F(q(x)).
\]
In particular, for a given $n$ and tuple $X\in\matngc$, we have  $L_B(p(X))\succeq 0$ if and only if $L_F(q(X))\succeq 0$ and  
$q$ is proper since $p$ is assumed to be.
If $p^\prime(0)=q^\prime(0)$ is not one-one, then there exists a non-zero $a\in \C^g$ such that   $q^\prime(0)a = 0$. 
Given  $S,$  a non-zero matrix nilpotent of order two,
let $X=a\otimes S = \begin{pmatrix} a_1 S, \dots, a_g S \end{pmatrix}.$   It follows that
\[
 q(tX) = t (q^\prime(0)a)\otimes S =0.
\]
  Since $\cD_{A}$ is bounded and contains $0$ in its interior and $X\ne 0$, there exists a $t$ such that $tX$ is in the boundary
  of $\cD_{A}$. On the other hand, since $q(tX)=0$, the tuple $tX$ is not in the boundary of $\cD_{F}$, contradicting
  the fact that $q$ is proper.  Hence $q^\prime(0)$ is one-one.
\end{proof}

\subsection{Affine linear change of variables for the range of $p$}
\label{sec:affinechange}
 In this section we compute explicitly the effect of an affine linear change of variables in the range space of $p$.
 This change of variable can be used to produce a new map $\tp$ with $\tp'(0) =I$, used 
 later in the proof of Theorem \ref{thm:PQ}.
Given a $g\times g$ matrix $M$ and an analytic mapping $q=\begin{pmatrix} q^1 & \cdots & q^g\end{pmatrix}$, let $qM$ denote the analytic mapping,
\[
  q(x)M = \begin{pmatrix} q^1(x) & \dots q^g(x)\end{pmatrix} M = \begin{pmatrix}\sum q^j(x) M_{j,1}, \dots, \sum q^j(x) M_{j,g} \end{pmatrix}.
\]
On the other hand, for $B\in \matngc$, we often write $MB$ for $(M\otimes I)B$ where $B$ is treated as a column vector. Thus,
\[
 MB = \begin{pmatrix} \sum M_{1,j}B_j \\ \vdots \\ \sum M_{g,j}B_j \end{pmatrix}.
\]
 Since we are viewing $x$ and $p(x)$ as row vectors, in the case $p$ has $\tg$ entries,  $p^\prime(0)$ is a $g\times \tg$ matrix.

\begin{prop}
  \label{prop:dF}
  Suppose $A\in\matdgc$ and $B\in\matetg$ and $p:\cD_A\to\cD_B$ is an
  analytic map with $p(0)=b\in\mathbb C^\tg$.  Let $\fcH$ denote the
  positive square root of
  \[
    \mathfrak{B}=L_B(b) = I+\sum_{j=1}^g b_j B_j+\sum_{j=1}^g (b_jB_j)^*,
  \]
 and let $F= p^\prime(0)\, \fcH^{-1}B\fcH^{-1}=  \fcH^{-1} \, p^\prime(0) B \, \fcH^{-1}$. 

 Suppose $p^\prime(0)$ is one-one and choose any invertible $\tg\times
 \tg$ matrix $M$ whose first $g$ rows are those of $p^\prime(0)$   and let $\ell$
 denote the affine linear polynomial $\ell(x) = (-b+x) M^{-1}.$ The
 analytic map $\tilde{p}= \ell \circ p$ maps $\cD_A$ into $\cD_F$ and
 satisfies $\tilde{p}(0)=0$ and $\tilde{p}^\prime(0)=\begin{pmatrix}
   I_d & 0 \end{pmatrix}$. Thus $\tilde{p}(x)=(x \ \ 0) + h(x)$ where $h(0)=0$
 and $h^\prime(0)=0$.  In particular, if $p$ maps the boundary of
 $\cD_A$ into the boundary of $\cD_B$, then $\tilde{p}$ maps the boundary
 of $\cD_A$ to the boundary of $\cD_F$; and
 if $p$ is bianalytic, then so is $\tilde{p}$.
  
Written in more expansive notation,  for each $1\leq i\leq g$, 
  \[
  \begin{split}
    F_i & = \fcH^{-1}(MB)_i\fcH^{-1} = \sum_{j=1}^g p'(0)_{i,j}\fcH^{-1}B_j\fcH^{-1},\\
      B_i & = \fcH(M^{-1}F)_i\fcH = \sum_{j=1}^g \left(M^{-1}\right)_{i,j}\fcH F_j \fcH.
      \end{split}
  \]
\end{prop}

\begin{proof}
Consider
  \begin{align*}
    \sum_{k=1}^\tg (M B)_k\otimes \tilde{p}(x)_k
      &=\sum_{k=1}^\tg (M B)_k\otimes \left((-b+p(x))M^{-1}\right )_k\\ 
      &=\sum_{k=1}^\tg\left[ \left(\sum_{j=1}^\tg M_{k,j}B_j \right)
        \otimes\left(\sum_{i=1}^\tg(-b_i+p_i(x))\left(M^{-1}\right)_{i,k}\right)\right]\\ 
      &=\sum_{j=1}^\tg\sum_{i=1}^\tg\left(\sum_{k=1}^\tg (M^{-1})_{i,k}M_{k,j}\right)
        \Big[B_j\otimes \big(-b_i+p_i(x)\big)\Big]\\
      &=\sum_{j=1}^\tg\sum_{i=1}^\tg (I_{i,j})\Big[B_j\otimes \big(-b_i+p_i(x)\big)\Big] =\sum_{j=1}^\tg B_j\otimes p_j(x) - \sum_{j=1}^\tg B_jb_j.
  \end{align*}
    Given a tuple $X$,  it follows that
  \begin{align*}
    L_F(\tilde{p}(X)) &= I+\sum_{j=1}^\tg F_j\otimes \tilde{p}_j(X)
      + \sum_{j=1}^\tg \adj{F_j}\otimes\adj{\tilde{p}_j(X)}\\
      &=I+\sum_{j=1}^\tg \left(\fcH^{-1} (M B) \fcH^{-1}\right)_j\otimes \tilde{p}_j(X)
      + \sum_{j=1}^\tg \left(\fcH^{-1} (MB) \fcH^{-1}\right)^*_j\otimes\adj{\tilde{p}_j(X)}\\
      &=(\fcH^{-1}\otimes I)\left(\mathfrak{B}+\sum_{j=1}^\tg  (MB)_j\otimes \tilde{p}_j(X) 
      + \sum_{j=1}^\tg (MB)^*_j\otimes\adj{\tilde{p}_j(X)}\right)(\fcH^{-1}\otimes I)\\
      &=(\fcH^{-1}\otimes I)\left(I+\sum_{j=1}^\tg B_j\otimes p_j(X)
      + \sum_{j=1}^\tg \adj{B}_j\otimes\adj{p(X)}_j\right)(\fcH^{-1}\otimes I)\\
      &=(\fcH^{-1}\otimes I)L_B(p(X))(\fcH^{-1}\otimes I).
  \end{align*}
  Since $\fcH^{-1}\otimes I$ is invertible,
  $L_F(\tilde{p}(X))\succeq 0$ if and only if $L_B(p(X))\succeq 0$. Assuming  $p:\cD_A\to\cD_B$ 
  and $X\in \cD_A$, it follows that $\tilde{p}(X)\in\cD_F$. 
  
  Next we note that $\tilde{p}(0) = (-b+p(0)) M^{-1}=0$ and that
  $\ell^\prime(x) = M^{-1}$. Hence, with $P = \begin{pmatrix} I & 0 \end{pmatrix}$,
  \[
      \tilde{p}^\prime(0) = \ell^\prime \big(p(0)\big)p^\prime(0) = M^{-1}p^\prime(0) 
  \]
  and thus, $\tilde{p}^\prime(0) = p(0)M^{-1} = P MM^{-1} = P$.
\end{proof}

\begin{remark}\rm
\label{rem:mainrelax}
 In Theorem \ref{thm:main}, without the assumption $p^\prime(0)=I$, the remaining 
 hypotheses imply $p^\prime(0)$ is invertible \cite[Theorem 3.4]{HKM11b}.  Applying Proposition \ref{prop:dF} with $b=0$ and $M=p^\prime(0)$,
 gives $\fcH=I$ and $F=p^\prime(0)B.$  Moreover, since $B$ is sv-generic, so is $F$.  The resulting $\tilde{p}$
 thus does satisfy the hypotheses of Theorem \ref{thm:main}. It is now just a matter of undoing the linear change
 of variables that sent $B$ to $F$.
\end{remark}

\subsection{Change of basis in the $\mathcal R$ module generated by $\mathcal A$}
\label{sec:chgstruc}
In the context of the results of Theorem \ref{thm:main}, 
the  formula for the \ct  mapping $p$ depends (only) upon the structure
matrices $\Xi$ for the module generated by the tuple $A$ 
  over the algebra generated by the tuple $R=(C-I)A$ with respect
 to the basis implicitly given by $A=(A_1,\dots,A_g)$.
 We now see that a  linear change of the $A$ variables 
 produces a simple linear ``similarity" transform
 $\tp$ of the mapping $p$. 

Starting with the identity of equation \eqref{eq:AZA}, 
consider a  linear change of variables determined by an invertible
matrix  $M \in \C^{g \times g}$.  That is,  $\tilde{A} = MA$ where $A$ is
regarded as the column
    of matrices $\bem
    A_1\\ \vdots \\ A_g \eem$,
    so $\tilde A_i=\sum_{j=1}^gM_{ij}A_j$ for $i=1,\ldots,g$.
    The matrix $M$ implements a change of basis
    on the span of $\{A_1,\dots,A_g\}$.
    We emphasize that  the vectors of variables and maps
    are row vectors. Observe, in view of equation \eqref{eq:AZA}, 
\[
\begin{split}
\tilde A(C-I) \tilde A_j&= M A (C-I) (\sum_{k=1}^gM_{jk}A_k)
=M(\sum_{k=1}^gM_{jk}(A(C-I)A_k))\\
&=M(\sum_{k=1}^gM_{jk}(\Xi_k A))
=(M(\sum_{k=1}^gM_{jk} \Xi_k) M^{-1})\tilde A.
\end{split}
\]
 Thus,  $(\tilde{A},C)$ satisfy the hypotheses of Proposition \ref{prop:AZA} 
 with structure matrices 
   \[
   \widetilde{\Xi}_j=M(\sum_{k=1}^gM_{jk}\Xi_k)M^{-1}.
   \]
Concretely,
\begin{equation}
\label{eq:Xitilde}
 (\widetilde{\Xi_j})_{s,q} = \sum_{t,k,p} M_{s,t} M_{j,k} (\Xi_k)_{t,p} (M^{-1})_{p,q}.
\end{equation}

\subsubsection{Computation of the mappings after linear change of coordinates}
 \label{sec:affine change}
  Recall the \ct mapping 
     $p(y)=y (1-\sum_{i=1}^g y_i \; \Xi_i)^{-1}$
   associated to the \ct tuple $\Xi$. 
   We now look at the effect of the
    linear  change of variable implemented by  $M$ on $p$.
     The rational function $\tilde{p}$ determined by 
     $\widetilde{\Xi_k}$ of equation \eqref{eq:Xitilde},  is
\[
\begin{split}
\tilde p(y)&=y(1-\sum_{i=1}^g y_i \widetilde{\Xi_i} )^{-1}
=y (1-\sum_{i=1}^gy_iM(\sum_{k=1}^gM_{ik} \Xi_k)
M^{-1})^{-1}\\
&=yM(1-\sum_{k=1}^g(\sum_{i=1}^gy_iM_{ik})\Xi_k)^{-1}M^{-1}.\\
\end{split}
\]
Note that  $\sum_{i=1}^gy_iM_{ik}$ is the $k$-th column of $yM$,
thus $\tp(y) =p(yM)M^{-1}$.
 Similarly  for $\tp$  inverse, denoted $\tq$, so we can summarize this as
\[
   \tp(y) =p(yM)M^{-1}, \qquad \qquad \tilde q(y)=q(yM)M^{-1}.
   \]
(For an example see  \eqref{eq:chCT} below.)

The following proposition summarizes the  mapping implications of this change of variable.

\begin{prop}
If $p:\cD_A \to \cD_B$, then  $\tp: \cD_\tA \to \cD_\tB$
where
    $$ \cD_\tA: =  \{ y : yM \in \cD_A \},
    \qquad \qquad
     \cD_\tB: =  \{ z : zM \in \cD_B \}  $$
\end{prop}

\begin{proof}
  Given $ y \in \cD_\tA$, set $ yM =x$, which by definition is in $\cD_A$.
  By the formula above $\tp(y) = p(x)M^{-1} =:z$.
  Thus $ zM = p(x) \in \cD_B$, hence by definition of $\cD_\tB$,
  we have $\tp(y) =z \in \cD_\tB$.
\end{proof}

\subsection{Composition of \ct maps is not necessarily \ct}\label{ssec:compose}
Suppose $p:\cD_A\to\cD_B$ and $q:\cD_B\to\cD_E$ are \ct maps
between the spectrahedral pairs $(\cD_A,\cD_B)$ and $(\cD_B,\cD_E)$.
In particular, the pairs $(A,B)$ and $(B,E)$ must satisfy rather
stringent algebraic conditions.  In this case, generically $q\circ p$
is again \ct by Theorem \ref{thm:one-sided}. 
On the other hand, in general, given \ct maps $p$ and $q$ (without
specifying domains and codomains), there
is no reason to expect that the composition $q\circ p$ is \ct.
Indeed the following example shows it need not be the case.

Let $g=2$ and let $p$ be the indecomposable \ct map of Type I from 
Section \ref{sec:examples}.  Let $\tilde{p}$ denote the convexotonic map
\[
\tilde{p}=(x_1 + x_2^2,x_2).
\]
It can be obtained by reversing the roles of $x_1$ and $x_2$ in $p$ or observing it 
belongs to the convexotonic tuple
\[
 \Xi_1 = 0, \ \ \Xi_2 = \begin{pmatrix} 0 &0\\1 & 0 \end{pmatrix}.
\]
In terms of the formalism in Subsection \ref{sec:chgstruc}, 
 consider the change of basis matrix
\[
M=\begin{pmatrix}
 0 & 1 \\
 1 & 0 
\end{pmatrix}
\]
and note 
\beq\label{eq:chCT}
\tilde p(x,y)=p(y,x)\begin{pmatrix}
 0 & 1 \\
 1 & 0 
\end{pmatrix}
= 
\begin{pmatrix} x+y^2 &y\end{pmatrix}.
\eeq
Now 
\[
p(\tilde p(x,y)) = \begin{pmatrix}
x+y^2 & y+x^2+xy^2+y^2x+y^4
\end{pmatrix}
\]
is not \ct by Proposition \ref{lem:gtg},  since it is a polynomial of degree exceeding two.

  \section{Constructing all \Ct Maps}
\label{sec:examples}

To construct all \ct maps in $g$ variables first one lists the indecomposable
ones, i.e., those associated with an indecomposable algebra.
Then  build general \ct maps as direct sums of these.
We illustrate this by giving all \ct maps 
in dimension $2$ in Subsections \ref{sec:two dim} and \ref{subsec:decompose}. Finally, in Subsection \ref{ssec:ball} we 
show how the automorphisms of the complex wild ball 
$\sum X_j^* X_j \preceq I$
are, after affine linear changes of variables, \ct.

\subsection{\Ct maps for $g=2$}
\label{sec:two dim}
  In very small dimensions $g\leq5$ indecomposable algebras are classified \cite{Maz79}. 
 We work out the corresponding \ct maps for $g=2$.
The following is the list of indecomposable two-dimensional
algebras over $\C$ (with basis $R_1,R_2$).
\[
\begin{array} {c|ccc|cccccccc}
\text{notation} & \qquad  &\qquad \text{nonzero} \ \text{products}  && \text{properties}\\ \hline
\rm I & R_1^2=R_2  &             && \text{commutative}  & \text{nilpotent}\\
\rm II & R_1^2=R_1 & R_1R_2=R_2&& \\
\rm III & R_1^2=R_1 & R_2R_1=R_2 & &\\
\rm IV & R_1^2=R_1 & R_1R_2=R_2  & R_2R_1=R_2 & \text{commutative} & \text{with identity} \\
\end{array}
\]
Accordingly we refer to these as algebras of type I -- IV.

\subsubsection{Type I}
If  $R_1$ is nilpotent of order 3, then
$
\Xi_1= \begin{pmatrix} 0 & 1 \\ 0 & 0 \end{pmatrix} , \Xi_2= 0.
$
These structure matrices  produce the \ct maps
\begin{equation*}
p(x_1,x_2)=\begin{pmatrix} x_1 & x_2+x_1^2 \end{pmatrix}
\qquad
q(x_1,x_2)=\begin{pmatrix} x_1 & x_2 - x_1^2 \end{pmatrix}.
\end{equation*}
Note $p(0)=0$, $p'(0)=I$ and likewise for $q$.

\subsubsection{Type II}
Let 
$
R_1=\bem 1 & 0 \\ 0 & 0 \eem, 
R_2=\bem 0 & 1 \\ 0 & 0 \eem.
$
The corresponding structure matrices $\Xi$ are 
$
\Xi_1= \bem 1 & 0 \\ 0 & 0 \eem, 
\Xi_2= \bem 0 & 1 \\ 0 & 0 \eem.
$
So 
\[
\begin{split}
p(x)&= 
\bem
(1-x_1)^{-1} x_1 &
(1 - x_1)^{-1} x_2
\eem
\qquad
q(x) =  \bem (1+x_1)^{-1}x_1  & (1+x_1)^{-1} x_2 \eem.
\end{split}
\]

\subsubsection{Type III}
Let 
$
R_1=\bem 1 & 0 \\ 0 & 0 \eem, 
R_2=\bem 0 & 0 \\ 1 & 0 \eem.
$
The structure matrices  are 
$
\Xi_1= I_2, $ $
\Xi_2= 0.
$
So 
\[
\begin{split}
p(x) & = \bem x_1 (1-x_1)^{-1} & x_2(1-x_1)^{-1} \eem
\qquad
q(x)= \bem x_1(1+x_1)^{-1}& x_2(1+x_1)^{-1} \eem.
\end{split}
\]

\subsubsection{Type IV} Take
$
R_1=\bem 1 & 0 \\ 0 & 1 \eem, 
R_2=\bem 0 & 1 \\ 0 & 0 \eem.
$
The corresponding structure matrices  are 
$
\Xi_1= \bem 1 & 0 \\ 0 & 1 \eem, 
$ $
\Xi_2= \bem 0 & 1 \\ 0 & 0 \eem.
$
So  
\[
\begin{split}
p(x) & = \bem
x_1 (1 - x_1)^{-1} &
(1 - x_1)^{-1}  x_2 (1 - x_1)^{-1} 
\eem
\\
q(x)&= \bem
x_1 (1 + x_1)^{-1} &
(1 + x_1)^{-1}  x_2 (1 + x_1)^{-1} 
\eem.
\end{split}
\]
 
The other indecomposable \ct maps correspond to these after a linear change of basis, cf.~Section \ref{sec:chgstruc}. If the change of basis corresponds
to an invertible $2\times 2$ matrix $M$, then the corresponding
\ct map is
\[
\tilde p(x)=p(xM)M^{-1}.
\]

\subsection{\Ct maps associated to decomposable algebras}\label{subsec:decompose}
 Here we explain 
which \ct maps arise from 
decomposable algebras.
Suppose $\cR=\cR'\oplus\cR''$ and $\cR',\cR''$ are indecomposable finite-dimensional algebras.
Let $\{R_1,\ldots,R_g\}$ be a basis for $\cR'$ and
let $\{R_{g+1},\ldots,R_h\}$ be a basis for $\cR''$. Then
$\{R_1\oplus 0,\ldots, R_g\oplus 0, 0\oplus R_{g+1},\ldots,0\oplus R_h \}$
is a basis for $\cR$ with the corresponding structure matrices
\[
\Xi_j =
\begin{cases}
\Xi_j'\oplus 0 &: j\leq g \\
0\oplus \Xi_j'' &: j> g,
\end{cases}
\]
where $\Xi_j'$ and $\Xi_j''$ denote the structure matrices
for $\cR'$ and $\cR''$, respectively. The \ct map corresponding
to $\cR$ is
\[
\begin{split}
p_{\cR}(x)& =\begin{pmatrix}
x_1  &\cdots  & x_g & x_{g+1}&\cdots & x_h
\end{pmatrix}
\Big(I-\sum_{j=1}^{h} \Xi_j x_j \Big)^{-1} \\
& =\begin{pmatrix}
x_1  &\cdots  & x_g & x_{g+1}&\cdots & x_h
\end{pmatrix}
\begin{pmatrix}
I-\sum_{j=1}^{g} \Xi_j' x_j & 0 
\\ 0 &  I-\sum_{j=g+1}^{h} \Xi_j'' x_j
\end{pmatrix}^{-1} \\
& = 
\begin{pmatrix}
p_{\cR'}(x_1,\ldots,x_g) & p_{\cR''}(x_{g+1},\ldots,x_h)
\end{pmatrix}.
\end{split}
\]

\subsection{Biholomorphisms of balls}\label{ssec:ball}
\def\cF{\mathcal F}
\def\bB{\mathbb B}
In this subsection we show how (linear fractional) biholomorphisms of balls in $\C^g$ can be presented using \ct maps. 
Let $\{\hat e_1,\ldots,\hat e_{g+1}\}$ denote the standard basis 
of row vectors
for $\C^{g+1}$ and let $A_j=\hat e_1^* \hat e_{j+1}$ for $j=1,\ldots,g$.  
Since $\cD_A(1)=\{z\in\C^g: \sum_j |z_j|^2\le 1\}$ is the unit ball in $\C^g$, the free spectrahedron $\cD_A$ is a free version of the ball.
That is, 
$\cD_A=\{X:\sum X_j^* X_j \preceq I\}$ 
consisting of all row contractions.
Fix a (row) vector $v\in\C^g$ with $\|v\|<1$  and let $xv^* = \sum \overline v_j x_j$, where $x=(x_1,\ldots,x_g)$ is a row vector of free variables.  By \cite{Po2}, 
up to rotation,
automorphisms of $\cD_A$ have the form, 
\[
\cF_v(x)=v- 
\big(1-vv^*\big)^{\frac12} \big(1-xv^*\big)^{-1} x \big(I-v^*v\big)^{\frac12}  .
\]
Modulo affine linear transformations, $\cF_v$ is of the form
\[
\big(1-xv^*\big)^{-1} x = 
\begin{pmatrix}
(1-xv^*)^{-1} x_1 & \cdots &
(1-xv^*)^{-1} x_g
\end{pmatrix}
\]
since $\big(1-vv^*\big)^{\frac12}$ is a number
and $\big(I-v^*v\big)^{\frac12}$ is a matrix independent of $x$. Further,
\[
\big(1-xv^*\big)^{-1} x = x \big(I-v^*x)^{-1}.
\]
Now let $\Xi$ denote the $g$-tuple of $g\times g$ matrices $\Xi_j = e_j \otimes v^*$, where $\{e_1,\ldots,e_g\}$ is the standard basis of row vectors for $\C^g$. Then $\Xi_j\Xi_k=\overline v_j \Xi_k
=\sum_s (\Xi_k)_{j,s} \Xi_s$, so 
the tuple $\Xi$ satisfies \eqref{eq:cttuple}; i.e., it is \ct.  Moreover,
\[
x (I-\lambda_\Xi(x))^{-1} = x (I-v^*x)^{-1} =  (1-xv^*)^{-1} x.
\]
Thus $\cF_v$ is  a \ct map.

\section{Bianalytic Spectrahedra that are not Affine Linearly Equivalent}
	\label{sec:PQDomain}	

  In this section we present bounded 
free spectrahedra
   that are polynomially equivalent, but not affine linearly equivalent (over $\mathbb C$).  
	
	Suppose $A$ and $B$ are eig-generic tuples, $\cD_A$ and $\cD_B$ are bounded  and there is a polynomial bianalytic map 
	$p:\cD_A\to\cD_B$ with $p(x)=x + h(x)$ ($h$ for higher order terms).  In particular, by Theorem \ref{thm:shotinthedark}, $A$ and $B$ 
	have the same size and $B= W^* V A W$ for unitaries $V$ and $W$. Further,  there is a representation for $p$ in terms of the $g$-tuple $\Xi$ of $g\times g$ matrices determined by
\begin{equation}
 \label{eq:AVIA}
 A_k(V-I)A_j = \sum_{s=1}^g (\Xi_j)_{ks} A_s.
\end{equation}

\subsection{A class of examples}
	Let $Q$ be an  invertible $2\times2$ matrix,  so that 
  \[
  \cD_Q(1)=\big\{c\in\mathbb C: I_2+ \big((cQ)^* +cQ\big)\succeq 0\big\}
  \] 
	is bounded. Choose $\Pu,\Pd,\Po$ invertible, same size as $Q$ with $\Pd\Pu=-2Q$. Now let
\begin{equation}
\label{eq:thisisA}
		A_1=\begin{pmatrix} 0 & \Pu\\ \Pd & \Po \end{pmatrix}, \, A_2 = \begin{pmatrix} 0 &0 \\ 0 & Q\end{pmatrix}.
\end{equation}
 Given $\gamma$ unimodular, let 
\begin{equation}
\label{eq:thisisVg}
 V_\gamma = \begin{pmatrix} \gamma I_2 & 0 \\ 0 & I_2 \end{pmatrix}.
\end{equation}

\begin{prop}
 \label{prop:VA}
With notation as above,
\begin{equation}
 \label{eq:VA}
A_k (V_\gamma-I)A_j = \sum_{s=1}^2 (\Xi_j)_{ks} A_s,
\end{equation}
 where $\Xi=(\Xi_1,\Xi_2)$ is the tuple defined by
\[
 \Xi_1 =\begin{pmatrix} 0 & -2(\gamma-1) \\ 0 & 0 \end{pmatrix}
\]
 and $\Xi_2=0$.  Thus the polynomial mapping 
\[
   p_\gamma(x_1,x_2) = x(I-\Lambda_\Xi(x))^{-1} = (x_1, x_2 +2(1-\gamma)x_1^2)
\]
 is a bianalytic $p:\cD_A\to\cD_B$ with $B=V_\gamma A$.

 Moreover, if $\Po = \alpha_1Q +\alpha_3(\Pu^*\Pu+\Pd\Pd^*)$, then for each unimodular $\varphi$, 
\[
\begin{split}
	  s_\varphi & =  (\varphi x_1,\, -(1-\varphi)\left(4\overline{\alpha_3}\varphi-\alpha_1\right)x_1+x_2 +2(1-\varphi^2)x_1^2)\\
   & \phantom{=}  + (\overline{\alpha_3}(1-\varphi), \, -\overline{\alpha_3}(1-\varphi)\left(2\overline{\alpha_3}(1-\varphi)+\alpha_1\right)).
\end{split}
\]
 is a polynomial automorphism of $\cD_A$. 
\end{prop}

\begin{proof}
Equation \eqref{eq:VA} follows from the computations,  $(V_\gamma-I)A_2=0 = A_2(V_\gamma-I)$ and 
\[
 A_1(V_\gamma-I)A_1 = -2(\gamma-1) A_2.
\]
 The converse portion of Theorem \ref{thm:shotinthedark} now implies that 
\[
\begin{split}
 p & = x(I-\Lambda_\Xi(x))^{-1} =  \begin{pmatrix} x_1 & x_2 \end{pmatrix} \, \begin{pmatrix} 1 & 2(\gamma-1) x_1 \\ 0 & 1\end{pmatrix}^{-1}\\
 & =   \begin{pmatrix} x_1 & x_2 \end{pmatrix} \begin{pmatrix} 1 & 2(1-\gamma)x_1 \\ 0 & 1\end{pmatrix} 
 = (x_1, 2(1-\gamma)x_1^2 +x_2)
\end{split}
\]
is bianalytic between $\cD_A$ and $\cD_{V_\gamma A}$ as claimed.

To prove the second part of the proposition, suppose $\varphi$  is unimodular. Let $\delta=\overline{\alpha_3}(1-\varphi)$ and $\eta = -4\varphi \delta + (1-\varphi)\alpha_1$ and  let $\rho$ denote the affine linear polynomial,
\[
\rho(x_1,x_2) = (\varphi x_1,\, x_2+\eta x_1) + \delta(1,\, 2\delta -\alpha_1).
\]
With these notations,
\[
 L_A(\rho(x_1,x_2)) = L_A(\rho(0,0)) + (\varphi A_1 +\eta A_2) x_1 + A_2 x_2 + (\overline{\varphi} A_1^* +\overline{\eta}A_2^*) x_1^* +  A_2^*  x_2^*
\]
and, using $P_{22}-\alpha_1 Q = \alpha_3(\Pu^*\Pu+\Pd\Pd^*)$,
\[
 \begin{split}
 L_A(\rho(0,0)) & =  I + \delta A_1 + \delta^* A_1^* +  (2\delta^2 -\delta \alpha_1)A_2 + \overline{(2\delta^2-\delta\alpha_1)} A_2^* \\
   &=  \begin{pmatrix} I & \delta P_{12} +\overline{\delta}P_{21}^* \\ \overline{\delta}P_{12}^* + \delta P_{21} &
    \delta \alpha_3(\Pu^*\Pu+\Pd\Pd^*) + \overline{\delta \alpha_3} (\Pu^*\Pu+\Pd\Pd^*) -2\delta^2 Q - 2 (\overline{\delta})^2 Q^* \end{pmatrix} \\
  & =  \sY^* \sY,
\end{split}
\]
 where
\[
 \sY = \begin{pmatrix} \varphi I & \varphi  \big ( \delta P_{12} + \overline{\delta} P_{21}^* \big ) \\ 0 & I \end{pmatrix}.
\]
 Indeed, the only entry of this equality that is not immediate occurs in the $(2,2)$ entry. Since $\varphi$ is unimodular, $|\delta|^2 = \alpha_3\delta + \overline{\delta \alpha_3}$ and thus the  $(2,2)$ entry of $\sY^*\sY$ is
\[
\begin{split}
  I + \big ( \delta P_{12} + \overline{\delta} P_{21}^*)^* \big ( \delta P_{12} + \overline{\delta} P_{21}^*)
 & =  I + |\delta|^2 \big ( P_{12}^* P_{12} + P_{21} P_{21}^*) + \delta^2 P_{21} P_{12}  +\overline{\delta}^2 P_{12}^* P_{21}^* \\
  = I + \delta \alpha_3(\Pu^*\Pu+\Pd\Pd^*) + & \overline{\delta \alpha_3} (\Pu^*\Pu+\Pd\Pd^*) -2\delta^2 Q - 2 (\overline{\delta})^2 Q^*,
\end{split}
\]
 where $P_{21} P_{12} = -2Q$ was also used. 

 Next, let $B=V_\gamma A$, where $\gamma =\varphi^2$. For notational ease let $Y= \delta P_{12} + \overline{\delta} P_{21}^*$ and verify 
\[
 \begin{split}
 \varphi [P_{21}Y + Y^*P_{12}] +P_{22} 
 & =  \varphi [ \delta (P_{21}P_{12} + P_{21}P_{12}) + \delta^*(P_{12}P_{12}^*+P_{21} P_{21}^*)] +P_{22} \\
 & =  \varphi [-4\delta Q + (1-\overline{\varphi})\alpha_3 (P_{12}P_{12}^*+P_{21} P_{21}^*)] +P_{22} \\
 & =  -4\varphi\delta Q + (\varphi-1) (P_{22} -\alpha_1 Q) +P_{22}\\
 & =  ((1-\varphi)\alpha_1 - 4\varphi \delta) Q +\varphi P_{22} \\
 & =  \eta Q + \varphi P_{22}.
 \end{split}
\]
Hence,
\[
 \begin{split}
 \sY^* B_1 \sY& =  \begin{pmatrix} 0 &  \varphi P_{12} \\ P_{21} &  \varphi [P_{21}Y + Y^*P_{12}] +P_{22} \end{pmatrix}
   = \varphi A_1 + \eta A_2.
\end{split}
\]
 Likewise,
\[
 \sY^* B_2 \sY = \sY^* A_2 \sY = A_2.
\]
 It follows that
\[
 L_A(\rho(x)) = \sY^* L_B(x) \sY
\]
 and thus, as $\sY$ is invertible, $\rho=\rho_\varphi$ is a bianalytic affine linear map from $\cD_B$ to $\cD_A$. Thus, $\rho_\varphi \circ p_\gamma$ is a polynomial automorphism of $\cD_A$.  Finally, since
\[
 \rho_\varphi \circ p_\gamma  (x) = s_\varphi(x),
\]
 the proof of the proposition is complete. 
\end{proof}

The next objective is to establish a converse of Proposition \ref{prop:VA} under some mild additional assumptions on $P_{ij}$ and $Q$.  As a corollary, we produce examples of tuples $A$ and $B$ such that $\cD_A$ and $\cD_B$ are polynomially, but not linearly, bianalytic.

\begin{theorem}
 \label{thm:PQ}
Suppose $\{\Pu^*\Pu, \Pd\Pd^*\}$ is linearly independent and $\cD_Q$ is bounded. 
In this case $\cD_A$ is bounded.

Suppose further that $A$ is eig-generic and $*$-generic and either  $B$ is eig-generic or has size $4$ (the same size as $A$). 
   \begin{enumerate}[\rm (1)]
     \item \label{it:bddpq}  If $p:\cD_A\to\cD_B$ is a polynomial bianalytic map with $p(x)=x+h(x)$, then there is a unimodular $\gamma$ such that, up to unitary equivalence, $B=V_\gamma A$ and 
\[ 
 p= (x_1,x_2 + 2(1-\gamma)x_1^2).
\]
\setcounter{counterPQ}{\theenumi}
\een
Now suppose further that $\{Q,Q^*,\Pu^*\Pu,\Pd\Pd^*\}$ is linearly independent, there is a $c\neq 0$ so that $\Pd^*+c\Pu$
	is not invertible but $\Pd-c\Pu$ is invertible. %

\ben[\rm (1)]
 \setcounter{enumi}{\thecounterPQ}
    \item \label{it:lin indep case}
  If $\{Q,\Po,\Pu^*\Pu,\Pd\Pd^*\}$ is linearly independent, then $\cD_A$ has no non-trivial polynomial automorphisms:  if 
		$q:\cD_A\to \cD_A$ is a bianalytic polynomial, then $q(x)=x$. 
    \item \label{it:lin dep case}
		If $\Po = \alpha_1Q+\alpha_2Q^*+\alpha_3\Pu^*\Pu+\alpha_4\Pd\Pd^*$, then either
		\begin{enumerate}[\rm (a)]
			\item  $\alpha_2\neq0$ and conclusion of item \eqref{it:lin indep case} holds; or
			\item $\alpha_2=0$ in which case $\alpha_3=\alpha_4$ and a polynomial  automorphism $s$ of $\cD_A$ must have the form
	\[
 \begin{split}
	  s=s_\varphi& =  (\varphi x_1,\, -(1-\varphi)\left(4\overline{\alpha_3}\varphi-\alpha_1\right)x_1+x_2 +2(1-\varphi^2)x_1^2)\\
   &  + (\overline{\alpha_3}(1-\varphi), \, -\overline{\alpha_3}(1-\varphi)\left(2\overline{\alpha_3}(1-\varphi)+\alpha_1\right))
\end{split}
\]
for some unimodular $\varphi$.
\een
\een
\end{theorem}

\begin{remark}\rm
 Of course the polynomial automorphisms of $\cD_A$ form a group under composition. In fact, as is straightforward to verify, $s_\varphi \circ s_\psi= s_{\varphi \psi}$. Further, combining items \eqref{it:lin dep case} and \eqref{it:bddpq} of Theorem \ref{thm:PQ}, produces a parameterization of all bianalytic polynomials  $\cD_A$ to $\cD_B$ (under the prevailing assumptions on $A$ and $B$).
 \end{remark}

\begin{example}\rm
As a concrete example, choose
\[
 Q=\begin{pmatrix} 0& 2 \\ \frac12 & 0 \end{pmatrix}.
\]
We note that $xQ+\overline{x}Q$ has both positive and negative eigenvalues for $x\neq 0$, so $\cD_Q$ is bounded. Let
\[
 \Pu = \begin{pmatrix} 1 & 1\\1 & 0\end{pmatrix}, \ \ \Pd = \begin{pmatrix} 2 & -2\\0 & 1\end{pmatrix}, \ \ \Po = I_2,
\]
and writing $A$ as was done above we claim $A$, as described in equation \eqref{eq:thisisA} is eig-generic. Furthermore $\{Q,Q^*,\Pu^*\Pu,\Pd\Pd^*\}$ and 
$\{Q,\Po,\Pu^*\Pu,\Pd\Pd^*\}$ are linearly independent, so Theorem \ref{thm:PQ} applies, thus 
$p(x) = (x_1,x_2+4x_1^2)$ is the unique bianalytic map between $\cD_A$ and $\cD_B$, where $B=V_{-1}A$
and $V_{-1}$ is defined by equation \eqref{eq:thisisVg}. In particular, $\cD_A$ and $\cD_B$ are bounded and polynomially
equivalent, but they are not affine linearly equivalent.

Alternatively, let 
\[
	\Po = \begin{pmatrix}
					10 & -1\\ -1 & 2
				\end{pmatrix} = \Pd\Pd^*+\Pu^*\Pu,
\]
then we have a form for $q(x)$ and a class of affine linear automorphisms of $\cD_A$.

Finally, letting $\Po = 0$, we get our family of automorphisms of $\cD_A$ parameterized by the unimodular complex numbers.
\end{example}

\subsection{The proof of Theorem \ref{thm:PQ}}
Before turning to the proof of Theorem \ref{thm:PQ} proper, we record a few  preliminary results.

\begin{prop}
 \label{prop:1 does it}
	Let $L$ be a linear pencil. If $\cD_L$ is bounded, then  $\cD_L(1)$ 
	is bounded. Conversely, if $\cD_L$ is not bounded, then there exists $\alpha\in \C^g$ 
	such that $t\alpha \in\cD_L(1)$ for all $t\in\R_{>0}$.
\end{prop}

\begin{proof}
 This result is the complex version of the full strength of  \cite[Proposition 2.4]{HKM}. Unfortunately, the statement of the result there is weaker than what is actually proved. Simply note that over the complex numbers, if $T$ is a matrix and  $\langle T\gamma,\gamma\rangle =0$ for all vectors $\gamma$, then, by polarization, $\langle T\gamma,\delta\rangle =0$ for all vectors $\gamma,\delta$ and hence $T=0$. (By comparison, over the real numbers the same conclusion holds provided $T$ is self-adjoint.)  
\end{proof}

\begin{cor}
	\label{cor:posneg eigs}
	Let $L$ be a monic linear pencil with truly linear part $\Lambda.$ Thus, $L(x) = I+\Lambda(x)+\Lambda(x)^*$. 
	The domain $\cD_{L}$ is bounded if and only if $\Lambda(\alpha)$ has both positive 
	and negative eigenvalues for each  $\alpha\in\C^g\setminus\{0\}$. 
\end{cor}

\begin{proof}[Proof of Theorem~\ref{thm:PQ}]
     First observe that independence of $\{\Pu^*\Pu,\Pd\Pd^*\}$ implies independence of $\{\Pu,\Pd^*\}$ 
		since $\Pu =t \Pd^*$ implies $\Pu^*\Pu= |t|^2 \Pd\Pd^*$. 
Let
	\[
          Z=xA_1+\overline{x}A_1^* + yA_2+\overline{y}A_2^* = \begin{pmatrix}0 & M \\ M^* & N\end{pmatrix}.
	\]
 We claim $Z$ has both positive and negative eigenvalues, provided not both $x$ and $y$ are $0$. 

In the case $M\ne 0$, the matrix $Z$ has both positive and negative eigenvalues. Note $M=x P_{12} +\overline{x} P_{21}^*$ and by independence $M=0$ if and only if $x=0$.  In the case $x=0$, (and thus $y\ne 0$), $N= yQ+(yQ)^*.$  Since, by hypothesis,  $\cD_Q$ is bounded, Corollary  \ref{cor:posneg eigs} implies $N$ has both positive and negative eigenvalues. Therefore, once again by Corollary \ref{cor:posneg eigs}, $\cD_A$ is a bounded domain.

To prove item \eqref{it:bddpq}, observe the hypotheses (and they imply $\cD_A$ is bounded) allow the application of  Corollary \ref{cor:main}. 
In particular, there exists a unitary $V$ satisfying  equation \eqref{eq:AVIA}  and, in terms of the tuple $\Xi$ of structure matrices,
\[
 p(x) =x(I-\Lambda_\Xi(x))^{-1}.
\]
   Write $V=(V_{jk})$ as a $2\times 2$ matrix to match the $2\times 2$ block structure of $A$. Straightforward computation gives,
\[
 (V-I)A_2 = \begin{pmatrix} 0 & V_{12} Q \\ 0 & (V_{22}-I)Q\end{pmatrix}.
\]
 Hence
\[
 A_1 (V-I)A_2 = \begin{pmatrix} 0 & \Pu(V_{22}-I)Q \\ 0 & \Pd V_{21}Q \end{pmatrix}.
\]
Since $P_{12}$ and $Q$ are invertible and, by equation \eqref{eq:AVIA},  $A_1(V-I)A_2$ lies in the span of $\{A_1,A_2\}$, it follows that $V_{22}-I=0$. Since $V$ is unitary,  $VV^*=I$. Thus
\[
	\begin{pmatrix} 
		V_{11} & V_{12} \\ 
		V_{21} & I 
	\end{pmatrix}
	\begin{pmatrix} 
		V_{11}^* & V_{21}^* \\ 
		V_{12}^* & I 
	\end{pmatrix}
	=\begin{pmatrix} 
		V_{11}V_{11}^*+V_{12}V_{12}^* & V_{11}V_{21}^*+V_{12} \\ 
		V_{21}V_{11}^*+V_{12}^* & V_{21}V_{21}^*+I 
	\end{pmatrix}
	=	\begin{pmatrix} 
		I & 0 \\ 
		0 & I 
	\end{pmatrix}.
\]
It follows that $V_{21} = 0$ and thus $V_{12}=0$ as well. Finally,
\[
 A_1 (V-I)A_1 = \begin{pmatrix} 0 & 0 \\ 0 & \Pd(V_{11}-I)\Pu \end{pmatrix}.
\]
Hence equation \eqref{eq:AVIA} holds in this case ($j=1=k$) if and only if there is a $\lambda$ such that $\Pd(V_{11}-I)\Pu=\lambda Q$
(note that $A_1(V-I)A_1 = \delta A_1+\lambda A_2$, but $\delta = 0$). 
 Since $\Pd\Pu = -2Q$, it follows that
\[
 V_{11}-I = \lambda \Pd^{-1} Q \Pu^{-1}=-\frac{1}{2}\lambda I.
\]
Thus  $V_{11} = (1-\frac12 \lambda)I$ and $|1-\frac12 \lambda| =1$.  Hence,
\[
 V= V_\gamma = \begin{pmatrix} \gamma I & 0 \\ 0& I \end{pmatrix}
\]
 for some unimodular $\gamma$.  The tuple $\Xi$ of structure matrices and polynomial $p$ are thus described in Proposition \ref{prop:VA}.

Turning to the second part of the theorem, 
suppose $q:\cD_A\to\cD_A$ is a polynomial automorphism. Let $b=q(0)$  and let
	$\cH$ denote the positive square root of $L_A(b)$. By Proposition \ref{prop:dF},
	there exist $F$ and a bianalytic polynomial  $\tilde{q}:\cD_A\to\cD_F$ with
        $\tilde{q}(x) = (-b+q(x))q^{\prime}(0)^{-1}$, such that
	$\tilde{q}(0) =0 $, $\tilde{q}'(0) = I$,
	\[
		A_j= (q'(0)^{-1})_{j,1}\cH F_1\cH +\dots+(q'(0)^{-1})_{j,g}\cH F_g\cH
	\]
	and
	\[
		F_j= q'(0)_{j,1}\cH^{-1}A_1\cH^{-1} +\dots+q'(0)_{j,g}\cH^{-1}A_g\cH^{-1}.
	\]

	Now $F$ is the same size as $A$ and $A$ is eig-generic and $A^*$ is $*$-generic,  hence we can apply item \eqref{it:bddpq} to the bianalytic polynomial  $\tilde{q}:\cD_A\to \cD_F$.  In particular, there is a unimodular $\gamma$ such that $F=V_\gamma A$ and $\tilde{q}=(x_1, x_2+2(1-\gamma)x_1^2)$.  By Proposition \ref{prop:dF},
\[
	A_i = \sum_{j=1}^g \left(q^\prime(0)^{-1}\right)_{i,j}\cH F_j \cH.
\]
Since $F_j = \sV^* V A_j \sV$, 
\[
	\sH^{-*}A_i \sH^{-1} = \sum_{j=1}^g \left(q^\prime(0)^{-1}\right)_{i,j}V A_j,
\]
where $\sH=\sV \cH$. Setting 
\[
	q^\prime(0)^{-1} = 
	\begin{pmatrix}
		\lambda_1 & \mu_1\\
		\lambda_2 & \mu_2
	\end{pmatrix},
\]
gives
\begin{equation}
\label{eq:B1}
 \sH^{-*} A_1 \sH^{-1}  =  \sH^{-*} \begin{pmatrix} 0 & \Pu \\ \Pd & \Po \end{pmatrix}\sH^{-1}  =  \begin{pmatrix} 0 & \lambda_1 \gamma \Pu \\ \lambda_1 \Pd & \mu_1 Q+\lambda_1\Po \end{pmatrix} 
\end{equation}
 and likewise
\begin{equation}
 \label{eq:B2}
 \sH^{-*} A_2 \sH^{-1} = \sH^{-*} \begin{pmatrix} 0 & 0\\ 0 & Q \end{pmatrix}\sH^{-1} = \begin{pmatrix} 0 & \lambda_2 \gamma \Pu \\ \lambda_2 \Pd & \mu_2 Q+\lambda_2\Po\end{pmatrix}. 
\end{equation}
By equation \eqref{eq:B1}, $\lambda_1\neq 0$ since $A_1$ is invertible. 

Let $Y=\sH^{-1}$ and write $Y=(Y_{j,k})$ in the obvious way as a $2\times 2$ matrix with $2\times 2$ block entries.  From equation \eqref{eq:B2}, 
\[
 \begin{pmatrix} Y_{21}^* Q Y_{21} &  Y_{21}^* Q Y_{12} \\ * & Y_{22}^* Q Y_{22} \end{pmatrix} = \begin{pmatrix} 0 & \lambda_2 \gamma \Pu \\ \lambda_2 \Pd & \mu_2 Q+\lambda_2\Po\end{pmatrix}. 
\]
 Thus, as $Q$ is invertible, $Y_{21}=0$ and therefore $\lambda_2=0$.  We also record $Y_{22}^* Q Y_{22} =\mu_2 Q$ or equivalently $Y_{22}^* \Pd\Pu Y_{22}=\mu_2 \Pd\Pu$.
 Turning to equation \eqref{eq:B1},
\[
 \begin{pmatrix} 0 & Y_{11}^* \Pu Y_{22} \\ Y_{22}^* \Pd Y_{11} & Y_{22}^* \Pd Y_{12} +  Y_{12}^* \Pu Y_{22}+Y_{22}^*\Po Y_{22} \end{pmatrix}
  = \begin{pmatrix} 0 & \lambda_1 \gamma P \\ \lambda_1 P & \mu_1 Q+\lambda_1\Po\end{pmatrix}. 
\]
 Hence,
\begin{equation}
 \label{eq:lots}
 \begin{split}
   \lambda_2& =  0\\
     Y_{22}^* \Pd\Pu Y_{22}& =  \mu_2 \Pd\Pu \\
    Y_{11}^* \Pu Y_{22}& =  \lambda_1 \gamma \Pu \\
     Y_{22}^* \Pd Y_{11}& =  \lambda_1 \Pd \\
    Y_{22}^* \Pd Y_{12} +  Y_{12}^* \Pu Y_{22}+Y_{22}^*\Po Y_{22}& =  \mu_1 Q+\lambda_1 \Po.
 \end{split}
\end{equation}
Taking determinants in the second of these equations gives $|\det(Y_{22})|^2 = \mu_2^2$ and therefore $\mu_2$ is real. Taking determinants in the third and fourth equations gives $\overline{\lambda_1^2} = \lambda_1^2 \gamma^2$. Thus,
\[
    \overline{\lambda_1} = \pm \lambda_1 \gamma. 
\]
In particular,
\begin{equation}
 \label{eq:lambda-gamma}
   |\overline{\lambda_1}|^2  = \pm \lambda_1^2  \gamma. 
\end{equation}

 Multiplying the third equation on the left by the fourth equation gives
\[
  Y_{22}^* \Pd Y_{11} Y_{11}^* \Pu Y_{22} = \lambda_1^2 \gamma \Pd\Pu. 
\]
 Using the second equation
\[
 Y_{22}^* \Pd Y_{11} Y_{11}^* \Pu Y_{22} = \frac{ \lambda_1^2 \gamma}{\mu_2} Y_{22}^*  \Pd\Pu Y_{22}.
\]
Since $Y_{22}^*\Pd$ and $\Pu Y_{22}$ are invertible,
\[
 Y_{11} Y_{11}^* = \frac{\lambda_1^2 \gamma}{\mu_2} >0.
\]
 In particular, $Y_{11}$ is a multiple of a unitary.

 Next multiply the fourth equation by its adjoint on the right to obtain
\begin{equation*}
  |\lambda_1|^2 \Pd\Pd^* =  Y_{22}^* \Pd Y_{11} (Y_{22}^* \Pd Y_{11})^* =\frac{\lambda_1^2 \gamma}{\mu_2} Y_{22}^* \Pd\Pd^* Y_{22}.
\end{equation*}
 Multiplying the third equation by its adjoint on the left gives (as $|\gamma|=1$)
\begin{equation*}
  |\lambda_1|^2 \Pu^*\Pu =   (Y_{11}^* \Pu Y_{22})^*  Y_{11}^* \Pu Y_{22} = \frac{\lambda_1^2 \gamma}{\mu_2} Y_{22}^* \Pu^* \Pu Y_{22}.
\end{equation*}

In view of equation \eqref{eq:lambda-gamma}, if $\mu_2>0$, then $\lambda_1^2\gamma =|\lambda_1|^2$ and if $\mu_2<0$, then $-\lambda_1^2\gamma=|\lambda_1|^2$. Hence,  with $\abs{\kappa}^2 = |\mu_2|$ and $Z=\kappa Y_{22}$ either $\mu_2>0$ and 
\begin{equation*}
 \begin{aligned}
  Z^* \Pd\Pu Z & =   \Pd\Pu   \\
  Z^* \Pu^*\Pd^* Z& = \Pu^*\Pd^* \\
\end{aligned} \qquad 
\begin{aligned}
  Z^* \Pd\Pd^*  Z& =  \Pd\Pd^*  \\
  Z^* \Pu^*\Pu  Z& =  \Pu^* \Pu \\
 \end{aligned}
\end{equation*}
 or $\mu_2<0$ and 
\begin{equation*}
 \begin{aligned}
  Z^* \Pd\Pu Z& =  -\Pd\Pu \\
  Z^* \Pu^*\Pd^* Z& =  -\Pu^*\Pd^* \\
\end{aligned} \qquad
  \begin{aligned}
  Z^* \Pd\Pd^*  Z& =  \Pd\Pd^* \\
  Z^* \Pu^*\Pu  Z& =  \Pu^* \Pu. \\
  \end{aligned}
\end{equation*}
We will argue that this second case does not occur. Recall we are assuming $\Pd$ and $\Pu$ are both invertible.
 This implies $Z$ is invertible.
Observe that, assuming this second set of equations, for complex numbers $c$,
\begin{equation}
 \label{eq:pmc}
 Z^*((\Pu +c\Pd^*)^*(\Pu+c\Pd^*))Z = (\Pu - c\Pd^*)^*(\Pu-c\Pd^*).
\end{equation}

 By assumption there is a $c\ne 0$ such that $\Pd^* + c\Pu$ is not invertible but $\Pd^* - c\Pu$ is invertible, leading  to the contradiction that the left hand side of equation \eqref{eq:pmc} is invertible, but the right hand side is not.
It follows that $\mu_2>0$ and $\lambda_1^2 \gamma = |\lambda_1|^2.$

 Assuming  $\{Q,Q^*, \Pd\Pd^*, \Pu^* \Pu \}$ is linearly independent, this set  spans the $2\times 2$ matrices. Hence
(using the fact that $A^*XA=X$ for all $2\times 2$ matrices $X$
implies $A$ is a multiple of the identity) %
 $Y_{22}=\kappa I$ for some $\kappa$ with $\abs{\kappa}^2=\mu_2$. Hence  many of the identities in equation \eqref{eq:lots} now imply that $Y_{11}$ is also a multiple of the identity. For instance, using the third equality, 
\[
 \lambda_1 \Pd =   Y_{22}^* \Pd Y_{11} = \overline{\kappa} \Pd Y_{11}
\]
 and hence $Y_{11} = \frac{\lambda_1}{\overline{\kappa}} I$.

  Thus, 
\[
 \sH^{-1}=Y =\begin{pmatrix} \frac{\lambda_1}{\overline{\kappa}}I & Y_{12} \\ 0 & \kappa I \end{pmatrix}
\]
 and consequently
\[
 \sV \calH= \sH =\begin{pmatrix} \frac{\overline{\kappa}}{\lambda_1}& - \frac{\overline{\kappa}}{\kappa\lambda_1} Y_{12} \\ 0 & \frac{1}{\kappa} \end{pmatrix}.
\]
 It follows that
\[
	\calH^2 = \sH^* \sH  = \begin{pmatrix} \frac{\abs{\kappa}^2}{\abs{\lambda_1}^2}I & -\frac{\bar{\kappa}}{\abs{\lambda_1}^2} Y_{12}  
		\\  -\frac{\kappa}{\abs{\lambda_1}^2}Y_{12}^* & \frac{1}{\abs{\kappa}^2}I + \frac{1}{\abs{\lambda_1}^2} Y_{12}^* Y_{12} \end{pmatrix}
\]
On the other hand,
\begin{align*}
	\calH^2 = L_A(b) &=(I+\sum b_j A_j +\sum (b_j A_j)^*) \\
	&= \begin{pmatrix} 
			I & b_1\Pu + \overline{b_1}\Pd^* 
			\\b_1 \Pd + \overline{b_1}\Pu^* & I+ b_2 Q + \overline{b_2} Q^* +b_1 \Po+\overline{b_1}\Po^*
		\end{pmatrix}.
\end{align*}
 It follows that,
\begin{equation}
 \label{eq:YP}
 \begin{split}
	\frac{\abs{\kappa}^2}{\abs{\lambda_1}^2}& =1\\
	-\frac{\bar{\kappa}}{\abs{\lambda_1}^2} Y_{12}& =   b_1\Pu + \overline{b_1}  \Pd^* \\
  \frac{1}{\abs{\kappa}^2}I + \frac{1}{\abs{\lambda_1}^2} Y_{12}^* Y_{12}& =  I + b_2 Q + \overline{b_2} Q^* +b_1 \Po+\overline{b_1}\Po^*.
\end{split}
\end{equation}
 Note that combining the first two of these equations gives,
\begin{equation}
 \label{eq:YP+}
   Y_{12} = - \kappa (b_1\Pu + \overline{b_1}  \Pd^*).
\end{equation}
Since $Y_{22}=\kappa I$, the last equality in equation \eqref{eq:lots} gives
\[
	\overline{\kappa}\Pd Y_{12} + \kappa Y_{12}^*\Pu + \abs{\kappa}^2\Po  = \mu_1 Q+\lambda_1 \Po.
\]
It follows, using the second equality in equation \eqref{eq:YP},
\[
\begin{split}
   \mu_1 Q+(\lambda_1-\abs{\kappa}^2) \Po
		=&-\abs{\lambda_1}^2\Pd(b_1\Pu + \overline{b_1}\Pd^*) - \abs{\lambda_1}^2(\overline{b_1}\Pu^*+b_1\Pd)\Pu\\
		=&- \abs{\lambda_1}^2\left(b_1\Pd\Pu + \overline{b_1}\Pd\Pd^* + \overline{b_1}\Pu^*\Pu+b_1\Pd\Pu\right).
 \end{split}
\]
Simplifying with $\Pd\Pu=-2Q$ and bringing to one side gives
\begin{equation}
 \label{eq:P lin com}
	0=(\mu_1+4b_1\abs{\lambda_1}^2)Q+\overline{b_1}\abs{\lambda_1}^2(\Pd\Pd^*+\Pu^*\Pu) + (\lambda_1-\abs{\kappa}^2)\Po.
\end{equation}

We now proceed to prove item \eqref{it:lin indep case}. Assuming  $\{Q,\Pu^*\Pu,\Pd\Pd^*,\Po\}$ 
is linearly independent, equation \eqref{eq:P lin com} and the fact that $\lambda_1\neq 0$ 
yields $b_1=0$. So $\mu_1=0$ and $\lambda_1 = \abs{\kappa}^2$. 
Furthermore, $\abs{\kappa} = \abs{\lambda_1}$, implies $\lambda_1 = \abs{\lambda_1}$. Hence $\lambda_1=1$, $\abs{\kappa}=1$ and $\gamma =1.$  It also  follows that $Y_{12} =0$ by the third equation in \eqref{eq:YP}. Furthermore, 
\[
	I = I + b_2Q+\overline{b_2}Q^*,
\]
so $b_2=0$, as $\{Q,Q^*\}$ is linearly independent. Hence $L_A(b)=I=\calH$ and 
\[
 \sV = \overline{\kappa}I_4.
\]
Finally, $Y_{12}=0$ also implies $\mu_2 = 1.$ Thus, $F=A$ and $V=V_\gamma =I$, $q(0)=0$ and finally, $q^\prime(0)= I$ too.  Hence, $q(x)=x$ and the  proof of item \eqref{it:lin indep case} is complete.

Moving on to item \eqref{it:lin dep case}, assume now
\[
	\Po = \alpha_1Q+\alpha_2Q^*+\alpha_3\Pu^*\Pu+\alpha_4\Pd\Pd^*.
\]
If $\alpha_2\neq 0$ then $\{Q,\Pu^*\Pu,\Pd\Pd^*,\Po\}$ must also be linearly independent, and hence
the conclusions of  item \eqref{it:lin indep case} hold.

To complete the proof of the theorem,  suppose $\alpha_2 =0$ and recall equation \eqref{eq:P lin com}, 
\begin{align*}
	0& =(\mu_1+4b_1\abs{\lambda_1}^2+\alpha_1(\lambda_1-\abs{\kappa}^2))Q\\
	&\phantom{=}+(\overline{b_1}\abs{\lambda_1}^2+\alpha_3(\lambda_1-\abs{\kappa}^2))\Pu^*\Pu
	+(\overline{b_1}\abs{\lambda_1}^2+\alpha_4(\lambda_1-\abs{\kappa}^2))\Pd\Pd^*.
\end{align*}
Since $\{Q,\Pu^*\Pu,\Pd\Pd^*\}$ is linearly independent, 
\begin{equation*}
 \begin{split}
	\mu_1+4b_1\abs{\lambda_1}^2+\alpha_1(\lambda_1-\abs{\lambda_1}^2)&=0\\
  \overline{b_1}\abs{\lambda_1}^2+\alpha_3(\lambda_1-\abs{\lambda_1}^2)&=0\\
	\overline{b_1}\abs{\lambda_1}^2+\alpha_4(\lambda_1-\abs{\lambda_1}^2)&=0.
\end{split}
\end{equation*}
It follows that $\alpha_3=\alpha_4$ and
\begin{equation}
 \label{eq:b1}
	b_1 = \frac{\overline{\alpha_3}(\lambda_1-1)}{\lambda_1}.
\end{equation}
Now, using equation \eqref{eq:YP+}  and looking at the third equation in equation  \eqref{eq:YP},
\[
\begin{split}
	&\frac{1}{\abs{\kappa}^2}I + \frac{\abs{\kappa}^2}{\abs{\lambda_1}^2}(\abs{b_1}^2\Pu^*\Pu+2b_1^2Q+2\overline{b_1}^2Q^*  +\abs{b_1}^2\Pd\Pd^*)\\
	& {}\qquad{} = I + b_2 Q + \overline{b_2} Q^* +b_1 \Po+\overline{b_1}\Po^*.
\end{split}
\]
Using $\Po = \alpha_1 Q+\alpha_3(\Pu^*\Pu+\Pd\Pd^*)$, 
\begin{equation}
 \label{eq:YComb}
 \begin{split}
	0=\left (1-\frac{1}{\abs{\kappa}^2}\right )I & +(b_2-2b_1^2+\alpha_1b_1)Q+(\overline{b_2}-2\overline{b_1}^2+\overline{\alpha_1}\overline{b_1})Q^*\\
		&+(b_1\alpha_3+\overline{b_1}\overline{\alpha_3}-\abs{b_1}^2)(\Pu^*\Pu+\Pd\Pd^*).
\end{split}
\end{equation}
Using equation \eqref{eq:b1}, 
\begin{align*}
	b_1\alpha_3+\overline{b_1}\overline{\alpha_3}-\abs{b_1}^2 
		&= \abs{\alpha_3}^2\left(\frac{\lambda_1-1}{\lambda_1}+\frac{\overline{\lambda_1}-1}{\overline{\lambda_1}}
		-\frac{(\lambda_1-1)(\overline{\lambda_1}-1)}{\lambda_1\overline{\lambda_1}}\right)\\
		&= \abs{\alpha_3}^2\left(1-\frac{1}{\abs{\lambda_1}^2}\right) = \abs{\alpha_3}^2\left(1-\frac{1}{\abs{\kappa}^2}\right).
\end{align*}
Let $z = (b_2-2b_1^2+\alpha_1b_1)$, solving for the $Q$ and $Q^*$ terms, equation \eqref{eq:YComb} becomes
\[
	zQ+\overline{z}Q^*
		=\left(1-\frac{1}{\abs{\kappa}^2}\right)\left(I+\abs{\alpha_3}^2(\Pu^*\Pu+\Pd\Pd^*)\right)
\]
Write $C = (1-\abs{\kappa}^{-2}),$ let $t\in\mathbb{R}$ with $tC>0$ and consider
\[
	L_Q(tz) = I+tzQ+t\overline{z}Q^* = (1+tC)I+\abs{\alpha_3}^2tC(\Pu^*\Pu+\Pd\Pd^*).
\]
But $\Pu^*\Pu,\Pd\Pd^*\succeq 0$, so $L_Q(tz)\succeq 0$ for all $t$
with $tC>0,$ contradicting the boundedness of  $\cD_Q.$ Hence both $z=0$ and $C(I+\abs{\alpha_3}^2(\Pu^*\Pu+\Pd\Pd^*))=0$.
So either $C=0$ or $I=-\abs{\alpha_3}^2(\Pu^*\Pu+\Pd\Pd^*)$. However $I\succeq0$ while 
$-\abs{\alpha_3}^2(\Pu^*\Pu+\Pd\Pd^*)\preceq 0$, hence this second equality never holds. 
Thus $C=0$.

It follows that $\abs{\kappa}^2=\abs{\lambda_1}^2=\mu_2=1$, so 
\begin{equation*}
	\begin{aligned}
		b_1 &= \overline{\alpha_3}(1-\overline{\lambda_1})\\
		b_2 &= \overline{\alpha_3}(1-\overline{\lambda_1})\left(2\overline{\alpha_3}(1-\overline{\lambda_1})-\alpha_1\right)\\
    \end{aligned} \qquad
    \begin{aligned}
		\mu_1 &= -\lambda_1(1-\overline{\lambda_1})\left(4\overline{\alpha_3}\overline{\lambda_1}+\alpha_1\right)\\
		Y_{12} &= \kappa(\overline{\alpha_3}(1-\overline{\lambda_1})P_{12}-\alpha_3(1-\lambda_1)P_{21}^*),
	\end{aligned}
\end{equation*}
and
\[
			F_1 = \calH^{-1}\left(\overline{\lambda_1}B_1
				+(1-\overline{\lambda_1})(4\overline{\alpha_3}\overline{\lambda_1}+\alpha_1)B_2\right)\calH^{-1}, 
				\ \  \ \ F_2 = \calH^{-1}B_2\calH^{-1}.
\]
Recall,
\[
	q(0) = (b_1,b_2), \ \ 
	q^\prime(0) = \begin{pmatrix} \lambda_1 & \mu_1\\ 0 & 1 \end{pmatrix}^{-1}
	=\begin{pmatrix} \overline{\lambda_1} & -\overline{\lambda_1}\mu_1\\ 0 & 1 \end{pmatrix},
\]
so plugging in;
\[
				q(0) =\begin{pmatrix}\overline{\alpha_3}\left(1-\overline{\lambda_1}\right),&
					\overline{\alpha_3}\left(1-\overline{\lambda_1}\right)\left(2\overline{\alpha_3}(1-\overline{\lambda_1})-\alpha_1\right)
				\end{pmatrix},
\]
and
\[	
				q^\prime(0) = 
				\begin{pmatrix}
				\overline{\lambda_1} & (1-\overline{\lambda_1})(4\overline{\alpha_3}\overline{\lambda_1}+\alpha_1)\\
				0 & 1
				\end{pmatrix}.
\]

Next, we know that $\ell(x) = (-b+x)q^\prime(0)^{-1}$ and $\ell^{-1}(x) = xq^\prime(0)+b$, , so yet again
plugging in;
\begin{align*}
	\ell(x) &= (-\lambda_1q(0)_1+\lambda_1x_1,\lambda_1q^\prime(0)_{1,2}-\lambda_1q^\prime(0)_{1,2}x_1+x_2),\\
	\ell^{-1}(x) &= \left(q(0)_1+\overline{\lambda_1}x_1,q(0)_2+q^\prime(0)_{1,2}x_1+x_2\right).
\end{align*}
Using the fact that $q = \ell^{-1}\circ \tilde{q},$ 
\[
	q(x) = \ell^{-1}\left(\tilde{q}(x)\right) 
	= \begin{pmatrix}q(0)_1+\overline{\lambda_1}x_1, & q(0)_2+q^\prime(0)_{1,2}x_1+x_2+2(1-\overline{\lambda_1}^2)x_1^2\end{pmatrix}.
\]
Observe $q=q_{\overline{\lambda_1}}$ i.e. $q$ depends upon the choice of the unimodular $\overline{\lambda_1}$.
Thus, taking a unimodular $\phi$ and setting $s_\phi = q_\phi$,
\[
	s_\phi^1(x) = \overline{\alpha_3}(1-\phi)+\phi x_1
\]
\[
	s_\phi^2(x)
	=-\overline{\alpha_3}(1-\phi)\left(2\overline{\alpha_3}(1-\phi)+\alpha_1\right)
	-(1-\phi)\left(4\overline{\alpha_3}\phi-\alpha_1\right)x_1+x_2
	+2(1-\phi^2)x_1^2,
\]
which by construction is an automorphism of $\cD_A$. Moreover, if $\psi$ is another unimodular, then
\[
	s_\phi\circ s_\psi = s_{\phi \psi}
\]

These automorphisms must be the only automorphisms of $\cD_A$, since if there were some other form of automorphism
then by composing with $q$ we would get a different form for a bianalytic polynomial  from $\cD_A$ to $\cD_A$ which
cannot happen. 
\end{proof}


\appendix

\def\ione{(1-x_1)^{-1}}
\def\itwo{(1-x_2)^{-1}}
\def\i3{(1-x_3)^{-1}}

\def\iplusone{(1+x_1)^{-1}}
\def\iplustwo{(1+x_2)^{-1}}
\def\iplus3{(1+x_3)^{-1}}

\def\iyOne{(1+y_1)^{-1} }
\def\iyTwo{ (1+y_2)^{-1} }

\def\iy3{(1+y_3)^{-1} }

\newpage

\section{\Ct maps in 3 variables}
\label{sec:appendix}
This is a list of all \ct maps coming from indecomposable 
3 dimensional algebras over $\C$.
These maps were generated in  a Mathematica Notebook
running under the package NCAlgebra.
The notebook, given defining relations 
on a basis
for a $g$-dimensional algebra,
produces the associated \ct maps $p$ and their
inverses $p^{-1}$. The notebook was prepared by Eric Evert
with help from Zongling Jiang (graduate students at UCSD)
and Shiyuan Huang and Ashwin Trisal (undergraduates)
in consultation with this papers' authors.

\linespread{1.1}

\addtocontents{toc}{\SkipTocEntry}\subsection*{Algebra 1}

Relations on a basis $R_1, R_2, R_3$:
\[
R_1 R_1=0, \quad R_2 R_2=0, \quad R_3 R_3=0,
\]
\[
R_1 R_2=0, \quad R_2 R_1=0, \quad R_1 R_3=R_2,
\quad
R_3 R_1=R_2, \quad R_2 R_3=0, \quad R_3 R_2=0.
\]

Structure matrices
{{\footnotesize
\[  \Xi=
\left(
\begin{pmatrix}
0 & 0 & 0 \\
0 & 0 & 0 \\
0 & 1 & 0 
\end{pmatrix},
\begin{pmatrix}
0 & 0 & 0 \\
0 & 0 & 0 \\
0 & 0 & 0 
\end{pmatrix},
\begin{pmatrix}
0 & 1 & 0 \\
0 & 0 & 0 \\
0 & 0 & 0 
\end{pmatrix}
\right)
\]
}}
$$
p(x)=(x_1, \   x_2 + x_1  x_3 + x_3  x_1 ,\  x_3)
\qquad 
p^{-1} (y)= ( y_1,\  y_2 - y_1  y_3 - y_3  y_1 , \ y_3 )
$$

\addtocontents{toc}{\SkipTocEntry}\subsection*{Algebra 2}

Let $\alpha$ be a real number. Note the case $\alpha=1 $
is Algebra 1.

Relations on a basis $R_1, R_2, R_3$:
\[
R_1 R_1=0, \quad R_2 R_2=0, \quad R_3 R_3=0,
\]
\[
R_1 R_2=0, \quad R_2 R_1=0, \quad R_1 R_3=R_2,
\quad
R_3 R_1= \alpha R_2, \quad R_2 R_3=R_2, \quad R_3 R_2=0.
\]

Structure matrices
{{\footnotesize
\[  \Xi=
\left(
\begin{pmatrix}
0 & 0 & 0 \\
0 & 0 & 0 \\
0 & \alpha & 0 
\end{pmatrix},
\begin{pmatrix}
0 & 0 & 0 \\
0 & 0 & 0 \\
0 & 0 & 0 
\end{pmatrix},
\begin{pmatrix}
0 & 1 & 0 \\
0 & 0 & 0 \\
0 & 0 & 0 
\end{pmatrix}
\right)
\]
}}
$$
p(x)=(x_1, \   x_2  + x_1  x_3 + \alpha x_3  x_1, \ x_3)
\qquad
p^{-1}(y) = 
( y_1, \    y_2 - y_1   y_3 - \alpha  y_3   y_1 ,  \  y_3)
$$

\addtocontents{toc}{\SkipTocEntry}\subsection*{Algebra 3}

Relations on a basis $R_1, R_2, R_3$:
\[
R_1 R_1=R_2, \quad R_2 R_2=0, \quad R_3 R_3=0,
\]
\[
R_1 R_2=R_3, \quad R_2 R_1=R_3, \quad R_1 R_3=0,
\quad
R_3 R_1= 0, \quad R_2 R_3=0, \quad R_3 R_2=0.
\]

Structure matrices
{{\footnotesize
\[   \Xi=
\left(
\begin{pmatrix}
0 & 1 & 0 \\
0 & 0 & 1 \\
0 & 0 & 0 
\end{pmatrix},
\begin{pmatrix}
0 & 0 & 1 \\
0 & 0 & 0 \\
0 & 0 & 0 
\end{pmatrix},
\begin{pmatrix}
0 & 0 & 0 \\
0 & 0 & 0 \\
0 & 0 & 0 
\end{pmatrix}
\right)
\]
}}
\[
\begin{split}
p(x)&=(x_1, \   x_2 + x_1  x_1, \    x_3 +
 x_1  (x_1^2  + x_2) + x_2  x_1)\\
p^{-1}(y)&=
( y_1, \   y_2 - y_1^2 , \   y_3 +
  y_1  ( y_1^2 -  y_2) -  y_2   y_1 )
\end{split}
\]

\addtocontents{toc}{\SkipTocEntry}\subsection*{Algebra 4}

Relations on a basis $R_1, R_2, R_3$:
\[
R_1 R_1=0, \quad R_2 R_2=0, \quad R_3 R_3=R_3,
\]
\[
R_1 R_2=0, \quad R_2 R_1=0, \quad R_1 R_3=R_2,
\quad
R_3 R_1= 0, \quad R_2 R_3=R_2, \quad R_3 R_2=0.
\]

Structure matrices
{{\footnotesize
\[  \Xi=
\left(
\begin{pmatrix}
0 & 0 & 0 \\
0 & 0 & 0 \\
0 & 0 & 0 
\end{pmatrix},
\begin{pmatrix}
0 & 0 & 0 \\
0 & 0 & 0 \\
0 & 0 & 0 
\end{pmatrix},
\begin{pmatrix}
0 & 1 & 0 \\
0 & 1 & 0 \\
0 & 0 & 1 
\end{pmatrix}
\right)
\]
}}
\[
\begin{split}
p(x)&=(x_1, \ ( x_2 + x_1  x_3 )  \i3, \  x_3  \i3)\\
p^{-1}(y) &=
 ( y_1, \ ( y_2  - y_1    y_3   )    \iy3,  \ y_3    \iy3) 
 \end{split}
 \]

\addtocontents{toc}{\SkipTocEntry}\subsection*{Algebra 5}

Relations on a basis $R_1, R_2, R_3$:
\[
R_1 R_1=0, \quad R_2 R_2=0, \quad R_3 R_3=R_3,
\]
\[
R_1 R_2=0, \quad R_2 R_1=0, \quad R_1 R_3=0,
\quad
R_3 R_1= R_1, \quad R_2 R_3=R_2, \quad R_3 R_2=0.
\]

Structure matrices
{{\footnotesize
\[   \Xi=
\left(
\begin{pmatrix}
0 & 0 & 0 \\
0 & 0 & 0 \\
1 & 0 & 0 
\end{pmatrix},
\begin{pmatrix}
0 & 0 & 0 \\
0 & 0 & 0 \\
0 & 0 & 0 
\end{pmatrix},
\begin{pmatrix}
0 & 0 & 0 \\
0 & 1 & 0 \\
0 & 0 & 1 
\end{pmatrix}
\right)
\]
}}
\[
\begin{split}
p(x)&=(\i3 x_1, \ x_2 \i3, x_3 \i3)\\
p^{-1}(y)&=
( {\iy3   y_1, \  y_2    \iy3, \ y_3    \iy3})
\end{split}
\]

\addtocontents{toc}{\SkipTocEntry}\subsection*{Algebra 6}

Relations on a basis $R_1, R_2, R_3$:
\[
R_1 R_1=0, \quad R_2 R_2=0, \quad R_3 R_3=R_3,
\]
\[
R_1 R_2=0, \quad R_2 R_1=0, \quad R_1 R_3=0,
\quad
R_3 R_1= R_2, \quad R_2 R_3=0, \quad R_3 R_2=R_2.
\]

Structure matrices
{{\footnotesize
\[  \Xi=
\left(
\begin{pmatrix}
0 & 0 & 0 \\
0 & 0 & 0 \\
0 & 1 & 0 
\end{pmatrix},
\begin{pmatrix}
0 & 0 & 0 \\
0 & 0 & 0 \\
0 & 1 & 0 
\end{pmatrix},
\begin{pmatrix}
0 & 0 & 0 \\
0 & 0 & 0 \\
0 & 0 & 1 
\end{pmatrix}
\right)
\]
}}
\[
\begin{split}
p(x)&=(x_1,  \  \i3  ( x_2  +   x_3 x_1 )  , \ x_3  \i3)\\
p^{-1} (y)&=
 ( {y_1, \ \iy3   ( y_2  - y_3      y_1) , \ y_3    \iy3})
 \end{split}
 \]

\addtocontents{toc}{\SkipTocEntry}\subsection*{Algebra 7}

Relations on a basis $R_1, R_2, R_3$:
\[
R_1 R_1=0, \quad R_2 R_2=R_2, \quad R_3 R_3=R_3,
\]
\[
R_1 R_2=R_1, \quad R_2 R_1=0, \quad R_1 R_3=0,
\quad
R_3 R_1= R_1, \quad R_2 R_3=0, \quad R_3 R_2=0.
\]

Structure matrices
{{\footnotesize
\[  \Xi=
\left(
\begin{pmatrix}
0 & 0 & 0 \\
0 & 0 & 0 \\
1 & 0 & 0 
\end{pmatrix},
\begin{pmatrix}
1 & 0 & 0 \\
0 & 1 & 0 \\
0 & 0 & 0 
\end{pmatrix},
\begin{pmatrix}
0 & 0 & 0 \\
0 & 0 & 0 \\
0 & 0 & 1 
\end{pmatrix}
\right)
\]
}}
\[
\begin{split}
p(x)&= (\i3  x_1  , \ x_2  , \  x_3  \i3)\\
p^{-1} (y)& ={\iy3   y_1   \iyTwo , 1 -  \iyTwo , y_3    \iy3}
\end{split}
\]

\addtocontents{toc}{\SkipTocEntry}\subsection*{Algebra 8}

Relations on a basis $R_1, R_2, R_3$:
\[
R_1 R_1=0, \quad R_2 R_2=0, \quad R_3 R_3=R_3,
\quad
R_1 R_2=0, \quad R_2 R_1=0, \quad R_1 R_3=R_1,
\]
\[
R_3 R_1= R_1, \quad R_2 R_3=R_2, \quad R_3 R_2=0.
\]

Structure matrices
{{\footnotesize
\[  \Xi=
\left(
\begin{pmatrix}
0 & 0 & 0 \\
0 & 0 & 0 \\
1 & 0 & 0 
\end{pmatrix},
\begin{pmatrix}
0 & 0 & 0 \\
0 & 0 & 0 \\
0 & 0 & 0 
\end{pmatrix},
\begin{pmatrix}
1 & 0 & 0 \\
0 & 1 & 0 \\
0 & 0 & 1 
\end{pmatrix}
\right)
\]
}}
\[
\begin{split}
p(x)&=(\i3  x_1  \i3, x_2  \i3, \i3 x_3)\\
p^{-1} (y) &= 
({\iy3   y_1   \iy3, \  y_2    \iy3, \ y_3    \iy3} )
\end{split}
\]

\addtocontents{toc}{\SkipTocEntry}\subsection*{Algebra 9}

Relations on a basis $R_1, R_2, R_3$:
\[
R_1 R_1=0, \quad R_2 R_2=0, \quad R_3 R_3=R_3,
\]
\[
R_1 R_2=0, \quad R_2 R_1=0, \quad R_1 R_3=0,
\quad
R_3 R_1= R_1, \quad R_2 R_3=R_2, \quad R_3 R_2=R_2.
\]

Structure matrices
{{\footnotesize
\[  \Xi=
\left(
\begin{pmatrix}
0 & 0 & 0 \\
0 & 0 & 0 \\
1 & 0 & 0 
\end{pmatrix},
\begin{pmatrix}
0 & 0 & 0 \\
0 & 0 & 0 \\
0 & 1 & 0 
\end{pmatrix},
\begin{pmatrix}
0 & 0 & 0 \\
0 & 1 & 0 \\
0 & 0 & 1 
\end{pmatrix}
\right)
\]
}}
\[
\begin{split}
p(x)&=(\i3  x_1, \  \i3  x_2  \i3, \  x_3  \i3)\\
p^{-1}(y) &= ({\iy3   y_1, \  \iy3   y_2    \iy3, \ y_3    \iy3})
\end{split}
\]

\addtocontents{toc}{\SkipTocEntry}\subsection*{Algebra 10}

Relations on a basis $R_1, R_2, R_3$:
\[
R_1 R_1=0, \quad R_2 R_2=0, \quad R_3 R_3=R_3,
\]
\[
R_1 R_2=0, \quad R_2 R_1=0, \quad R_1 R_3=R_1,
\quad
R_3 R_1= R_1, \quad R_2 R_3=R_2, \quad R_3 R_2=R_2.
\]
Structure matrices
{{\footnotesize
\[  \Xi=
\left(
\begin{pmatrix}
0 & 0 & 0 \\
0 & 0 & 0 \\
1 & 0 & 0 
\end{pmatrix},
\begin{pmatrix}
0 & 0 & 0 \\
0 & 0 & 0 \\
0 & 1 & 0 
\end{pmatrix},
\begin{pmatrix}
1 & 0 & 0 \\
0 & 1 & 0 \\
0 & 0 & 1 
\end{pmatrix}
\right)
\]
}}
\[
\begin{split}
p(x)&=(\i3  x_1  \i3, \i3  x_2  \i3, x_3  \i3) \\
p^{-1} (y)&=
({\iy3   y_1   \iy3,  \ \iy3   y_2    \iy3, \ y_3    \iy3})
\end{split}
\]

\addtocontents{toc}{\SkipTocEntry}\subsection*{Algebra 11}

Relations on a basis $R_1, R_2, R_3$:
\[
R_1 R_1=0, \quad R_2 R_2=0, \quad R_3 R_3=R_3,
\]
\[
R_1 R_2=0, \quad R_2 R_1=0, \quad R_1 R_3=R_2,
\quad
R_3 R_1= R_2, \quad R_2 R_3=R_2, \quad R_3 R_2=R_2.
\]

Structure matrices
{{\footnotesize
\[  \Xi=
\left(
\begin{pmatrix}
0 & 0 & 0 \\
0 & 0 & 0 \\
0 & 1 & 0 
\end{pmatrix},
\begin{pmatrix}
0 & 0 & 0 \\
0 & 0 & 0 \\
0 & 1 & 0 
\end{pmatrix},
\begin{pmatrix}
0 & 1 & 0 \\
0 & 1 & 0 \\
0 & 0 & 1 
\end{pmatrix}
\right)
\]
}}
\[
\begin{split}
p(x)&=(x_1, 
  \i3  (   x_2 + x_1  x_3 + x_3  x_1 - x_3  x_1  x_3 ) 
    \i3, x_3  \i3)\\
p^{-1} (y)&=
( {y_1, \ \iy3   ( y_2   + y_1 y_3   +  y_3 y_1    -  y_3  y_1 y_3 ) \iy3,
  \ y_3    \iy3})
  \end{split}
  \]

\addtocontents{toc}{\SkipTocEntry}\subsection*{Algebra 12}

Relations on a basis $R_1, R_2, R_3$:
\[
R_1 R_1=R_2, \quad R_2 R_2=0, \quad R_3 R_3=R_3,
\]
\[
R_1 R_2=0, \quad R_2 R_1=0, \quad R_1 R_3=R_1, \quad
R_3 R_1= R_1, \quad R_2 R_3=R_2, \quad R_3 R_2=R_2.
\]

Structure matrices
{{\footnotesize
\[  \Xi=
\left(
\begin{pmatrix}
0 & 1 & 0 \\
0 & 0 & 0 \\
1 & 0 & 0 
\end{pmatrix},
\begin{pmatrix}
0 & 0 & 0 \\
0 & 0 & 0 \\
0 & 1 & 0 
\end{pmatrix},
\begin{pmatrix}
1 & 0 & 0 \\
0 & 1 & 0 \\
0 & 0 & 1 
\end{pmatrix}
\right)
\]
}}
{\footnotesize{
\[
\begin{split}
p(x)&=(\i3  x_1  \i3, \ \i3  (\; x_1  \i3  x_1 + x_2 \; )  \i3, \
 x_3  \i3)\\
p^{-1} (y)&=
(  \iy3   y_1   \iy3, 
 -\iy3   (\; y_1   \iy3   y_1 - y_2 \; )   \iy3,  
 y_3    \iy3 )
 \end{split}
 \]}}
An appealing way to present the map $p$ in this case is to write it as
$$
\pi(x)=( x_1 , \  x_2 +  x_1  \i3  x_1   , \
 x_3 - x_3^2  )
$$
with each entry multiplied on the  right and the left by $\i3$.

\vfill

\end{document}